\documentclass{article}
\usepackage[utf8]{inputenc}
\usepackage{amsmath, amssymb,amsthm}
\usepackage{mathrsfs}
\usepackage{enumitem}
\usepackage{fullpage}
\usepackage{xcolor}
\usepackage{dsfont}
\usepackage{graphicx}
\usepackage[font=small,skip=0pt]{caption}
\theoremstyle{plain}
\usepackage[hidelinks]{hyperref}
\hypersetup{
    colorlinks,
    linkcolor={blue!80!black},
    citecolor={green!50!black},
}
\colorlet{linkequation}{blue}

\newcommand{\E}{\mathbb{E}}

\newcommand{\N}{\mathbb{N}}
\renewcommand{\P}{\mathbb{P}}

\newcommand{\R}{\mathbb{R}}

\newcommand{\Z}{\mathbb{Z}}

\renewcommand{\AA}{\mathcal{A}}

\newcommand{\CC}{\mathcal{C}}
\newcommand{\DD}{\mathcal{D}}

\newcommand{\FF}{\mathcal{F}}
\newcommand{\GG}{\mathcal{G}}

\newcommand{\MM}{\mathcal{M}}
\newcommand{\NN}{\mathcal{N}}

\newcommand{\RR}{\mathcal{R}}

\newcommand{\TT}{\mathcal{T}}

\newcommand{\given}{\,|\,}
\newcommand{\bgiven}{\,\Big|\,}
\newcommand{\bbgiven}{\,\,\bigg|\,\,}

\newcommand{\eps}{\varepsilon}
\newcommand{\one}{\mathds{1}}
\newcommand{\la}{\lambda}
\newcommand{\normal}{\mathsf{N}}

\newcommand{\sig}{\sigma}
\newcommand{\pto}{\stackrel{p}{\longrightarrow}}
\newcommand{\GW}{\textsf{GW}}

\newcommand{\ga}{\gamma}
\newcommand{\sumz}{\sum_{i=1}^{q}Z_i}
\newcommand{\sumi}{\sum_{i=1}^{q}}
\newcommand{\sumj}{\sum_{j=1}^{\gamma}}
\newcommand{\norm}[1]{\left\|{#1}\right\|}
\newcommand{\beq}{\begin{equation}}
\newcommand{\eeq}{\end{equation}}
\newcommand{\beqn}{\begin{equation*}}
\newcommand{\eeqn}{\end{equation*}}

\DeclareMathOperator{\Var}{Var}
\DeclareMathOperator{\Cov}{Cov}

\DeclareMathOperator{\tr}{tr}

\DeclareMathOperator*{\argmax}{arg\,max}

\newtheorem{thm}{Theorem}[section]
\newtheorem{prop}[thm]{Proposition}
\newtheorem{cor}[thm]{Corollary}
\newtheorem{lemma}[thm]{Lemma}
\newtheorem{claim}[thm]{Claim}
\newtheorem{conj}[thm]{Conjecture}

\theoremstyle{definition}
\newtheorem{defn}[thm]{Definition}

\newtheorem{remark}[thm]{Remark}
\newtheorem{assumption}[thm]{Assumption}
\newcommand{\rev}[1]{#1}
\title{Exact Phase Transitions for Stochastic Block Models and Reconstruction on Trees}
\author{
Elchanan Mossel\thanks{Department of Mathematics and IDSS, Massachusetts Institute of Technology. Email: \textup{\tt elmos@mit.edu}} \and 
Allan Sly\thanks{Department of Mathematics, Princeton University. Email: \textup{\tt asly@math.princeton.edu}}\and 
Youngtak Sohn \thanks{Department of Mathematics, Massachusetts Institute of Technology. Email: \textup{\tt youngtak@mit.edu}}}
\begin{document}

\maketitle
\begin{abstract}
In this paper we continue to rigorously establish the predictions in 
ground breaking work in statistical physics by Decelle, Krzakala, Moore, Zdeborov\'a (2011) regarding the block model,
in particular in the case of $q=3$ and $q=4$ communities. 

We prove that 
 for $q=3$ and $q=4$ there is no computational-statistical gap if the average degree is above some constant by showing it is information theoretically impossible to detect below the Kesten-Stigum bound.

The proof is based on showing that for the broadcast process on Galton-Watson trees, reconstruction is impossible for $q=3$ and $q=4$ if the average degree is sufficiently large. This improves on the result of Sly (2009), who proved similar results for regular trees for $q=3$. Our analysis of the critical case $q=4$ provides a detailed picture 
showing that the tightness of the Kesten-Stigum bound in the antiferromagnetic case depends on the average degree of the tree.

Our results prove conjectures of 
Decelle, Krzakala, Moore, Zdeborov\'a (2011), Moore (2017), 
Abbe and Sandon (2018) and Ricci-Tersenghi, Semerjian, and Zdeborov{\'a} (2019). 
Our proofs are based on a new general coupling of the tree and graph processes and on a refined analysis of the broadcast process on the tree. 
\end{abstract}
\maketitle

\section{Introduction}
In this paper we study the block model and establish a number of exact phase transitions that were conjectured in the case of more than two communities. 

The block model is a random graph model generalizing the famous Erdos-Renyi random graph~\cite{ErdosRenyi:60} and is a special case of 
inhomegnuous random graphs, see e.g.~\cite{BoJaRi:07}. 
It has been studied extensively in statistics as a model of communities
\cite{HoLaLe:83}, see e.g.~\cite{SnijdersNowicki:97,BickelChen:09,RoChYu:11} 
 and in computer science as a model to study the average case behavior of clustering algorithms, see e.g.~\cite{DyerFrieze:89,JerrumSorkin:98,CondonKarp:01,McSherry:01,CojaOghlan:10}.  The papers above mostly concentrate on cases where the average degree is at least of order $\log n$, where $n$ is the number of nodes in the graph.  We focus on the sparse regime where the average degree is constant.
The model is defined as follows:  

\begin{defn} \label{def:SBM}
(\textit{The Stochastic Block Model}) For $n\geq 1$, $q\geq 2$, and $p_{\textnormal{in}},p_{\textnormal{out}}\in (0,1)$, let $\GG(n,q,p_{\textnormal{in}},p_{\textnormal{out}})$ denote the model of random graphs with $q$ colors in which each vertex $u$ \rev{belonging to the vertex set $\{1,2,\ldots, n\}$} is assigned a label $\sig_u\in \{1,2,...,q\}$ uniformly at random, and then each possible edge $(u,v)$ is included with probability $p_{\textnormal{in}}$ if $\sig_u=\sig_v$ and with probability $p_{\textnormal{out}}$ if $\sig_u \neq \sig_v$.
\end{defn}

We are interested in detection (see Defintion \ref{def:detection} for a formal definition) of the \rev{Stochastic} Block Model (SBM) in the sparse regime $G\sim \GG(n,q,\frac{a}{n},\frac{b}{n})$. Denote
\begin{equation}\label{eq:ab.dlambda}
    d=\frac{a+(q-1)b}{q},\quad \la= \frac{a-b}{a+(q-1)b}.
\end{equation}
\rev{Here, $d$ and $\la$ respectively encode the average degree and the second eigenvalue of the transition matrix of the broadcast process associated with the block model (see Definition~\ref{def:broadcast}).} The block model $G\sim \GG(n,q,\frac{a}{n},\frac{b}{n})$ became a major object of research 
 due to a landmark paper in statistical physics~\cite{DKMZ:11}  
where the authors predicted:

\begin{conj}[] \label{conj:block}
For the block model \rev{with $a,b,q$ fixed and $n \to \infty$:} 
\begin{enumerate}
\item[I.] 
For all $q$, it is possible to detect communities better than random if $d \la^2 > 1$.
\item[II.] 
For $q \leq 4$ it is information theoretically impossible to predict better than random if $\la \geq 0$ and $d \la^2 < 1$.
\item[III.] For $q \geq 5$, it is information theoretically possible to predict better than random for some $\la$ with $d\la^2 < 1$, but not in a computationally efficient way. 
\end{enumerate} 
\end{conj} 
The threshold $d\la^2=1$ is called the \textit{Kesten-Stigum} (KS) bound. 
It is a natural threshold for the trade off between noise and duplication. It was first discovered in the context of the broadcast process on trees in work of Kesten and Stigum~\cite{KestenStigum:66}, and has later played important \rev{role} in understanding the broadcast process ~\cite{BlRuZa:95,Ioffe:96a,Ioffe:96b,EvKePeSc:00,MosselPeres:03,JansonMossel:04} as we discuss below, in Phylogenetic reconstruction~\cite{Mossel:04,DaMoRo:11,MoRoSl:11,RochSly:17} and the block model which we discuss next. We note that, in the case of $q=2$, it is known to be tight for reconstruction in all of these models~\cite{BlRuZa:95,Ioffe:96a,Ioffe:96b,EvKePeSc:00,Mossel:04,Massoulie:14,MoNeSl:18,BoLeMa:15}

For the block model, it was also conjectured in~\cite{DKMZ:11} that the optimal algorithm for detection is Belief-Propagation. 
While we are still unable to analyze Belief-Propagation for detecting block models, it is possible to analyze related algorithms based on non-backtracking or self-avoiding walks (as was proposed by 
\cite{Krzakala_etal:13}). 
Thus in a series of works~\cite{Massoulie:14,MoNeSl:18,BoLeMa:15,AbbeSandon:15,AbbeSandon:18} part I. of Conjecture \ref{conj:block} was established. \rev{In particular, it was established in \cite{AbbeSandon:18} that whenever $d\la^2>1$, there exists a polynomial time algorithm to detect communities.}

Parts II and III predict the regime where one can expect a computational-statistical gap, \rev{meaning a gap between what is achievable information theoretically and what is achievable with known computationally efficient algorithms.} 
Some special cases of part III of the conjecture are given rigorous support 
in~\cite{BMNN:16, AbbeSandon:18}. In particular~\cite{BMNN:16} find regimes where detection is information-theoretically possible below the Kesten-Stigum threshold when $\la < 0$ and $q \geq 5$ and when $\la > 0$ and $q \geq 11$. The paper~\cite{AbbeSandon:18} find such regimes for small $d$ when $q=4$, $\la < 0$.

After the work~\cite{DKMZ:11}, a number of refined conjectures were made in~\cite{Moore17,AbbeSandon:18,RiSeZd:19} regarding the regimes where computational-statistical gap exist\rev{s} for $q=3,4$. In particular~\cite{RiSeZd:19} conjectures that:
\begin{conj} \label{conj:RSZ} For the block model:

\begin{itemize}
\item[I.] {${\bf q=3}$}. When $q=3$ it is information theoretically impossible to detect communities whenever $d \la^2 < 1$.
\item[II.] ${\bf q=4, \la > 0}$.  When $q=4, \la > 0$, 
it is information theoretically impossible to detect communities whenever $d\la^2<1$.
\item[III.] ${\bf q=4, \la < 0}$. When $q=4, \la < 0 $, there is a critical degree $d^{\ast}$ such that if $d > d^{\ast}$ then it is impossible to detect when $d \la ^2 < 1$ but if $d < d^{\ast}$ it is possible to detect by a (not necessarily efficient) algorithm in a regime where $d \la^2 < 1$.
\end{itemize}
\end{conj}

The critical number of communities $q=4$ is the most interesting one. We note that~\cite{AbbeSandon:18} made a different conjecture in the case $q=4, \la < 0$, as they conjectured that for all degrees $d > 1$, it is possible to detect in a regime where $d \la^2 < 1$.

Prior work has \rev{only} shown that there is no statistical-computational gap when $q=2$~\cite{MoNeSl:15}, 
where it is shown that for $q=2$ it is information theoretically impossible to detect better than random if $d \la^2 \leq 1$ based on coupling of the graph and tree processes. 
It is natural to ask if the coupling with the tree process can be carried out in further generality. Also, can we prove nonreconstruction in the tree process in the case when $q=3$ and in some cases where $q=4$? These are natural questions that were asked before, see e.g.~\cite{Moore17}.

\subsection{Main Results for the block model}

We use the following definition of detection. 
\begin{defn} \label{def:detection}
We say that {\em detection is possible} 
for $\GG(n,q,\frac{a}{n},\frac{b}{n})$ if
there exists an $\epsilon > 0$ and an estimator $\hat{\sigma}(G)$ 
that takes as an input the graph $G \sim \GG(n,q,\frac{a}{n},\frac{b}{n})$ and returns 
$\hat{\sigma} : [n] \to [q]$, such that
\[
\limsup_{n \to \infty}
\frac{1}{n}
\E\Big[
\max_{\Gamma \in S_q}
\big|\{\rev{u}: \hat{\sigma}\rev{_u} = \Gamma(\sigma\rev{_u})\}\big|
\Big] \geq \frac{1}{q} + \epsilon. 
\]
\end{defn}

We remark that 
\begin{itemize}
\item In the definition above there are no computational requirements for $\hat{\sigma}$.
Thus, our results showing that detection is impossible prove an information theoretical impossibility result.
\item The symmetry between the different communities means that we can only detect communities up to a \rev{permutation} of $q$. This is the reason for taking the maximum of $\Gamma \in S_q$.
\item The definition above is a little weaker than some earlier definitions, e.g.~\cite{AbbeSandon:18}, where the algorithm is required to produce $1/q + \epsilon$ overlap with probability $1-o(1)$ over the graph. Since our main results prove that detection is impossible, they also imply the impossibility of detection under these stronger definitions.
\item \rev{Detection is also called ``weak recovery'' or ``partial recovery'' in the literature. For different terminologies of various statistical inference tasks, we refer to Section 2.3 of the survey~\cite{abbe18survey}.} 
\end{itemize}

In our first main result, \rev{for large enough average degree $d$}, we establish part II of Conjecture~\ref{conj:block} and Conjecture~\ref{conj:RSZ}.

\begin{thm}\label{thm:SBM}
For $q\in\{3,4\}$, there exists a universal constant $d_{0}<\infty$ such that if $d\geq d_0$ and $d\la^2\leq 1$, i.e. $(a-b)^2\leq q(a+(q-1)b)$, then it is information-theoretically impossible to detect for $G\sim \GG(n,q,\frac{a}{n},\frac{b}{n})$. 
That is, for any estimator $\hat{\sigma}$ that takes as input the graph $G$ (is $G$\rev{-}measurable) it holds that:
\begin{equation} \label{eq:SBM.defn}
\limsup_{n \to \infty}
\frac{1}{n}
\E\Big[
\max_{\Gamma \in S_q}
\big|\{\rev{u}: \hat{\sigma}\rev{_u} = \Gamma(\sigma\rev{_u})\big|
\Big] =\frac{1}{q}. 
\end{equation}
\end{thm}
We remark that the condition that $d$ is large enough is necessary for the case $q=4$, because~\cite{AbbeSandon:18} shows that there exist regimes where $d$ is low enough and the KS bound can be beaten for $q=4$. Also note that the $\hat{\sigma}$ that returns $1$ for all vertices has 
\[
\frac{1}{n}
\E\Big[
\big|\{\rev{u}: \hat{\sigma}\rev{_u} = \Gamma(\sigma\rev{_u})\}\big|
\Big] =\frac{1}{q}, 
\]
for all $n$, so the interest in~(\ref{eq:SBM.defn}) is in proving a $1/q$ upper bound. Theorem \ref{thm:SBM} can also be formulated in terms of the \rev{two point correlation} of spins.
\begin{cor}\label{cor:two.point}
Under the same hypothesis as Theorem~\ref{eq:SBM.defn} for any fixed vertices $u,v$ and a color $i\in\{1,2,..,q\}$,
\begin{equation}\label{eq:SBM.2.point.defn}
    \P(\sig_u=i \given G,\sig_v=1) \pto \frac{1}{q}.
\end{equation}
\end{cor}
We indeed show in Section \ref{s:coupling} that \eqref{eq:SBM.defn} and \eqref{eq:SBM.2.point.defn} are equivalent.

It was shown in \cite{AbbeSandon:18} that whenever $d\la^2>1$, there exists a polynomial time algorithm to detect communities, thus Theorem \ref{thm:SBM} and the result of \cite{AbbeSandon:18} altogether imply that there exists \textit{no} statistical-computational gap in the block model with $3$ or $4$ communities whose degree is a large enough constant.

The seminal work \cite{CKPZ:18} used the so-called ``quiet planting'' method to show that in the \textit{anti-ferromagnetic regime} $\la <0$, i.e. $a<b$, there exists the so-called \textit{condensation} threshold $d_{\textnormal{cond}}(q,\la)$ from physics \cite{MezardMontanari:09},  such that detection is impossible if $d<d_{\textnormal{cond}}(q,\la)$ while there exists an (not necessarily efficient) algorithm that detects the communities if $d>d_{\textnormal{cond}}(q,\la)$. The work \cite{CKPZ:18} deals with a variety of models beyond the block model, and verifies much of the conjecture from physics \cite{KMRSZ:07}. However, it is far from clear \rev{whether to check if} $d_{\textnormal{cond}}(q,\la)$, which is expressed in terms of an infinite dimensional stochastic optimization, matches the Kesten-Stigum bound \rev{or not}.

More importantly, the method of \cite{CKPZ:18} falls short in the \textit{ferromagnetic regime} $\la>0$. Indeed, the recent work \cite{DominiguezMourrat22} conjectured that even when $q=2$, the formula for the asymptotic mutual information \rev{in the case $\la>0$} has a completely different form than in the case $\la<0$. Our results show that the information theoretic threshold equals the Kesten-Stigum bound regardless of $\la>0$ or $\la<0$ when $q\in \{3,4\}$ and $d$ large enough.

Furthermore, \cite{CEJKK:18} combined a delicate second moment method and a small graph conditioning method to show that when $d<d_{\textnormal{cond}}(q,\la)$ and $\la<0$, the block model and the corresponding \textit{null} model \rev{are} contiguous (see Theorem 2.6 of \cite{CEJKK:18}). Since \rev{by combining} Theorem~\ref{thm:SBM} and the result of \cite{AbbeSandon:18}, \rev{we have} that $d_{\textnormal{cond}}(q,\la)=\frac{1}{\la^2}$ for $q\in \{3,4\}$ and $-\frac{1}{\rev{\sqrt{d_0}}}<\la<0$, we have the following corollary.

\begin{cor}\label{cor:contig}

Suppose that $q\in \{3,4\}$, and $d\geq d_0$, where $d_0$ is a universal constant. If $a\leq b$ and $(a-b)^2< q(a+(q-1)b)$ hold, then the block model $\mathcal{G}(n,q,\frac{a}{n},\frac{b}{n})$ is contiguous to the Erdos-Renyi graph $\mathcal{G}(n,\frac{d}{n})$. Thus, if $a\leq b$ and $(a-b)^2< q(a+(q-1)b)$, there is no statistical test that distinguishes with probability $\rev{1-o(1)}$ if a sample is sampled from $\mathcal{G}(n,q,\frac{a}{n},\frac{b}{n})$ or from the Erdos-Renyi graph $\mathcal{G}(n,\frac{d}{n})$. Moreover, there is no consistent estimator for $a$ and $b$ in this regime.
\end{cor}

\rev{Note that unlike most of the results in this paper the last statement regarding consistent estimator in Corollary~\ref{cor:contig} assumes that we do not know the values of $a,b$. Other results of the paper regarding the impossibility of detection are stated assuming the values $a$ and $b$ are known.} Also, we note that the constraint $a\leq b$, or equivalently $\lambda \leq 0$, in Corollary \ref{cor:contig} is necessary for the results of \cite{CKPZ:18, CEJKK:18} to hold, and it would be interesting to show the analogous result for $a>b$.

\subsection{Main results for the broadcast model}

The major step of the proof of Theorem~\ref{thm:SBM} is to show that the reconstruction problem for a corresponding broadcast model is not solvable as illustrated in the impossibility result for $q=2$ in \cite{MoNeSl:15}. 
We now define these notions.

\begin{defn}\label{def:broadcast}
(\textit{The broadcast model on a tree}) Given a (possibly infinite) tree $T$ with root $\rho$, the (symmetric) broadcast model with state space $\CC= \{1,2,...,q\}$ and second eigenvalue $\la\in [-\frac{1}{q-1},1)$ is a probability distribution on $\CC^{V(T)}$ defined as follows. The spin at the root $\sig_{\rho}\in \CC$ is chosen uniformly at random and is then propagated along the edges of $T$ according to the following rule: if the vertex $v$ is the parent of $u$ in the tree, then 
\begin{equation}\label{eq:transition:matrix}
    \P(\sig_u=j\given \sig_{\rev{v}}=i)= 
    \begin{cases}
    \la + \frac{1-\la}{q} &\text{if $i=j$}\\
    \frac{1-\la}{q} &\text{otherwise}.
    \end{cases}
\end{equation}
\end{defn}
Given $T$, let $\sig(n)$ denote the spins at distance $n$ from the root $\rho$ and let $\sig^{i}(n)$ denote $\sig(n)$ conditioned on $\sig_{\rho}=i$.
\begin{defn}(\textit{Reconstruction on trees})
We say that the reconstruction problem is solvable for a broadcast model on a tree $T$ with second eigenvalue $\la$ if there exists $i,j \in \CC$ such that
\begin{equation*}
    \limsup_{n\to\infty} d_{\textnormal{TV}}(\sig^{i}(n),\sig^{j}(n))>0.
\end{equation*}
Otherwise, we say that the model has nonreconstruction at $\la$ for $T$.
\end{defn}

The broadcast process on the trees was first studied in the context of multi-type branching processes, where Kesten and Stigum~\cite{KestenStigum:66} identified the bound $d \la^2 = 1$ as a critical threshold for the behavior of the fluctuations of this model.
Some related models were studied in trying to understand spin-glasses~\cite{CCST:86,CCST:90,CCCST:90}.
The interest in this model grew significantly \rev{when} it was established that for $q=2$ the Kesten-Stigum (KS) bound is tight for reconstruction. 
This was established first for regular trees
~\cite{BlRuZa:95,Ioffe:96a} and then on general trees~\cite{EvKePeSc:00,Ioffe:96b} where the degree is replaced by the branching number~\cite{Lyons:89}. 

The fact that the KS bound is not tight for reconstruction on trees \rev{(i.e. reconstruction is solvable for a broadcast model at some $\lambda$ and $d$-ary tree satisfying $d\la^2<1$ and $q\geq 3$ states)} was first established 
in~\cite{Mossel:01} with much tighter results obtained by~\cite{Sly:09,Sly:11} as we discuss below. Interestingly, the KS bound is tight for all $q$ for count-reconstruction~\cite{MosselPeres:03} and for robust reconstruction~\cite{JansonMossel:04}.
There are also interesting conjectures relating the computational complexity of estimating the root from the leaves to the KS bound~\cite{MoMoSa:20,KoehlerMossel:22}.
To read more about the broadcast model see~\cite{Mossel:04,Mossel:23}.

\rev{Recall that the Galton-Watson tree is a \textit{random} tree, where each node in the tree has i.i.d. number of children drawn from some offspring distribution $\mu_d$ supported on $\Z_{\geq 0}$ (see e.g. \cite[Section 4.3.4]{Durrett_2019}).} Denote \rev{by} $\GW(\mu_d)$ the law of the Galton-Watson tree with offspring distribution $\mu_d$ with average degree $d$, \rev{where $d$ is not necessarily an integer}. We then analyze the reconstruction problem for the broadcast model on a \textit{random} tree $T\sim \GW(\mu_d)$. We consider the following mild assumptions on the offspring distribution $\{\mu_d\}_{d>1}$.
\begin{assumption}(\textit{Uniform tail for the offspring distribution})\label{assumption:unif:tail}
The distributions $\{\mu_d\}_{d\rev{>}1}$ satisfy $\E_{\mu_{d}}[X_d]=d$, where $\E_{\mu_{d}}$ denotes the expectation with respect to $X_d\sim \mu_d$, and for every $\theta\geq 0$, we have
\begin{equation}\label{eq:def:K}
    K(\theta):= \sup_{d>1}\E_{\mu_d}\left[\exp\left(\frac{\theta X_d}{d}\right)\right]<\infty.
\end{equation}
\end{assumption}
\begin{assumption}(\textit{Tightness of the offspring distribution})
\label{assumption:tight}
For any $\eps>0$, there exist positive constants $C_1(\eps), C_2(\eps)>0$, and $d_0(\eps)>1$, which only depend on $\eps$, such that for every $d\geq d_0(\eps)$,
\begin{equation*}
    \P_{\mu_d}\left(\frac{X_d}{d}\in \Big[C_1(\eps),C_2(\eps)\Big]\right)\geq 1-\eps,
\end{equation*}
where $\P_{\mu_d}$ denotes the probability with respect to $X_d\sim \mu_d$.
\end{assumption}

\begin{thm}\label{thm:KS:tight}
Suppose $\{\mu_d\}_{d>1}$ satisfy Assumptions \ref{assumption:unif:tail} and \ref{assumption:tight}.
Then, for $q=3,4$, there exists a constant $d_0<\infty$, which only depends on the function $K(\cdot)$ in \eqref{eq:def:K}, such that for $d\geq d_0$, the Kesten-Stigum bound is sharp. Moreover, there is nonreconstruction at the Kesten-Stigum bound. That is, there is nonreconstruction almost surely over $T\sim \GW(\mu_d)$, whenever $d\la^2\leq 1$ and $d\geq d_0$.
\end{thm}
We will establish Theorem \ref{thm:SBM} from Theorem \ref{thm:four:antiferro} with $\mu_d$ being the Poisson distribution with mean $d$. In this case, \rev{we have $K(\theta)=\sup_{d>1}\exp\big(d(e^{\frac{\theta}{d}}-1))\big)<\infty$} and $\frac{X_d}{d}\pto 1$ as $d\to\infty$ for $X_d\sim \textsf{Poi}(d)$.

It was shown in \cite{Sly:11} that for $q=3$, the Kesten-Stigum bound is tight for reconstruction on the $d$-ary tree $T_d$ with large enough $d$. Thus, \rev{the} $q=3$ case in Theorem \ref{thm:KS:tight} is a generalization of \cite{Sly:11} to the case of Galton-Watson trees. The critical case $q=4$ however, is much more challenging to analyze as one might expect by Conjecture \ref{conj:RSZ}. Indeed, $q=4$ was open up to date even for the $d$-ary tree, and there were conflicting conjectures in \cite{AbbeSandon:18} and \cite{RiSeZd:19} whether the KS bound is tight for $q=4$. This is partly because for $q=4$, the tightness of the KS bound depends on the magnitude of $d$ and the sign of $\la$, as shown in the next theorem, and we give significant new insights into the reasons why this is the case. Define

\begin{equation*}
\begin{split}
    \la^{+}(q,d)&:=\sup\left\{\la>0: \textnormal{there is nonreconstruction at $\la$ for $T\sim \GW(\mu_d)$ a.s.}\right\}\,,\\
    \la^{-}(q,d)&:=\inf\left\{\la<0: \textnormal{there is nonreconstruction at $\la$ for $T\sim \GW(\mu_d)$ a.s.}\right\}\,.
\end{split}
\end{equation*}
\begin{thm}\label{thm:four:antiferro}
Given $\{\mu_d\}_{d>1}$ which satisfies Assumption \ref{assumption:unif:tail}, let $d^\star\equiv d^\star(\{\mu_d\}_{d>1})$ be
\begin{equation}\label{eq:def:d:star}
    d^\star :=\inf\left\{d>1: \frac{15+36(\sqrt{d}-1)}{7d(d-1)}\leq \frac{\E_{\mu_d}\left[X(X-1)(X-2)\right]}{\left(\E_{\mu_d}\left[X(X-1)\right]\right)^2}\right\}.
\end{equation}
Here, we take the convention that the infimum over an empty set is $\infty$. If $d^\star>1$ holds, then for all $d\in (1,d^\star)$, the Kesten-Stigum bound is not sharp in the antiferromagnetic regime for $q=4$, i.e. $\la^{-}(4,d)>-d^{-1/2}$.
Moreover, the Kesten-Stigum bound is not sharp both in the ferromagnetic and the antiferromagnetic regime for $q\geq 5$, i.e. $\la^{+}(q,d)<d^{-1/2}$ and $\la^{-}(q,d)>-d^{-1/2}$.
\end{thm}
Theorem \ref{thm:four:antiferro}, which establishes reconstruction below the KS bound with positive probability, can be viewed as evidence for part III of Conjecture~\ref{conj:block} and Conjecture~\ref{conj:RSZ}. Moreover, Theorem~\ref{thm:four:antiferro} confirms one side of part III of Conjecture~\ref{conj:RSZ} for the broadcast model, and the specific value $d=d^{\ast}$ in \eqref{eq:def:d:star} matches the prediction from \cite{RiSeZd:19}.

\rev{We note that the definition of $d^\star$ in \eqref{eq:def:d:star} stems from certain recursion~\eqref{eq:exact:third:order} of a quantity $x_n$ called \textit{average magnetization}~~\eqref{eq:def:avg:magnet}.} The notable cases for $d^\star$ are for the Poission offspring distribution $d^\star_{\textsf{Poi}}\equiv d^\star\left((\textsf{Poi}(d))_{d>1}\right)$ and for the $d-$regular tree case $d^\star_{\textnormal{Reg}}\equiv d^\star\left((\delta_d)_{d>1}\right)$: from \eqref{eq:def:d:star}, it is straightforward to calculate
\begin{equation*}
    d^\star_{\textsf{Poi}}=\left(\frac{18+\sqrt{226}}{7}\right)^2\approx 22.2694\quad\quad  d^\star_{\textnormal{Reg}}=\left(\frac{18+\sqrt{275}}{7}\right)^2\approx 24.408.
\end{equation*}
We note that the value $d^\star_{\textsf{Poi}}\approx 22.2694$ matches the predictions of \cite{RiSeZd:19}.

{\em Open Problems.}
We note that while \rev{this paper} provide\rev{s} some sharp result\rev{s} showing that there is no computational-statistical gap for block models when $q=3$ and in some cases where $q=4$, some interesting challenges remain, including: 
\begin{itemize}
\item Can one prove Theorem~\ref{thm:SBM} for $q=3$ and small $d$? Can this be done for $q=4$ and $\la \geq 0$?
\item Can we improve on~\cite{BMNN:16,AbbeSandon:18} and chart more accurately the regimes below the KS bound, where it is information theoretically possible to detect? 
\item \rev{The tightness of KS bound for reconstruction on general trees was established by~\cite{EvKePeSc:00} for $q=2$ case, where the degree is replaced by the branching number~\cite{Lyons:89}. It is thus natural to try to extend Theorem~\ref{thm:KS:tight} from Galton-Watson trees to general trees.} 
\item It was conjectured in~\cite{RiSeZd:19} that in the setting of Theorem~\ref{thm:four:antiferro} for $q=4$ and the antiferromagnetic case, when $d > d^{\ast}$ the KS bound is sharp. It is natural to try and prove this. By the coupling argument in our paper this would imply that detection is impossible for the block model for the same parameters. 

\end{itemize}

\section{Proof overview}
In this section, we introduce the main ideas to prove our main results. We first explain how Theorem \ref{thm:KS:tight} implies Theorem \ref{thm:SBM} by coupling the block model to multiple independent broadcast models. Then, we introduce the high level ideas to establish Theorem \ref{thm:KS:tight} and Theorem \ref{thm:four:antiferro}.
\subsection{Coupling SBM to multiple broadcast processes}

We prove Theorem~\ref{thm:SBM} by establishing that nonreconstruction on the tree implies nonreconstruction on the corresponding stochastic block model.
\begin{thm}\label{t:nonrecon.to.sbm}
If there is nonreconstruction on the Poisson(d) Galton-Watson tree for \rev{the broadcast process} with second eigenvalue $\lambda$ then detection is information theoretically impossible on the stochastic block model $\GG(n,q,\frac{a}{n},\frac{b}{n})$ for $(a,b)$ corresponding to $(d,\la)$ by equation~\eqref{eq:ab.dlambda}.  
\end{thm}

Results of~\cite{MoNeSl:15} proved the analogue of Theorem~\ref{t:nonrecon.to.sbm} in the case of the symmetric 2-state SBM \rev{by establishing an approximate Markov random field property. Namely, given the observed graph, the conditional distribution of a neighborhood of $u$, conditional on all the spins outside the neighborhood of $u$, is approximately the same if we only condition on the spins on the boundary of the neighborhood of $u$.} This is more subtle than it looks at first, as to establish it, one has to account for the information contained in the absence of edges from the neighborhood to the remainder of the graph.  In Section~\ref{s:coupling} we present a new approach to prove Theorem~\ref{t:nonrecon.to.sbm} for the general stochastic block model $\mathcal{G}(n,q,\frac{a}{n},\frac{b}{n})$. 

The connection between these two notions of nonreconstruction follows from the fact that the local weak limit of the SBM and its spins is the  Poisson(d) Galton-Watson tree with the symmetric channel broadcast model.  This means that we can couple the local neighborhood of a vertex in the SBM with Galton-Watson tree including the spins.  In fact we show that we can simultaneously couple  the neighborhoods of $\rev{\lfloor n^{1/5}\rfloor}$ randomly chosen vertices $u_i$ at a time.  

We show that asymptotically the Maximum a Posteriori (MAP) estimator cannot correctly assign vertices with rate greater than $\frac1{q}$ asymptotically.  We do so by partitioning the $\rev{\lfloor n^{1/5} \rfloor}$ local neighborhoods according to their tree topology and the spin configuration on their boundary configuration.  By the symmetry of the model, the spin configurations within the neighborhoods are exchangeable within each partition class.  By our coupling and the fact that we are in the 
nonreconstruction regime on the tree, in most partition classes the root has each given color in about $\frac1{q}$ fraction of $u_i$.  It follows that even with the additional information of the boundary configurations, the MAP estimator correctly assigns only $\frac1{q}+o(1)$ spins asymptotically.  Since this holds for a randomly chosen set, we also cannot do better than a $\frac1{q}+o(1)$ fraction on the whole graph which implies~\eqref{eq:SBM.defn}. We believe that this new approach applies to a wide range of settings.

\begin{proof}[Proof of Theorem~\ref{thm:SBM}]
The result follows from Theorem~\ref{t:nonrecon.to.sbm} because if $(a-b)^2\leq q(a+(q-1)b)$ then $d\lambda^2\leq 1$ and so for large enough $d$ and $q\in\{3,4\}$ by Theorem~\ref{thm:KS:tight} we have nonreconstruction on the Poisson(d) Galton-Watson tree.
\end{proof}

\subsection{Analysis of the broadcast model on Galton-Watson trees}
\label{subsec:overview:broadcast}
In this subsection, we give an outline of the proof of Theorem \ref{thm:KS:tight} and Theorem \ref{thm:four:antiferro}. In particular, we describe the main ideas in dealing with random trees $T\sim \GW(\mu_d)$ and the critical case $q=4$.

We start by introducing the necessary notations. Given a (random) tree $T$ with root $\rho$, let $\sig$ be the spin configurations sampled from the broadcast model and let $\sig(n)$ be the spins at distance $n$ from the root $\rho$. The posterior probability of the root given the spins at depth $n$ is 
\begin{equation}\label{eq:def:posterior}
    f_{n,T}(i,\tau)=\P(\sig_{\rho}=i\given \sig(n)=\tau).
\end{equation}
Recall that $\sig^i$ is the random configuration $\sig$ conditioned on $\sig_{\rho}=i$. Then, define
\begin{equation*}
   X^{+}=X^{+}(n,T)=f_{n,T}(1,\sig^1(n))\quad\textnormal{and}\quad X^{-}=X^{-}(n,T)=f_{n,T}(2,\sig^1(n)).
\end{equation*}
That is, $X^{+}$ is the probability of being correct when we guess $\sig_{\rho}$ according to the posterior distribution. Note that there are two sources of randomness in $X^{+}$, namely $T$ and $\sig$. To this end, we \rev{write} $\E_{\sig}$ \rev{for} the expectation with respect to the randomness of $\sig$, i.e. the conditional expectation given a tree $T$. Define
\begin{equation*}
x_n(T)=\E_{\sig}X^{+}(n)-\frac{1}{q}.
\end{equation*}
Borrowing the terminology from statistical physics, we often call the quantity $x_n(T)$ the \textit{magnetization} of \rev{$\sig(n)$}. From the inferential perspective, $x_n(T)$ encodes the amount of information that the spins at depth $n$ has on the root $\sig_{\rho}$. As noted in \cite{CCST:86,BCMR:06, Sly:11}, $x_n(T)$ is closely related to nonreconstruction.
 In particular, Corollary 2.4 of \cite{Sly:11} have proved for the $d$-ary tree $T_d$, we have $x_n(T_d)\geq 0$ and $\lim_{n\to\infty} x_n(T_d)= 0$ is equivalent to nonreconstruction.

We prove that almost sure nonreconstruction for $T\sim \GW(\mu_d)$ is equivalent to vanishing \rev{\textit{average magnetization}}, defined by
\begin{equation}\label{eq:def:avg:magnet}
    x_n=\E[x_n(T)] = \E X^{+}(n)-\frac{1}{q},
\end{equation}
A key observation is that the amount of information $x_n(T)$ deteriorates as the depth $n$ increases for \textit{arbitrary} tree $T$. The proof of Lemma \ref{lem:equiv:xn:nonreconstruct} is given in Section \ref{subsec:basic:id}.
\begin{lemma}\label{lem:equiv:xn:nonreconstruct}
$(x_n)_{n\geq 1}$ is a monotonically decreasing sequence with $x_n\geq 0$. Moreover, we have nonreconstruction almost surely over $T\sim \GW(\mu_d)$ if and only if $\lim_{n\to\infty} x_n=0$.
\end{lemma}
Having Lemma \ref{lem:equiv:xn:nonreconstruct} in hand, we analyze $x_n$ by taking advantage of the recursive nature of the Galton-Watson trees. Denote \rev{by} $\gamma$ the degree of the root $\rho$, so we have $\gamma\sim \mu_d$. Let $u_1,...,u_{\gamma}$ be the children of $\rho$ in $T$, and denote $T_{j}$ by the subtree of descendents of $u_j$ including $u_j$. Further, \rev{let $L(n)$ be the vertices at depth $n$} and denote \rev{by} $\sig_j(n)$ the spins in $L(n)\cap T_j$. For a configuration $\tau$ on $L(n+1)\cap T_{j}$, the posterior function at $u_j$ is
\begin{equation*}
    f_{n,T_j}(i,\tau)=\P(\sig_{u_j}=i\given \sig_j(n+1)=\tau).
\end{equation*}
Observe that $f_{n+1,T}$ can be written in terms of $(f_{n,T_j})_{j\leq \gamma}$ by  Bayes rule: given a configuration $\tau$ on $L(n+1)$, let $\tau_j$ be its restriction to $L(n+1)\cap T_j$. Recalling the transition matrix from \eqref{eq:transition:matrix}, Bayes rule gives for $i_0\in \{1,...,q\}$,
\begin{equation}
\label{eq:bayes}
\begin{split}
    f_{n+1,T}(i_0, \tau)
    &= \frac{\prod_{j=1}^{\gamma}\left(\frac{1+(q-1)\la}{q}f_{n,T_j}(i_0,\tau_j)+\sum_{\ell \neq i_0} \frac{1-\la}{q}f_{n,T_j}(\ell,\tau_j)\right)}{\sum_{i=1}^{q}\prod_{j=1}^{\gamma}\left(\frac{1+(q-1)\la}{q}f_{n,T_j}(i,\tau_j)+\sum_{\ell \neq i} \frac{1-\la}{q}f_{n,T_j}(\ell,\tau_j)\right)}\\
    &=\frac{\prod_{j=1}^{\gamma}\left(1+\la q (f_{n,T_j}(i_0,\tau_j)-1/q)\right)}{\sum_{i=1}^{q}\prod_{j=1}^{\gamma}\left(1+\la q (f_{n,T_j}(i,\tau_j)-1/q)\right)},
\end{split}
\end{equation}
where the last equality is due to $\sum_{\ell=1}^{q}f_{n,T_j}(\ell,\tau_j)=1$. We take the convention that the product over an empty set is $1$, e.g. $\prod_{j=1}^{0}f(j)=1$, so the equality above still holds when $\gamma=0$ (the posterior probability is always equal to $1/q$ when $T$ is a single node). Having \eqref{eq:bayes} in hand, define the posterior probability that $\rev{\sigma_{u_j}}=i$ given $\sig_j^1(n+1)$ by
\begin{equation*}
    Y_{ij}=Y_{ij}(n)=f_{n,T_j}(i,\sig^1_j(n+1)).
\end{equation*}
We take the convention that $Y_{i1}$ is defined to be $\frac{1}{q}$ when $\gamma=0$ so that $Y_{i1}$ is always well-defined. Then, plugging in $\tau=\sig^{1}(n+1)$ and $i_0=1$ to \eqref{eq:bayes} shows
\begin{equation}\label{eq:key:recursion}
    X^{+}(n+1)=\frac{Z_1}{\sum_{i=1}^{q}Z_i},\quad\textnormal{where}\quad Z_i\equiv \prod_{j=1}^{\gamma}\left(1+\la q\left(Y_{ij}-\frac{1}{q}\right)\right).
\end{equation}
Of course, the analogue $X^{-}(n+1)=\frac{Z_2}{\sumz}$ holds as well.

Equation \eqref{eq:key:recursion} plays a key role in relating $X^{+}(n+1)$ from $X^{+}(n)$. Observe that unlike in the case of \textit{non-random} $d$-ary trees, the subtree $T_j$ is not the same as the whole tree $T$ in general, so the distribution of $Y_{ij}$ cannot be directly written in terms of $X^{+}(n)$ and $X^{-}(n)$. 

However, $T_j$ has the same \textit{distribution} as $T\sim \GW(\mu_d)$, so we can leverage this fact to relate the distribution of $Y_{ij}$ and $(X^{+}(n),X^{-}(n))$ \textit{after} taking expectation with respect to $\sigma$. In particular, we have the following Proposition which follows from the definition of Galton-Watson trees, Markov random field property, and the symmetries of the model.  Recall that $Y_{ij}$ are defined so that the root is conditioned on $\sig_{\rho}=1$.
\begin{prop}\label{prop:Y:dist}
Denote $Y_j=(Y_{ij})_{i\leq q}\in \R^q$ for $j \leq \gamma$. Then, we have the following properties.
\begin{enumerate}[label=(\alph*)]
    \item Given $T$, $Y_j$'s for $j=1,2,..,\gamma$ are conditionally independent.
    \item Let $f:\R^q \to \R$ be a continuous function. Then, conditional on $\gamma$, $\Big\{\E_{\sig}f(Y_j)\Big\}_{j\leq \gamma}$ are independent and identically distributed. Moreover, $\E_{\sig}f(Y_1)$ is independent of $\gamma$.
    \item Let $g:\R^2 \to \R$ be a continuous function and $i\neq \ell\in \{1,...,q\}$ be $2$ distinct colors. Then, the distribution of $\E_{\sig}\Big[g(Y_{i1},Y_{\ell 1})\given \sig^{1}_{u_1}=i\Big]$ is the same as the distribution of $\E_{\sig}g\Big(X^{+}(n),X^{-}(n)\Big)$.
    \item Given $T$, $\sig^1_{u_1}$, and $Y_{\sig^1_{u_1} 1}$, the distribution of $(Y_{ij})_{i\neq \sig^1_{u_1}}$ is conditionally exchangeable.
\end{enumerate}
\end{prop}
 Using \eqref{eq:key:recursion} and Proposition \ref{prop:Y:dist}, we estimate $x_{n+1}$ in terms of $x_n$ when $x_n$ is small. Indeed, \cite{Sly:11} used a similar strategy to show that when $T=T_d$ and $x_n(T_d)$ small enough,
\begin{equation}\label{eq:expand:sec:order:sly}
    x_{n+1}(T_d)=d\la^2 x_n(T_d)+(1+o(1))\frac{d(d-1)}{2}\frac{q(q-4)}{q-1}\la^4 x_n(T_d)^2.
\end{equation}
The approach of \cite{Sly:11} can be summarized as follows. Define
    \begin{equation}\label{eq:def:q}
        z_n(T):=\E_{\sig}\left(X_{+}(n)-\frac{1}{q}\right)^2,\quad\quad z_n:=\E[z_n(T)]=\E\left(X_{+}(n)-\frac{1}{q}\right)^2.
    \end{equation}
\begin{enumerate}
    \item First, Taylor expand the $(\sum_{i=1}^{q}Z_i)^{-1}$ in \eqref{eq:key:recursion} around $q^{-1}$ up to second order term. In particular, using $\frac{1}{1+x}=1-x+\frac{x^2}{1+x}$ for $x=\frac{\sumz -q}{q}$, it follows that
    \begin{equation}\label{eq:expand:Z:sly}
    x_{n+1}(T_d)=\E_{\sig}\frac{Z_1\rev{-1}}{q}-\E_{\sig}\frac{Z_1(\sum_{i=1}^{q}Z_i-q)}{q^2}+\E_{\sig}\frac{Z_1}{\sum_{i=1}^{q}Z_i}\frac{\left((\sum_{i=1}^{q}Z_i)-q\right)^2}{q^2}.
    \end{equation}
    \item Show that $\E_{\sig}\prod_{i=1}^{q}Z_i^{k_i}$ for $0\leq k_i\leq 2$ can be approximated with error $O(x_n^3)$ by a polynomial of $\E_{\sig}\prod_{i=1}^{q}(Y_{i1}-\frac{1}{q})^{k_i}$, which is a linear combination of $x_n(T_d)$ and $z_n(T_d)$. Using this, approximate $\E_{\sig}Z_1, \E_{\sig}Z_1(\sumi Z_i-q)$, and $\E_{\sig}(\sumi Z_i-q)^2$ up to the second order w.r.t. $x_n(T_d)$ and $z_n(T_d)$. In particular, $\E_{\sig}(\sumi Z_i-q)^2=O(x_n^2)$ holds.
    \item By using concentration inequalities, show that when $x_n$ is small enough, $\frac{Z_1}{\sumi Z_i}$ is concentrated around $\frac{1}{q}$, and that $z_n = (1+o(1))\frac{x_n}{q}$. Plugging in these estimates to the right hand side of equation \eqref{eq:expand:Z:sly} implies the estimate \eqref{eq:expand:sec:order:sly}.
    \item In order to show that the KS bound is tight for $q=3$, we need that $x_n$ is small enough to apply equation \eqref{eq:expand:sec:order:sly}. This can be established for large enough $d$ by using a central limit theorem to show that $x_{n+1}\approx g_3(d\la^2 x_n)$ for some continuous function $g_3(\cdot)$ which can be shown to satisfy $g_3(s)<s$ for all $0<s\leq\frac{2}{3}$.
\end{enumerate}

Note that however, equation \eqref{eq:expand:sec:order:sly} alone barely gives any information for $q=4$. In order to gain information when $q=4$, we need lower order terms in the equation~\eqref{eq:expand:sec:order:sly}. To this end, we expand further in \eqref{eq:expand:Z:sly}.
\begin{equation}\label{eq:def:xi}
\begin{split}
    x_{n+1}&= \Xi_1+\Xi_2\quad\textnormal{where}\\
    \Xi_1&:= \E\frac{Z_1 -1}{q}-\E\frac{Z_1(\sumz -q)}{q^2}+\E\frac{(\sumz -q)^2}{q^3};\\
    \Xi_2&:= \E\frac{(Z_1-1)(\sumz -q)^2}{q^3}-\E\frac{Z_1(\sumz -q)^3}{q^4}+\E\frac{Z_1}{\sumz}\left(\frac{\sumz-q}{q}\right)^4.
\end{split}
\end{equation}
Then, we approximate $\Xi_1$ and $\Xi_2$ separately up to third order w.r.t. $x_n$. In such a task, the previous techniques of of \cite{Sly:11} \textit{cannot} be applied directly. We summarize the major difficulties as follows.
\begin{enumerate}
    \item Even for approximation of $\Xi_1$, we need a better estimate of $z_n$ than just $z_n=(1+o(1))\frac{x_n}{q}$. That is, we need to control how this $o(1)$ term behaves. Indeed, we will see that this lower order term in $z_n$ affects $x_{n+1}$ by a non-negligible amount $\Theta(x_n^3)$.
    \item To approximate $\E Z_1(\sumz-q)^{\ell}$ for $\ell=3,4,$ using the previous strategy, it is crucial to calculate $\E_{\sig}\prod_{i=1}^{q}(Y_{i1}-\frac{1}{q})^{k_i}$ for $\sumi k_i=3,4$ and approximate them \textit{solely} with $x_n$. This amounts to introducing new variables other than $x_n$ and $z_n$ (e.g., $\E(X^{+}(n)-\frac{1}{q})^{\ell}$ for $\ell=3,4$), and leads to a complicated chain of recursions between different variables. It is far from clear how to analyze such chain, let alone expressing the variables only with $x_n$.
\end{enumerate}
Our proof consists of three major parts.

In Section \ref{sec:apriori}, we show that the previous approximation scheme of \cite{Sly:11} can be generalized to Galton-Watson trees with offspring distribution satisfying Assumption \ref{assumption:unif:tail}. Then, we use such scheme to estimate of $z_n$ up to second order of $x_n$. A key insight is that such estimate can be \textit{iterated} to give \textit{apriori} estimate $\left|z_n-\frac{x_n}{q}\right| \leq O_q(x_n^2)$, which induces an accurate estimate of $\Xi_1$.

In Section \ref{sec:full}, we give a new approximation scheme for estimating $\E f(Z)$, where $f:\R^q \to \R$ is a polynomial (see Proposition \ref{prop:crucial:product}) by taking advantage of \textit{apriori} estimates obtained in Section \ref{sec:apriori}. The key feature of the new scheme is that it can take advantage of \textit{polynomial factorization} of the polynomial $f$. For example, it can be shown with Proposition \ref{prop:crucial:product} that $\E f(Z)=O_q(x_n^3)$ holds if $f=\prod_{\ell=1}^{L}f_{\ell}$ for polynomials $(f_{\ell})_{\ell\leq L}$ which does not have a constant term and $L\geq 5$ (see Lemma \ref{lem:useful} below). Then, we use such scheme to give a finer estimate of $z_n$, and also the exact third order estimates of $\Xi_1$ and $\Xi_2$ to show that when $x_n$ is small enough,
\begin{multline}\label{eq:exact:third:order}
    x_{n+1}=d\la^2 x_n+\E\binom{\gamma}{2}\la^4 \frac{q(q-4)}{q-1}x_n^2+\E\binom{\gamma}{3}\la^6 \frac{q^2(q^2-18q+42)}{(q-1)^2}x_n^3\\
    +\left(\E\binom{\gamma}{2}\right)^2 \la^6\left(\frac{q^2(q+1)}{(q-1)^2}\frac{1}{d(d-1)}-\frac{6q^2(q-2)}{(q-1)^2}\frac{1}{d(d\la-1)}+o(1)\right)x_n^3.
\end{multline}
When $q=4$, the coefficient in front of $x_n^3$ is always negative when $\la=d^{-1/2}$ while at $\la=-d^{-1/2}$, it is positive if and only if $d<d^\star$. This establishes non-tightness of the KS bound in the antiferromagnetic regime for $q=4$ and $d<d^\star$.

In Section \ref{sec:clt}, we establish Theorem \ref{thm:KS:tight}. A key ingredient in proving Theorem \ref{thm:KS:tight} is to use a normal approximation \textit{conditional} on the tree $T$. In particular, it follows from Assumption \ref{assumption:tight} that the degree $\gamma$ is large with high probability when $d$ is large, so we use a normal approximation to show that $x_{n+1}$ can be approximated by $x_n$ by a \textit{deterministic} function $g_q:[0,\frac{q-1}{q}] \to [0,\frac{q-1}{q}]$ regardless of whether $x_n$ is small or large. It can be shown that for $q=3,4,$, $g_q(s)<s$ holds for all $0<s\leq \frac{q-1}{q}$, which guarantees that for large enough $d$, $x_n$ eventually becomes small enough to apply the estimate \eqref{eq:exact:third:order}. Unfortunately, for small $d$, we are not able to show that $x_n$ becomes small enough to apply the estimate \eqref{eq:exact:third:order}.

\textbf{Notations:} Throughout, we fix a class of offspring distributions $\{\mu_d\}_{d>1}$ that satisfy Assumption \ref{assumption:unif:tail}. We denote $\gamma\sim \mu_d$ by the degree at the root of the Galton Watson tree $T\sim \GW(\mu_d)$. We denote $\E_{\sig}$ by the expectation with respect to the randomness of spins $\sig$, i.e. the conditional expectation given a tree $T$. We reserve the notation $C_q$ for a positive constant that depends only on $q$ and $K(\cdot)$ defined in \eqref{eq:def:K}, i.e. that does not depend on $n,d,\la$. For quantities $f=f(n,d,\la, q)$ and $g=g(n,d,\la, q)$, we write $f=O_q(g)$ if there exists a constant $C_q$ which only depends on $q$ and $K(\cdot)$ such that $|f|\leq C_q g$. 

\section*{Acknowledgements}
We thank Guilhem Semerjian for sending us the reference \cite{RiSeZd:19}. Youngtak Sohn thanks Amir Dembo for helpful discussions. Elchanan Mossel and Youngtak Sohn are supported by Simons-NSF collaboration on deep learning NSF DMS-2031883 and Vannevar Bush Faculty Fellowship award ONR-N00014-20-1-2826. Elchanan Mossel is also supported by ARO MURI W911NF1910217 and a Simons Investigator Award in Mathematics (622132). Allan Sly is supported by NSF grants DMS-1855527, DMS-1749103, a Simons Investigator grant, and a MacArthur Fellowship.
\section{Apriori estimates of the magnetization}
\label{sec:apriori}
In this section, we estimate $\Xi_1$ defined in \eqref{eq:def:xi}. Various estimates for quantities such as $z_n-\frac{x_n}{q}$ obtained in this section will be used for Section \ref{sec:full}.

To begin with, we show that the nonreconstruction is a monotone property w.r.t. $\la$.
\begin{lemma}\label{lem:la:nonreconstruct}
For arbitrary tree $T$, consider a broadcast model in Definition \ref{def:broadcast}. If there is nonreconstruction at $\la$, then for any $c\in [0,1]$, there is also nonreconstruction at $c\la$.
\end{lemma}
\begin{proof}
In the broadcast model the transition matrix from parent to child is given by $K_{\la} = \la I_q + (1-\la) J_q$, where $I_q$ is the $q \times q$ identity matrix and $J_q$ is the rank one transition matrix where all entries are $1/q$.
Therefore we have:
\[
K_{c \la} = K_c K_{\la},
\]
and note that since $c > 0$, the markov chain with matrix $K_c$, copies with probability $c$ and randomizes uniformly, otherwise. 
This implies that if $\sigma^{i}(n)$ is a sample
of the broadcast process with $\la$ and $\eta^{i}(n)$ is the process with $c \la$, then $\eta^{i}$ can be obtain from $\sigma^{i}$ by the following random process: 
\begin{enumerate}[label=(\arabic*)]
    \item
Perform edge percolation on $T$ with probability $c$.
\item For vertices $v$ at level $n$ that are at the same component of the root, let $\eta^{(i)}(v) = \sigma^{(i)}(v)$.
\item For vertices in other components, run the broadcast process on each component with a uniformly random chosen root and parameter $\la$.
\end{enumerate}
Note that the process to generate $\eta$ from $\sigma$ does not depend on the root and therefore 
\[
d_{TV}(\eta^{i}(n), \eta^{j}(n)) \leq 
d_{TV}(\sigma^{i}(n), \sigma^{j}(n)), 
\]
which implies the statement of the lemma. 
\end{proof}
Having Lemma \ref{lem:la:nonreconstruct} in hand, we focus on the regime where $\la$ is sufficiently close to $\pm \frac{1}{\sqrt{d}}$. In particular, we assume throughout the paper that there is a universal constants $c_0>0$ such that
\begin{equation}\label{eq:dlambda:lower}
\rev{d^5|\la|^{9}>1+c_0,\quad\textnormal{and}\quad  1-c_0\leq d\la^2\leq 1.}
\end{equation}
\rev{In particular, this implies that $d^{\ell}|\la|^{2\ell-1}>1+c_0$ holds for all $1\leq \ell\leq 5$.}
\subsection{Basic identities and the proof of Lemma \ref{lem:equiv:xn:nonreconstruct}}
\label{subsec:basic:id}
Recalling the posterior probability from \eqref{eq:def:posterior},  denote
\begin{equation*}
    X_i(n)=f_{n,T}(i,\sig(n)).
\end{equation*}
In addition to $x_n(T)$ and $z_n(T)$, define
\begin{equation}\label{eq:def:v:w}
\begin{split}
        v_n(T)&:=\E_{\sig}\left(X^{+}(n)-\frac{1}{q}\right)^3,\qquad\qquad\qquad\qquad\qquad\qquad\quad\quad~~v_n:=\E[v_n(T)]
        \\
        w_n(T)&:= \E_{\sig}\left(X^{+}(n)-\frac{1}{q}\right)\left(X^{-}(n)-\frac{1}{q}\right)\left(X^{+}(n)-X^{-}(n)\right),~~ w_n = \E[w_n(T)]
\end{split}
\end{equation}
\begin{lemma}\label{lem:basic}
For any tree $T$, the following relations holds:
\begin{equation}\label{eq:lem:basic:1}
    x_n(T)= \E_{\sig}\left(X^{+}(n)-\frac{1}{q}\right)^{2}+(q-1)\E_{\sig}\left(X^{-}(n)-\frac{1}{q}\right)^{2}\geq z_n(T),
\end{equation}
and
\begin{equation}\label{eq:lem:basic:2}
   z_n(T)-\frac{x_n(T)}{q} = \E_{\sig}\left(X^{+}(n)-\frac{1}{q}\right)^{3}+(q-1)\E_{\sig}\left(X^{-}(n)-\frac{1}{q}\right)^{3},
\end{equation}
and
\begin{equation}\label{eq:lem:basic:3}
   v_n(T)-\frac{1}{q}\left(z_n(T)-\frac{x_n(T)}{q}\right) = \E_{\sig}\left(X^{+}(n)-\frac{1}{q}\right)^{4}+(q-1)\E_{\sig}\left(X^{-}(n)-\frac{1}{q}\right)^{4}.
\end{equation}
\rev{Moreover, the inequalities $|v_n(T)|\leq x_n(T)$ and $|w_n(T)|\leq x_n(T)$ hold.}
\end{lemma}
\begin{proof}
From the definition of conditional probabilities, we have for $k\geq 1$
\begin{equation*}
    \E X^{+}(n)^{k}=\sum_{\tau} f_{n,T}(1,\tau)^{k}\cdot \P(\sig(n)=\tau \given \sig_{\rho}=1)=q\sum_{\tau}f_{n,T}(1,\tau)^{k+1}\cdot \P(\sig(n)=\tau),
\end{equation*}
where the final equality holds because $\P(\sig_{\rho}=1)=1/q$. The final expression equals $q\cdot \E X_1(n)^{k+1}$ and the distribution of $(X_i(n))_{i\leq q}$ is exchangeable by symmetry, so
\begin{equation}\label{eq:lem:basic:4}
    \E X^{+}(n)^{k}=q\cdot \E X_1(n)^{k+1}=\sum_{i=1}^{q}\E X_i(n)^{k+1}.
\end{equation}
Using \eqref{eq:lem:basic:4} for $k=1$ and $\sum_{i=1}^{q}X_i(n)=1$, we have that $x_n(T)=\sum_{i=1}^{n}(X_i(n)-1/q)^2$. Thus, since conditional on $\sig_{\rho}$, $X_{\sig_{\rho}}(n)$ is distributed as $X^{+}(n)$ while $X_{i}(n)$ for $i\neq \sig_{\rho}$ is distributed as $X^{-}(n)$, \eqref{eq:lem:basic:1} holds. Using \eqref{eq:lem:basic:4} for $k=2,3$, and repeating the same argument gives \eqref{eq:lem:basic:2} and \eqref{eq:lem:basic:3}. \rev{Finally, we have $|v_n(T)|\leq \E_{\sigma}\left(X^{+}(n)-1/q\right)^2\leq x_n(T)$, and 
\[
|w_n(T)|\leq 2\E_{\sigma}\left|X^{+}(n)-\frac{1}{q}\right|\cdot\left|X^{-}(n)-\frac{1}{q}\right|\leq \E_{\sigma}\left(X^{+}(n)-\frac{1}{q}\right)^2+\E_{\sigma}\left(X^{+}(n)-\frac{1}{q}\right)^2,
\]
which is upper bounded by $x_n(T)$. Therefore, $|v_n(T)|\vee |w_n(T)|\leq x_n(T)$ holds.
}
\end{proof}

\begin{proof}[Proof of Lemma \ref{lem:equiv:xn:nonreconstruct}]
We first claim that for each fixed tree $T$, $x_n(T)$ converging to $0$ as $n\to \infty$ is equivalent to nonreconstruction. Denote \rev{by} $p_n(T)=\E \max_{1\leq i \leq q} X_i(n)$ the probability of optimal reconstruction. Then, Proposition 14 of \cite{Mossel:01} shows that $p_n(T)\to \frac{1}{q}$ is equivalent to nonreconstruction. It is shown in Lemma 2.3 of \cite{Sly:11} that 
\rev{
\begin{equation}\label{eq:MLE:x}
    x_n(T) \leq p_n(T)-\frac{1}{q}\leq x_n(T)^{1/2},
\end{equation}
}
 so $x_n(T)\to 0$ is equivalent to nonreconstruction.

Next, we claim that $\{x_n(T)\}_{n\geq 1}$ is monotonically decreasing with $n$ for each tree $T$. Note that $\{X_i(n)\}_{n\geq 1}$ is a backwards martingale. \rev{Indeed, if we let $\tau(n)$ be the spins outside of depth $n$ neighborhood of the root $\rho$, then $X_i(n)=\P(\sig_{\rho}=i\given \tau(n))$ holds since $\sig_{\rho}$ is conditionally independent of all spins at depths larger depth $n$ given $\sig(n)$.} Thus, $\E\left(X_i(n)-1/q\right)^2$ is decreasing over $n$ for each $i$, so this fact and Lemma \ref{lem:basic} imply that $\{x_n(T)\}_{n\geq 1}$ is monotonically decreasing.

Note that $x_n(T)\geq 0$ holds due to Lemma \ref{lem:basic}. Thus, $\{x_n(T)\}_{n\geq 1}$ is monotonically decreasing and bounded from below, so denote its limit by $x_{\infty}(T)=\lim_{n\to\infty} x_n(T)$. Then, by monotone convergence theorem, $x_n \to \E[x_{\infty}(T)]$, so $x_n\to 0$ is equivalent to $x_{\infty}(T)=0$ a.s., which in turn is equivalent to nonreconstruction a.s. by our first claim.
\end{proof}
\begin{remark}\label{rmk:x:positive}
Note that $x_n(T)\geq 0$ can be improved to $x_n(T)>0$ when $\la \neq 0$. Indeed, it was shown in Lemma 2.9 of \cite{Sly:11} that $x_{n+1}(T)\geq \la^4 x_n(T_j)^2$ holds by noting that the estimator that outputs $i$ with probability $\P(\sig_{\rho}=i\given \sig_1(n+1))$ correctly guesses the root with probability $\frac{1}{q}+\la^2 x_n(T_j)$, so  $\frac{1}{q}+\la^2 x_n(T_j)\leq p_n(T)$ must hold. \rev{Combining with \eqref{eq:MLE:x}, it follows that $x_{n+1}(T)\geq \la^4 x_n(T_j)^2$.} Since we assumed $\la\neq 0$ (cf. \eqref{eq:dlambda:lower}) and $x_0(T)=\frac{q-1}{q}$ holds for any $T$, we have that $x_n(T)>0$.
\end{remark}
\begin{lemma}\label{lem:Y:moment}
The first moment of $(Y_{i1})_{i\leq q}$ is given by the following: for $2\leq i\leq q$,
\begin{equation}\label{eq:Y:firstmo}
    \E\left(Y_{11}-\frac{1}{q}\right)=\la x_n,\qquad\qquad \E\left(Y_{i1}-\frac{1}{q}\right)=-\frac{\la x_n}{q-1}\quad\textnormal{for}\quad i\neq 1.
\end{equation}
The second moment of $(Y_{i1})_{i\leq q}$ is given by  the following: for all distinct pairs $(i_1,i_2)\in \{2,...,q\}^2$,
\begin{equation}\label{eq:Y:secondmo}
\begin{split}
    &\E\left(Y_{11}-\frac{1}{q}\right)^2
    =\frac{x_n}{q}+\la\left(z_n-\frac{x_n}{q}\right),\qquad\qquad \E\left(Y_{i_1 1}-\frac{1}{q}\right)^2=\frac{x_n}{q}-\frac{\la}{q-1}\left(z_n-\frac{x_n}{q}\right),\\
    &\E\left(Y_{11}-\frac{1}{q}\right)\left(Y_{i_1 1}-\frac{1}{q}\right)
    =-\frac{x_n}{q(q-1)}-\frac{\la}{q-1}\left(z_n-\frac{x_n}{q}\right),\\
    &\E\left(Y_{i_1 1}-\frac{1}{q}\right)\left(Y_{i_2 1}-\frac{1}{q} \right)=-\frac{x_n}{q(q-1)}+\frac{2\la}{(q-1)(q-2)}\left(z_n-\frac{x_n}{q}\right).
\end{split}
\end{equation}
Moreover, we have for all distinct pairs $(i_1,i_2)\in \{2,...,q\}^2$, 
\begin{equation}\label{eq:Y:thirdmo}
\begin{split}
    &\E\left(Y_{11}-\frac{1}{q}\right)^3
    =\frac{1-\la}{q}\left(z_n-\frac{x_n}{q}\right)+\la v_n,\\
    &\E\left(Y_{i_1 1}-\frac{1}{q}\right)^3=\frac{q-1+\la}{q(q-1)}\left(z_n-\frac{x_n}{q} \right)-\frac{\la}{q-1}v_n,\\
    &\E\left(Y_{11}-\frac{1}{q}\right)^2\left(Y_{i_1 1}-\frac{1}{q}\right)
    =-\frac{1-\la}{q(q-1)}\left(z_n-\frac{x_n}{q}\right)-\frac{\la}{q-1}v_n,\\
    &\E\left(Y_{11}-\frac{1}{q}\right)\left(Y_{i_1 1}-\frac{1}{q}\right)^2
    =-\frac{1-\la}{q(q-1)}\left(z_n-\frac{x_n}{q}\right)-\frac{\la}{q-1}v_n-\la w_n,\\
    &\E\left(Y_{11}-\frac{1}{q}\right)\left(Y_{i_1 1}-\frac{1}{q}\right)\left(Y_{i_2 1}-\frac{1}{q}\right)\\
    &=\frac{2(1-\la)}{q(q-1)(q-2)}\left(z_n-\frac{x_n}{q}\right)+\frac{2\la}{(q-1)(q-2)}v_n+\frac{\la}{q-2}w_n.
\end{split}
\end{equation}
\rev{
Finally, for distinct indices $(i_1,i_2,i_3)\in \{2,...,q\}^3$,
\begin{equation}\label{eq:Y:thirdmo:2}
\begin{split}
    &\E\left(Y_{i_1 1}-\frac{1}{q}\right)^2\left(Y_{i_2 1}-\frac{1}{q}\right)=-\frac{q-2+\la}{q(q-1)(q-2)}\left(z_n-\frac{x_n}{q}\right)+\frac{2\la}{(q-1)(q-2)}v_n+\frac{\la}{q-2}w_n,\\
    &\E\left(Y_{i_1 1}-\frac{1}{q}\right)\left(Y_{i_2 1}-\frac{1}{q}\right)\left(Y_{i_3 1}-\frac{1}{q}\right)
    \\
    &=-\frac{2(q-3+2\la)}{q(q-1)(q-2)(q-3)}\left(z_n-\frac{x_n}{q}\right)-\frac{6\la}{(q-1)(q-2)(q-3)}v_n+\frac{3\la}{(q-2)(q-3)}w_n.
\end{split}
\end{equation}
}
\end{lemma}
\begin{proof}
The proof of \eqref{eq:Y:firstmo} and \eqref{eq:Y:secondmo} is the same or simpler than the proof of \eqref{eq:Y:thirdmo}, and their analogues for the $d$-ary tree case was proven in Lemma 2.5 of \cite{Sly:11}, so we only provide the proof of \eqref{eq:Y:thirdmo} and \eqref{eq:Y:thirdmo:2}.

We start with the first identity of \eqref{eq:Y:thirdmo}: since $\P(\sig^1_{u_1}=i)=(1-\la)/q$ for $i\neq 1$, we have
\begin{equation*}
\begin{split}
    &\E_{\sig}\left(Y_{11}-\frac{1}{q}\right)^3\\
    &= \frac{1+(q-1)\la}{q}\E_{\sig}\left[\left(Y_{11}-\frac{1}{q}\right)^3\bbgiven\sig^1_{u_1}=1\right]+\frac{1-\la}{q}\sum_{i=2}^{q}\E_{\sig}\left[\left(Y_{11}-\frac{1}{q}\right)^3\bbgiven\sig^1_{u_1}=i\right].
\end{split}
\end{equation*}
By Proposition \ref{prop:Y:dist}-$(c)$, $\E_{\sig}[(Y_{11}-1/q)^3\given \sigma^1_{u_1}=i]$ is distributed as $\E_{\sig}(X^{+}(n)-\frac{1}{q})^3$ if $i=1$, and $\E_{\sig}(X^{-}(n)-\frac{1}{q})^3$ otherwise. It follows that
\begin{equation}\label{eq:Y:X:plus:minus}
\begin{split}
    \E \left(Y_{11}-\frac{1}{q}\right)^3
    &=\frac{1+(q-1)\la}{q}\E\left(X^{+}(n)-\frac{1}{q}\right)^3+\frac{(q-1)(1-\la)}{q}\E\left(X^{-}(n)-\frac{1}{q}\right)^3\\
    &=\frac{1-\la}{q}\left(z_n-\frac{x_n}{q}\right)+\la v_n,
\end{split}
\end{equation}
where in the second equality, we used equation Lemma \ref{lem:basic}. Similarly, we have
\begin{equation*}
\begin{split}
\E\left(Y_{i_1 1}-\frac{1}{q}\right)^3
&=\frac{1-\la}{q}\E\left(X^{+}(n)-\frac{1}{q}\right)^3+\frac{q-1+\la}{q}\E\left(X^{-}(n)-\frac{1}{q}\right)^3\\
&=\frac{q-1+\la}{q(q-1)}\left(z_n-\frac{x_n}{q} \right)-\frac{\la}{q-1}v_n.
\end{split}
\end{equation*}
By conditional exchangeability in Proposition \ref{prop:Y:dist}-$(d)$ and $\sum_{i=1}^{q}(Y_{i1}-1/q)=0$, we have
\begin{equation}\label{eq:cond:exchange}
\begin{split}
    \E\left[\E_{\sig}\left(Y_{11}-\frac{1}{q}\right)^2\left(Y_{i_1 1}-\frac{1}{q}\right)\right]
    &=-\frac{1}{q-1}\E\left[\E_{\sig}\left(Y_{11}-\frac{1}{q}\right)^3\right]\\
    &=-\frac{1-\la}{q(q-1)}\left(z_n-\frac{x_n}{q}\right)-\frac{\la}{q-1}v_n.
\end{split}
\end{equation}
For the fourth identity, note that
\begin{equation*}
\begin{split}
    &\E_{\sig}\left[\left(Y_{11}-\frac{1}{q}\right)^2\left(Y_{i_1 1}-\frac{1}{q}\right)-\left(Y_{11}-\frac{1}{q}\right)\left(Y_{i_1 1}-\frac{1}{q}\right)^2\right]\\
    &=\sum_{i=1}^{q}\E_{\sig}\left[\left(Y_{11}-\frac{1}{q}\right)\left(Y_{i_1 1}-\frac{1}{q}\right)\left(Y_{11}-Y_{i_1 1}\right)\bbgiven\sig^1_{u_1}=i\right]\cdot\P(\sig^1_{u_1}=i)
\end{split}
\end{equation*}
By Proposition \ref{prop:Y:dist}-$(c), (d)$, it follows that
\begin{equation*}
\begin{split}
   &\E_{\sig}\left[\left(Y_{11}-\frac{1}{q}\right)\left(Y_{i_1 1}-\frac{1}{q}\right)\left(Y_{11}-Y_{i_1 1}\right)\bbgiven\sig^1_{u_1}=i\right]\\
   &\stackrel{d}{=}
    \begin{cases}
    \E_{\sig}\left(X^{+}(n)-\frac{1}{q}\right)\left(X^{-}(n)-\frac{1}{q}\right)\left(X^{+}(n)-X^{-}(n)\right) &\text{if $i=1$};\\
    -\E_{\sig}\left(X^{+}(n)-\frac{1}{q}\right)\left(X^{-}(n)-\frac{1}{q}\right)\left(X^{+}(n)-X^{-}(n)\right)&\text{if $i=i_1$};\\
    0 &\text{otherwise}.
    \end{cases}
\end{split}
\end{equation*}
Since $\P(\sig^1_{u_1}=1)=(1+(q-1)\la)/q$ and $\P(\sig^1_{u_1}=i_1)=(1-\la)/q$, taking expectations in the two displays above show
\begin{equation*}
\begin{split}
    \E\left(Y_{11}-\frac{1}{q}\right)\left(Y_{i_1 1}-\frac{1}{q}\right)^2
    &=\E\left(Y_{11}-\frac{1}{q}\right)^2\left(Y_{i_1 1}-\frac{1}{q}\right)-\la w_n\\
    &=-\frac{1-\la}{q(q-1)}\left(z_n-\frac{x_n}{q}\right)-\frac{\la}{q-1}v_n-\la w_n.
\end{split}
\end{equation*}
The fifth identity \rev{of \eqref{eq:Y:thirdmo} and \eqref{eq:Y:thirdmo:2}} follow by a similar argument as in \eqref{eq:cond:exchange}, namely by using conditional exchangeability in Proposition \ref{prop:Y:dist} and the previously obtained identities, so we do not repeat the argument here.
\end{proof}

Observe that we expressed the moments of $Y_{i1}$'s in Lemma \ref{lem:Y:moment} by a linear combination of $x_n,u_n,v_n,w_n$, where we define $u_n$ by
\begin{equation}\label{eq:def:u}
u_n:=z_n-\frac{x_n}{q}=\frac{q-1}{q}\left(\E\left(X^{+}(n)-\frac{1}{q}\right)^2-\E\left(X^{-}(n)-\frac{1}{q}\right)^2\right).
\end{equation}
This linear transformation from $z_n$ to $u_n$ will be convenient especially since we prove that $u_n= O_q(x_n^2)$ in the next subsection.

We close this subsection with a useful identity which follows from Proposition \ref{prop:Y:dist}.
\begin{lemma}\label{lem:Y:indep}
Suppose we have non-negative integers $s_{ij}\geq 0$ for $i,j\in \Z_{+}$ and real numbers $a_i, b_i \in \R$ for $1\leq i \leq q$. Then,
\begin{equation*}
    \E\prod_{i=1}^{q}\prod_{j=1}^{\ga} \Big(a_i Y_{ij}+b_i\Big)^{s_{ij}}=\E\prod_{j=1}^{\ga}\left(\E\prod_{i=1}^{q}\Big(a_i Y_{i1}+b_i\Big)^{s_{ij}}\right)
\end{equation*}
where the outer expectation in the right hand side is the expectation with respect to $\gamma$.
\end{lemma}
\begin{proof}
By tower property and Proposition \ref{prop:Y:dist}-$(a)$, we have that
\begin{equation*}
\begin{split}
    \E\prod_{j=1}^{\ga} \prod_{i=1}^{q}\Big(a_i Y_{ij}+b_i\Big)^{s_{ij}}
    &=\E\left[\prod_{j=1}^{\ga} \left(\E_{\sig} \prod_{i=1}^{q}\Big(a_i Y_{ij}+b_i\Big)^{s_{ij}}\right)\right]\\
    &=\E\left[\E\left[\prod_{j=1}^{\ga} \left(\E_{\sig} \prod_{i=1}^{q}\Big(a_i Y_{ij}+b_i\Big)^{s_{ij}}\right)\bbgiven \gamma \right]\right].
\end{split}
\end{equation*}
Note that by Proposition \ref{prop:Y:dist}-$(b)$, conditional on $\gamma$, $\left(\E_{\sig} \prod_{i=1}^{q}\left(a_i Y_{ij}+b_i\right)^{s_{ij}}\right)_{j\leq \ga}$ are i.i.d., so the conditional expectation w.r.t. $\ga$ of the right hand side above equals
\begin{equation*}
\begin{split}
  \E\left[\prod_{j=1}^{\ga} \left(\E_{\sig} \prod_{i=1}^{q}\Big(a_i Y_{ij}+b_i\Big)^{s_{ij}}\right)\bbgiven \gamma \right]
  &= \prod_{j=1}^{\ga}\left(\E\left[ \E_{\sig} \prod_{i=1}^{q}\Big(a_i Y_{i1}+b_i\Big)^{s_{ij}}\bbgiven \gamma \right]\right)\\
  &=\prod_{j=1}^{\ga}\left(\E\prod_{i=1}^{q}\Big(a_i Y_{i1}+b_i\Big)^{s_{ij}}\right),
\end{split}
\end{equation*}
which concludes the proof.
\end{proof}
\subsection{Apriori estimates on $u_n, v_n, w_n$}
Recall the definition of $u_n$ in \eqref{eq:def:u} and $v_n, w_n$ in \eqref{eq:def:v:w}. In this subsection, we show that $|u_n|,|v_n|, |w_n|$ is $O_q(x_n^2)$ \rev{if $n$ is large enough and $x_n\leq \delta$ holds for some \rev{small enough} $\delta\equiv \delta(q)$.}
\begin{lemma}\label{lem:est:Z:mono}
For $k\in \Z$, there exists $C=C(q,k)>0$ such that for each $0\leq k_1,...,k_q\leq k$ and $m\geq 0$,
\begin{equation*}
\left|\E \prod_{i=1}^{q}Z_i^{k_i}-\sum_{\ell=0}^{m}\E\binom{\gamma}{\ell}\left(\E \prod_{i=1}^{q}\left(1+\la q\left(Y_{i1}-\frac{1}{q}\right)\right)^{k_i}-1\right)^{\ell}\right|\leq C (d\la^2 x_n)^{m+1}. 
\end{equation*}
\end{lemma}
\begin{proof}
Denote $y=\E \prod_{i=1}^{q}\left(1+\la q\left(Y_{i1}-1/q\right)\right)^{k_i}-1$. It was shown in Lemma 2.6 of \cite{Sly:11} that for the $d$-ary tree case, $|y|\leq C^\prime\la^2 x_n$ holds for some $C^\prime=C^\prime(k,q)>0$. \rev{Their argument proceeded by expanding $\prod_{i=1}^{q}\left(1+\la q\left(Y_{i1}-1/q\right)\right)^{k_i}$ and realizing that the moments of $\lambda q(Y_{i1}-1/q)$ is bounded by $C\la^2 x_n$.} Having Lemma \ref{lem:basic} and Lemma \ref{lem:Y:moment} in hand, the same argument shows that $|y|\leq C^\prime\la^2 x_n$ holds for the Galton-Watson tree case as well. \rev{Next}, we show how such bound on $|y|$ implies our goal: by Lemma \ref{lem:Y:indep}, $\E\prod_{i=1}^{q}Z_i^{k_i}=\E\prod_{i=1}^{q}\prod_{j=1}^{\ga}\left(1+\la q(Y_{ij}-1/q)\right)^{k_i}=\E(1+y)^{\ga}$ holds, so we have that
\begin{equation*}
    \left|\E\prod_{i=1}^{q}Z_i^{k_i}-\sum_{\ell=0}^{m}\E\binom{\gamma}{\ell}y^{\ell}\right|=\left|\sum_{\ell\geq m+1}\E\binom{\gamma}{\ell}y^{\ell}\right|\leq \sum_{\ell\geq m+1}\E\binom{\gamma}{\ell}\left(C^\prime \la^2 x_n\right)^{\ell},
\end{equation*}
where we used \rev{triangle} inequality and the bound $|y|\leq C^\prime \la^2 x_n$ in the last inequality. Since $\binom{\gamma}{\ell}\leq \frac{\gamma^{\ell}}{\ell!}$ and $d\la^2 x_n\leq x_n<1$ hold, we can further bound the right hand side above by
\begin{equation*}
    \sum_{\ell\geq m+1} \E\binom{\gamma}{\ell}\left(C^\prime \la^2 x_n\right)^{\ell}\leq (d\la^2 x_n)^{m+1}\sum_{\ell\geq m+1} \E\left[\frac{1}{\ell !}\left(\frac{C^\prime \gamma}{d}\right)^{\ell}\right]\leq (d\la^2 x_n)^{m+1} \E\left[\exp\left(\frac{C^\prime \gamma}{d}\right)\right]
\end{equation*}
Now, recall the function $K(\cdot)$ in \eqref{eq:def:K} in Assumption \ref{assumption:unif:tail}. Then, we have shown that
\begin{equation*}
     \left|\E\prod_{i=1}^{q}Z_i^{k_i}-\sum_{\ell=0}^{m}\E\binom{\gamma}{\ell}y^{\ell}\right|\leq (d\la^2 x_n)^{m+1} K(C^\prime),
\end{equation*}
which concludes the proof.
\end{proof}
A key consequence of Lemma \ref{lem:est:Z:mono} is the following: for any polynomial $f(z_1,z_2,...,z_q)$ in $q$ variables,
 \begin{equation}\label{eq:observation:Z}
     \left|\E f(Z_1,Z_2,...,Z_q)-(1-d)f(1,1,...,1)-d\E f\left(\Big(1+\la q(Y_{i1}-1/q)\Big)_{i\leq q}\right)\right|\leq C_q(d\la^2 x_n)^2,
 \end{equation}
which can be seen by applying Lemma \ref{lem:est:Z:mono} for $m=1$ to every monomial of $f(Z_1,...,Z_q)$. Specializing \eqref{eq:observation:Z} to the case of $f(z_1,...,z_q)=g(z_1,...,z_q)\cdot\left(\sum_{i=1}^{q}z_i-q\right)$ for some polynomial $g$, we have a useful bound
\begin{equation}\label{eq:observation:Z:useful}
    \left|\E g(Z_1,Z_2,...,Z_q)\Big(\sum_{i=1}^{q}Z_i-q\Big)\right|\leq C_q(d\la^2 x_n)^2,
\end{equation}
where we used $\sum_{i=1}^{q}\left(Y_{i1}-1/q\right)=0$.
In Section \ref{sec:full}, we will generalize \eqref{eq:observation:Z} to higher orders, which will play a crucial role in estimating $\Xi_2$.

With \eqref{eq:observation:Z} and \eqref{eq:observation:Z:useful}, we can estimate $x_{n+1},u_{n+1},v_{n+1}$, and $w_{n+1}$ up to second order:
\begin{lemma}\label{lem:recurse:first:order}
There exists $C_q>0$ such that for every $n\geq 1$,
\begin{equation*}
    \max\left\{\left|x_{n+1}-d\la^2 x_n\right|, \left|u_{n+1}-d\la^3 u_n\right|\right\}\leq C_q (d\la^2 x_n )^2,
\end{equation*}
and
\begin{equation*}
 \max\left\{\left|v_{n+1}-\left(d\la^4 v_n+\frac{d\la^3(1-\la)}{q}u_n\right)\right|, \left|w_{n+1}-d\la^4 w_n\right|\right\}\leq C_q (d\la^2 x_n )^2.
\end{equation*}
\end{lemma}
\begin{proof}
The claim for $\{x_n\}_{n\geq 1}$ is immediate from the expansion \eqref{eq:expand:Z:sly} and \eqref{eq:observation:Z:useful}: since $Z_1/\sum_{i=1}^{q}Z_i\leq 1$ in the right hand side of \eqref{eq:expand:Z:sly},
\begin{equation*}
  x_{n+1}=\E\frac{Z_1-1}{q}+O_{q}\left((d\la^2 x_n)^2\right)=d\la^2 x_n +O_{q}\left((d\la^2 x_n)^2\right),
\end{equation*}
where we used \eqref{eq:observation:Z} and $\E(Y_{11}-1/q)=\la x_n$ (cf. Lemma \ref{lem:Y:moment}) in the last equality.

For the case of $\{u_n\}_{n\geq 1}$, the identity $\frac{1}{(1+x)^2}=1-2x+\frac{x^2}{(1+x)^2}(3+2x)$ for $x=\left(\sum_{i=1}^{q}Z_i-q\right)/q$ gives
\begin{equation*}
\begin{split}
 z_{n+1}
 &=\E \frac{1}{q^2}\left(Z_1-\frac{\sum_{i=1}^{q}Z_i}{q}\right)^2-\frac{2}{q^2}\left(Z_1-\frac{\sum_{i=1}^{q}Z_i}{q}\right)\left(\frac{\sum_{i=1}^{q}Z_i-q}{q}\right)\\
 &\qquad+\left(\frac{Z_1}{\sum_{i=1}^{q}Z_i}-\frac{1}{q}\right)^2\left(\frac{\sum_{i=1}^{q}Z_i-q}{q}\right)^2\left(3+2\frac{\sum_{i=1}^{q}Z_i-q}{q}\right).
\end{split}
\end{equation*}
Note that \rev{the final term is positive since} $3+2\frac{\sum_{i=1}^{q}Z_i-q}{q}>1$ \rev{holds}. \rev{Moreover, we have the crude bound} $\left(\frac{Z_1}{\sum_{i=1}^{q}Z_i}-\frac{1}{q}\right)^2\leq 1$, thus using \eqref{eq:observation:Z:useful} in the right hand side above, we have the estimate
\begin{equation}\label{eq:z:recurse}
\begin{split}
    z_{n+1}&=\E\frac{1}{q^2}\left(Z_1-\frac{\sum_{i=1}^{q}Z_i}{q}\right)^2+O_{q}\left((d\la^2 x_n)^2\right)\\
    &=\frac{d\la^2 x_n}{q}+d\la^3\left(z_n-\frac{x_n}{q}\right)+O_{q}\left((d\la^2 x_n)^2\right),
\end{split}
\end{equation}
where we used \eqref{eq:observation:Z} and Lemma \ref{lem:Y:moment} in the last equality. Since $x_{n+1}=d\la^2 x_n+O_{q}\left((d\la^2 x_n)^2\right)$, we have
\begin{equation*}
    u_{n+1}\equiv z_{n+1}-\frac{x_{n+1}}{q}=d\la^3 u_n+O_{q}\left((d\la^2 x_n)^2\right).
\end{equation*}

For the case of $\{v_n\}_{n\geq 1}$ and $\{w_n\}_{n\geq 1}$, we use the identity $\frac{1}{(1+x)^3}=1-3x+\frac{x^2}{(1+x)^3}(6+8x+3x^2)$ for $x=(\sumi Z-q)/q$ to see
\begin{equation}\label{eq:v:recurse}
    \begin{split}
   v_{n+1}&=\E\frac{1}{q^3} \left(Z_1-\frac{\sumz}{q}\right)^3-\frac{3}{q^3}\left(Z_1-\frac{\sumz}{q}\right)^3\left(\frac{\sumz -q}{q}\right)\\
        &+\left(\frac{Z_1}{\sum_{i=1}^{q}Z_i}-\frac{1}{q}\right)^3\left(\frac{\sumz -q}{q}\right)^2\left(6+8\frac{\sumz -q}{q}+3\left(\frac{\sumz -q}{q}\right)^2\right),
    \end{split}
\end{equation}
and
\begin{equation}\label{eq:w:recurse}
    \begin{split}
   w_{n+1} &=\E\frac{1}{q^3} \left(Z_1-\frac{\sumz}{q}\right)\left(Z_2-\frac{\sumz}{q}\right)(Z_1-Z_2)\\
        &\quad-\frac{3}{q^3}\left(Z_1-\frac{\sumz}{q}\right)\left(Z_2-\frac{\sumz}{q}\right)(Z_1-Z_2)\left(\frac{\sumz -q}{q}\right)\\
        &\quad+\left(\frac{Z_1}{\sum_{i=1}^{q}Z_i}-\frac{1}{q}\right)\left(\frac{Z_2}{\sum_{i=1}^{q}Z_i}-\frac{1}{q}\right)\left(\frac{Z_1-Z_2}{\sum_{i=1}^{q}Z_i}\right)\left(\frac{\sumz -q}{q}\right)^2\\
        &\quad\quad\quad\quad\cdot\left(6+8\frac{\sumz -q}{q}+3\left(\frac{\sumz -q}{q}\right)^2\right).
    \end{split}
\end{equation}
Regarding the last terms in the right hand side of \eqref{eq:v:recurse} and \eqref{eq:w:recurse}, note that $6+8x+3x^2=3(x+4/3)^2+2/3>0$ holds for any $x\in \R$, and 
\begin{equation*}
    -1 \leq \left(\frac{Z_1}{\sum_{i=1}^{q}Z_i}-\frac{1}{q}\right)^3\leq 1, \quad -1\leq \left(\frac{Z_1}{\sum_{i=1}^{q}Z_i}-\frac{1}{q}\right)\left(\frac{Z_2}{\sum_{i=1}^{q}Z_i}-\frac{1}{q}\right)\left(\frac{Z_1-Z_2}{\sum_{i=1}^{q}Z_i}\right)\leq 1.
\end{equation*}
Thus, we can proceed as in \eqref{eq:z:recurse}. Using \eqref{eq:observation:Z}, \eqref{eq:observation:Z:useful}, and Lemma \ref{lem:Y:moment} to the equations above, it follows that
\begin{equation*}
    v_{n+1}=\E\frac{1}{q^3} \left(Z_1-\frac{\sumz}{q}\right)^3+O_q\left((d\la^2 x_n^2)\right)=d\la^4+\frac{d\la^3(1-\la)}{q}u_n+O_q\left((d\la^2 x_n^2)\right),
\end{equation*}
and
\begin{equation*}
\begin{split}
    w_{n+1}
    &=\E\frac{1}{q^3} \left(Z_1-\frac{\sumz}{q}\right)\left(Z_2-\frac{\sumz}{q}\right)(Z_1-Z_2)+O_q\left((d\la^2 x_n^2)\right)\\
    &=d\la^4 w_n+O_q\left((d\la^2 x_n^2)\right),
\end{split}
\end{equation*}
which concludes the proof.
\end{proof}
A crucial observation in the statement of Lemma \ref{lem:recurse:first:order} is that the estimates for $u_{n+1},v_{n+1}$, and $w_{n+1}$ has an extra factor of $\la$ compared to the estimate for $x_{n+1}$. \rev{By} taking advantage of this \rev{extra factor}, we \rev{will prove in Lemma~\ref{lem:u:v:w:apriori}} that $|u_n|, |v_n|,|w_n|=O_q(x_n^2)$ hold. \rev{To prove this, the following result will be useful, which is a generalization of Lemma 2.9 of \cite{Sly:11} and shows that $x_n$ cannot drop from a very large value to a very small one.}

\begin{lemma}\label{lem:drop:small}
\rev{There exists a constant $C=C_q$ such that for any $n$,
\[
x_{n+1}\geq C_{q} \la^4 x_n
\]
In particular,} for any $\kappa>0$, there exists a constant $c=c(q,\kappa,d)$ such that whenever $|\la|>\kappa$, $x_{n+1}\geq c\cdot x_n$ \rev{holds } for all $n$.

\end{lemma}
\begin{proof}
Recall from Remark \ref{rmk:x:positive} that Lemma 2.9 of \cite{Sly:11} showed that \rev{for an arbitrary tree $T$ and $j\leq \ga$}, $x_{n+1}(T)\geq \la^4 x_n(T_j)^2$ for every $n$. Letting $j=1$ and, taking expectation with respect to $T\sim \GW(\mu_d)$, 
\begin{equation*}
    x_{n+1}\geq \la^4 \E\left[x_n(T_1)^2\right]\geq \la^4 \E\left[x_n(T_1)^2\bgiven \ga \geq 1\right]\cdot \P(\ga \ge 1).
\end{equation*}
Note that for the Galton-Watson tree $T\sim \GW(\mu_d)$, the subtree $T_1$ follows the same distribution $T_1\sim \GW(\mu_d)$ conditional on the event $\ga \geq 1$.
\rev{
Thus, using Cauchy-Schwartz inequality, it follows that
\begin{equation}\label{eq:x:lower:bound}
 x_{n+1}\geq \la^4 \Big(\E[x_n(T_1)\given \ga \geq 1]\Big)^2 \cdot \P(\ga \ge 1)=\la^4 x_n^2\cdot\P(\ga \ge 1).
\end{equation}
Here, we recall that $\ga\sim \mu_d$, so $\P(\ga\geq 1)$ depends on $d$. Next, we claim that under Assumption~\ref{assumption:unif:tail}, there exists a constant $c>0$ which only depends on $K(\cdot)$ such that $\P(\ga\geq 1)\geq c$ for any $d>1$. To see this, fix a constant $C>1$ to be determined below. By definition of $K(\cdot)$ in~\eqref{eq:def:K}, we have for any $d>1$ that 
\[
K(1)\geq \E\big[e^{\ga/d}\one(\ga\geq Cd)\big]\geq \E\Big[\frac{\ga}{d}\cdot \frac{e^{C}}{C}\one(\ga\geq Cd)\Big]\,,
\]
where the last inequality holds since $x\to e^{x}/x$ is increasing for $x>1$, thus $e^{\ga/d}\geq \frac{\ga}{d}\frac{e^{C}}{C}$ holds on the event $\ga/d\geq C$. Hence, by taking $C>1$ to be sufficiently large,
\[
\E\big[\ga \one (\ga\geq Cd)\big]\leq \frac{C\cdot K(1)}{e^{C}}d\leq \frac{d}{2}.
\]
Since $\ga$ is integer valued and $\E \ga =d$, this further implies that
\[
\E\big[\ga \one (1\leq \ga\leq Cd)\big]=d-\E\big[\ga \one (\ga\geq Cd)\big]\geq \frac{d}{2}.
\]
Note that the left hand side is upper bounded by $Cd\cdot \P(\ga\geq 1)$. Thus, the inequality above implies that $\P(\ga\geq 1)\geq 1/(2C)$, which proves our claim. Consequently, combining with~\eqref{eq:x:lower:bound}, 
\[
x_{n+1}\geq c\la^4 x_n^2.
\]
Finally, if $|\la|\geq \kappa$ holds, this implies that $x_{n+1}\geq c\kappa^4 x_n$ concluding the proof.
}
\end{proof}
\begin{lemma}\label{lem:u:v:w:apriori}
There exist $\delta=\delta(q)$ and $C=C_q>0$ such that the following holds: if \rev{$x_{n_0}\leq \delta$ for some $n_0\geq 1$, then for all $n\geq n_0+5$,}
\begin{equation*}
\max\left(|u_n|,|v_n|,|w_n|\right)\leq C x_n^2
\end{equation*}
\end{lemma}
\begin{proof}
 We start with the bound for $u_n$: by Lemma \ref{lem:recurse:first:order}, there exists $C_q>0$ such that 
\begin{equation}\label{eq:u:upper}
    |u_{n+1}|\leq d|\la|^3|u_n|+C_q (d\la^2 x_n)^2\quad\textnormal{and}\quad x_{n+1}\geq d\la^2 x_n-C_q (d\la^2 x_n)^2.
\end{equation}
\rev{Recalling the constant $c_0$ from~\eqref{eq:dlambda:lower}, we let $\delta(q)$ to be a small enough constant that only depends on $q$ such that $(1+c_0)\left(1-C_q\delta(q)\right)^2\geq 1+\frac{c_0}{2}$ and $\delta(q)<\frac{q-1}{q}$ hold. Now, suppose there exists $n_0$ such that $x_{n_0}<\delta$. Since $(x_{n})_{n\geq 0}$ is non-increasing by Lemma~\ref{lem:equiv:xn:nonreconstruct}, w.l.o.g., we can take
$n_0=\inf\{n\geq 0:x_n\leq \delta(q)\}$. Then, note that
\[
x_n\leq \delta(q)\quad\text{for all}\quad n\geq n_0\quad\text{and}\quad x_{n_0-1}\geq \delta(q).
\]
} Iterating the first inequality \eqref{eq:u:upper} $n$ times from $n+n_0$ gives 
\begin{equation}\label{eq:u:upper:iterate}
    |u_{n+n_0}|\leq C_q d^2\la^4 \sum_{i=1}^{n}\left(d|\la|^3\right)^{i-1}x_{n+n_0-i}^2 + \left(d|\la|^3\right)^n |u_{n_0}|.
\end{equation}
Note that the second inequality in \eqref{eq:u:upper} tells us that $x_{n+1}=(1-o(1))d\la^2 x_n$, which enables us to bound the sum in \eqref{eq:u:upper:iterate}: \rev{since $x_n\leq \delta(q)$ holds for all $n\geq n_0$, we have for all $n\geq n_0$}
\begin{equation*}
    x_{n+1} \geq d\la^2(1-C_q \delta_q)x_n,
\end{equation*}
thus for $1\leq i \leq n$,
\begin{equation}\label{eq:x:square:tech}
\begin{split}
    d^2\la^4\left(d|\la|^3\right)^{i-1}x_{n+n_0-i}^2
    &\leq \frac{1}{(1-C_q\delta(q))^2}\left(\frac{1}{d|\la|(1-C_q \delta(q))^2}\right)^{i-1}x_{n+n_0}^2\\
    &\leq 2\left(\frac{1}{1+\frac{c_0}{2}}\right)^{i-1}x_{n+n_0}^2, 
\end{split}
\end{equation}
where \rev{the final inequality holds since $d|\la|>1+c_0$ (cf.~\eqref{eq:dlambda:lower}) and we have set $\delta(q)$ small enoughg so that $(1+c_0)\left(1-C_q\delta(q)\right)^2\geq 1+\frac{c_0}{2}$ holds.} Similarly, we can bound
\begin{equation}\label{eq:u:neg:1}
    \left(d|\la|^3\right)^n |u_{n_0}|=\left(d|\la|^3\right)^n x_{n_0}^2\cdot \frac{|u_{n_0}|}{x_{n_0}^2}\leq \rev{\left(\frac{1}{d|\la|(1-C_q \delta(q))^2}\right)^{n}}\frac{|u_{n_0}|}{x_{n_0}^2}x_{n+n_0}^2.
\end{equation}
\rev{We now argue that if $n\geq 5$, then the term $\left(\frac{1}{d|\la|(1-C_q \delta(q))^2}\right)^{n}\frac{|u_{n_0}|}{x_{n_0}^2}$ is $O_q(1)$. Recall that by Lemma~\ref{lem:drop:small}, the decrease from $x_{n_0-1}$ to $x_{n_0}$ cannot be too large. Namely, since $x_{n_0-1}\geq \delta(q)$, we have the lower bound $x_{n_0}\geq C_q^\prime \la^4\delta(q)$. Moreover, $|u_{n_0}|\leq x_{n_0}$ holds by Lemma~\ref{lem:basic}, so
\[
\left(\frac{1}{d|\la|(1-C_q \delta(q))^2}\right)^{5}\frac{|u_{n_0}|}{x_{n_0}^2}\leq \frac{1}{(1-C_q\delta(q))^{10}}\frac{1}{d^{5}|\la|^{5}x_{n_0}}\leq \frac{1}{C^\prime_q\delta(q)(1-C_q\delta(q))^{10}}\cdot \frac{1}{d^{5}|\la|^{9}}.
\]
Recall that $d^{5}|\la|^{9}\geq 1+c_0$ holds (cf.~\eqref{eq:dlambda:lower}) and that $(1+c_0)\left(1-C_q\delta(q)\right)^2\geq 1+\frac{c_0}{2}$ holds. Thus, plugging these estimates into the right hand side of~\eqref{eq:u:neg:1} shows that for $n\geq 5$,
\begin{equation}\label{eq:u:neg}
    \left(d|\la|^3\right)^n |u_{n_0}|\leq \left(\frac{1}{1+\frac{c_0}{2}}\right)^{n-5}\cdot O_q(x_{n+n_0}^2).
\end{equation}
}
\rev{ Hence, by \eqref{eq:u:upper:iterate}, \eqref{eq:x:square:tech}, and \eqref{eq:u:neg}, we have for $n\geq 5$ that}
\begin{equation}\label{eq:u:upper:final}
    |u_{n+n_0}|\leq \left(2C_q\sum_{i=1}^{n}\left(\frac{1}{1+\frac{c_0}{2}}\right)^{i-1}+\left(\frac{1}{1+\frac{c_0}{2}}\right)^{\rev{n-5}}\cdot \rev{C_q^{\prime\prime}}\right)x_{n+n_0}^2=O_q(x_{n+n_0}^2),
\end{equation}
\rev{which concludes the proof for $u_n$}. For the case of $v_n$ and $w_n$, Lemma \ref{lem:recurse:first:order} and \eqref{eq:u:upper:final} show that there exists $C_q>0$ such that
\begin{equation*}
    |v_{n+1}|\leq d|\la|^3|v_n|+C_q\left(d\la^2 x_n\right)^2\quad\textnormal{and}\quad |w_{n+1}|\leq d|\la|^3|w_n|+C_q\left(d\la^2 x_n\right)^2,
\end{equation*}
where we used the fact $|\la|\leq 1$ and $d|\la|>1+c_0$. Therefore, we can repeat the previous argument to conclude that $|v_{n}|=O_q(x_n^2)$ and $|w_{n}|=O_q(x_n^2)$ hold \rev{for $n\geq n_0+5$}.
\end{proof}
\rev{
We note that Lemma~\ref{lem:u:v:w:apriori} implies that if $\lim_{n\to\infty} x_n\leq \delta(q)$ holds, then for sufficiently large enough $n\geq N(q,d,\lambda)$, $\max(|u_n|,|v_n|, |w_n|)=O_q(x_n^2)$ holds. This assertion, although less general, is sufficient to prove the tightness of the KS bound stated in Theorem~\ref{thm:KS:tight}. However, to establish the non-tightness of the KS bound in Theorem~\ref{thm:four:antiferro}, it is beneficial to know that $\max\left(|u_n|,|v_n|,|w_n|\right)=O_q(x_n^2)$ holds if $n\geq n_0+N(q,d)$ and $x_{n_0}\leq \delta$, where $N(q,d)$ does not depend on $\lambda$. The detailed proof of Theorem~\ref{thm:four:antiferro} can be found at the end of Section~\ref{sec:full}. 
}

Recalling the definition of $\Xi_1$ in \eqref{eq:def:xi}, we estimate $\Xi_1$ up to third order with respect to $x_n$.
\begin{lemma}\label{lem:xi1}
There exist $\delta=\delta(q)$ and $C=C_q>0$ such that the following holds: if \rev{$x_{n_0}\leq \delta$ for some $n_0\geq 1$, then for all $n\geq n_0+5$,}
\begin{equation}\label{eq:lem:xi1}
    \bigg|\Xi_1-\left(d\la^2 x_n+\frac{q(q-4)}{q-1}\E\binom{\gamma}{2}\la^4 x_n^2+\frac{q^2(q^2-21q+33)}{(q-1)^2}\E\binom{\gamma}{3}\la^6 x_n^3\frac{4q^2}{q-1}\E\binom{\gamma}{2}\la^5 x_n u_n\right)\bigg|\leq C_q x_n^4
\end{equation}
\end{lemma}
\begin{proof}
Applying Lemma \ref{lem:est:Z:mono} with $m=3$ and Lemma \ref{lem:Y:moment} gives
\begin{equation}\label{eq:lem:xi1:1}
    \E\frac{Z_1-1}{q}=d\la^2 q x_n+\E \binom{\gamma}{2}\la^4 q x_n^2+\E \binom{\gamma}{3}\la^6 q^2 x_n^3+O_q(x_n^4).
\end{equation}
Similarly, we have that $\E Z_1(\sumz -q)/q^2$ equals
\begin{equation}\label{eq:lem:xi1:2}
\begin{split}
    &\frac{\E Z_1^2}{q^2}+\frac{(q-1)\E Z_1Z_2}{q^2} -\frac{\E Z_1}{q}\\
    &=\E\binom{\gamma}{2}\la^4\left((3x_n+\la q u_n)^2+\frac{1}{q-1}((q-3)x_n-\la q u_n)^2-q x_n^2\right)\\
    &\quad\quad+\E \binom{\gamma}{3}\la^6\left(q(3x_n+\la q u_n)^3+\frac{q}{(q-1)^2}((q-3)x_n-\la q u_n)^3-q^2 x_n^3\right)+O_q(x_n^4),
\end{split}
\end{equation}
and that $\E (\sumi Z_i-q)^2/q^2$ equals
\begin{equation}\label{eq:lem:xi1:3}
\begin{split}
    &\frac{\E Z_1^2}{q^3}+\frac{(q-1)\E Z_2^2}{q^3}+\frac{2(q-1)\E Z_1Z_2}{q^3}-\frac{(q-1)(q-2)\E Z_2 Z_3}{q^3}-\frac{2 \E Z_1}{q^2}-\frac{2(q-1) \E Z_2}{q^2}+q^2\\
    &=\E\binom{\gamma}{2}\la^4\bigg(\frac{1}{q}(3x_n+\la q u_n)^2+\frac{3}{q(q-1)}\left((q-3)x_n-\la q u_n\right)^2\\
    &\qquad\qquad\qquad+\frac{1}{q(q-1)(q-2)}\left((3q-6)x_n-2\la q u_n\right)^2-2x_n^2-\frac{2}{q-1}x_n^2\bigg)\\
    &\quad\E\binom{\gamma}{3}\la^6\bigg((3x_n+\la q u_n)^3+\frac{3}{(q-1)^2}\left((q-3)x_n-\la q u_n\right)^3\\
    &\qquad\qquad\qquad -\frac{1}{(q-1)^2(q-2)^2}\left((3q-6)x_n-2\la q u_n\right)^3-2q x_n^3+\frac{2q}{(q-1)^2}x_n^3\bigg)+O_q(x_n^4).
\end{split}
\end{equation}
Observe that by Lemma \ref{lem:u:v:w:apriori}, $u_n^2+x_n^2|u_n|=O_q(x_n^4)$ \rev{holds since we assumed that $n\geq n_0+5$ for some $n_0$ such that $x_{n_0}\leq \delta$ holds}. Also, Assumption \ref{assumption:unif:tail} implies that $\E\binom{\gamma}{\ell}= O(d^{\ell})$ for $\ell =2,3$, since we have
\begin{equation*}
    \E \frac{1}{d^{\ell}}\binom{\gamma}{\ell}\leq \E\frac{1}{\ell !}\left(\frac{\gamma}{d}\right)^{\ell}\leq \E \exp\left(\frac{\gamma}{d}\right).
\end{equation*}
Thus, in the equations in \eqref{eq:lem:xi1:2} and \eqref{eq:lem:xi1:3}, we can absorb all the terms which involve $u_n^2, x_n^2 u_n, x_n u_n^2,$ or $u_n^3$ into $O_q(x_n^4)$. After such simplification, subtracting \eqref{eq:lem:xi1:2} from \eqref{eq:lem:xi1:1} and adding \eqref{eq:lem:xi1:3} gives the estimate \eqref{eq:lem:xi1}.
\end{proof}

\section{Full expansion of the magnetization}
\label{sec:full}
In this section, we estimate $\Xi_2$ up to third order and establish Theorem \ref{thm:four:antiferro}.

\subsection{Expansion for $\Xi_2$}
As explained in Section \ref{subsec:overview:broadcast}, in order to estimate $\Xi_2$ accurately up to the third order with respect to $x_n$, we need a better way of estimating the expectation of polynomials of $Z_i$'s than Lemma \ref{lem:est:Z:mono}. A key observation is that the third and higher moments of $Y_{i1}$'s are of order $O_q(x_n^2)$.
\begin{lemma}\label{lem:Y:higher:mo}
There exists a $C_q>0$ such that for every for every non-negative integers $(k_{i})_{i\leq q}$ with $(k_1,...,k_q)\neq (0,...,0)$,
\begin{equation}\label{eq:lem:Y:lower:mo}
     \left|\E\prod_{i=1}^{q}\left(Y_{i1}-\frac{1}{q}\right)^{k_i}\right|\leq C_q x_n.
\end{equation}
Moreover, there exists $\delta=\delta(q)$ such that if \rev{$x_{n_0}\leq \delta$ holds for some $n_0\geq 1$, then for all $n\geq n_0+5$ and} for every non-negative integers $(k_{i})_{i\leq q}$ with $\sum_{i=1}^{q}k_i\geq 3$,
\begin{equation}\label{eq:lem:Y:higher:mo}
    \left|\E\prod_{i=1}^{q}\left(Y_{i1}-\frac{1}{q}\right)^{k_i}\right|\leq C_q x_n^2.
\end{equation}
\end{lemma}
\begin{proof}
We first prove \eqref{eq:lem:Y:lower:mo}. When $\sum_{i=1}^{q}k_i\leq 2$, \eqref{eq:lem:Y:lower:mo} is immediate since $0\leq z_n\leq x_n$ holds by Lemma \ref{lem:basic}, and the second moments of $Y_{i1}$'s are expressed in terms of a linear combination of $x_n$ and $z_n$ by Lemma \ref{lem:Y:moment}. \rev{Thus, we have by AM-GM that for $i_1,i_2\in \{1,\ldots, q\}^2$,
\begin{equation*}
    \E\left|\left(Y_{i_1 1}-\frac{1}{q}\right)\left(Y_{i_2 1}-\frac{1}{q}\right)\right|\leq \frac{1}{2}\left(\E\left(Y_{i_1 1}-\frac{1}{q}\right)^2+\E\left(Y_{i_1 1}-\frac{1}{q}\right)^2\right)=O_q(x_n).
\end{equation*}
Note that the absolute value is inside the expectation in the left hand side. Moreover, $\left|Y_{i1}-1/q\right|\leq \frac{q-1}{q}$ holds by definition, thus \eqref{eq:lem:Y:lower:mo} holds when $\sum_{i=1}^{q}k_i\geq 3$.}

Next, we prove \eqref{eq:lem:Y:higher:mo} for the case where $\sum_{i=1}^{q}k_i=3$. \rev{Note that by Lemma~\ref{lem:Y:moment}, the third moments of $Y_{i1}$'s are expressed in terms of a linear combination of $u_n, v_n$, and $w_n$. Moreover, $|u_n|\vee|v_n|\vee|w_n|\leq C_q x_n^2$ holds under the stated condition by Lemma \ref{lem:u:v:w:apriori}, thus \eqref{eq:lem:Y:higher:mo} holds when $\sum_{i=1}^{q}k_i=3$.} If $\sum_{i=1}^{q}k_i\geq 4$, then we have by Lemma \ref{lem:basic} that
\begin{equation*}
  \E\left(X^{+}(n)-\frac{1}{q}\right)^{4}+(q-1)\E \left(X^{-}(n)-\frac{1}{q}\right)^{4}= v_n-\frac{1}{q}u_n\leq C_q x_n^2, 
\end{equation*}
where the final inequality is due to Lemma \ref{lem:u:v:w:apriori}. Thus, proceeding as in \eqref{eq:Y:X:plus:minus}, Proposition~\ref{prop:Y:dist}~-$(d)$ implies that we can bound for any $i\leq q$, 
\begin{equation*}
    \E\left(Y_{i1}-\frac{1}{q}\right)^4\leq \max\left(\E\left(X^{+}(n)-\frac{1}{q}\right)^{4}, \E \left(X^{-}(n)-\frac{1}{q}\right)^{4}\right)\leq C_q x_n^2.
\end{equation*}
This implies that our goal \eqref{eq:lem:Y:higher:mo} holds whenever $\sum_{i=1}^{q}k_i=4$: by the weighted AM-GM inequality, we have for $\sum_{i=1}^{q}k_i=4$,
\begin{equation}\label{eq:bound:Y:fourth}
    \E\left|\prod_{i=1}^{q}\left(Y_{i1}-\frac{1}{q}\right)^{k_i}\right|\leq \E \sum_{i=1}^{q}\frac{k_i}{4}\left(Y_{i1}-\frac{1}{q}\right)^4 \leq C_q x_n^2.
\end{equation}
Note that $\left|Y_{i1}-1/q\right|\leq \frac{q-1}{q}$ holds, so whenever $\sum_{i=1}^{q}k_i\geq 4$, \eqref{eq:bound:Y:fourth} continues to hold with the same constant $C_q>0$.
\end{proof}

With Lemma \ref{lem:Y:higher:mo} in hand, we prove the following proposition which plays a crucial role in estimating $\Xi_2$.
\begin{prop}\label{prop:crucial}
Given integers $m\geq 1, \gamma\geq 0$ and a polynomial $f\in \R[z_1,...,z_q]$, define the polynomial $\Phi_{m,f}\equiv \Phi_{m,f,\gamma}\in \R\left[(x_{ij})_{i\leq q, j\leq \gamma}\right]$ in $q\times \gamma$ variables as follows. Given $f$, let $F_{f}\equiv F_{f,\gamma}\in \R\left[(x_{ij})_{i\leq q, j\leq \gamma}\right]$ be
\begin{equation}\label{eq:def:F}
F_{f}\Big((x_{ij})_{i\leq q, j\leq \gamma}\Big):=f(z_1,...,z_q),\quad\textnormal{where}\quad z_i=\prod_{j=1}^{\gamma}\left(1+\la q x_{ij}\right).
\end{equation}
Then, we define $\Phi_{m,f}\equiv \Psi_m[F_{f,\gamma}]$, where $\Psi_m[\cdot]$ is an operator acting on the space $\R[(x_{ij})_{i\leq q, j\leq \gamma}]$ by deleting the monomials $\prod_{i=1}^{q}\prod_{j=1}^{\gamma}x_{ij}^{s_{ij}}$, where $s=(s_{ij})_{i\leq q, j\leq \gamma}$ satisfies \textit{either} of the following.
\begin{enumerate}[label=(\alph*)]
    \item There exist $m$ distinct integers $j_1,...,j_{m}\in [\gamma]$ such that $\sum_{i=1}^{q}s_{ij_{\ell}}\geq 1$ for every $1\leq \ell \leq m$.
    \item For some positive integer $r$ in the range $\frac{m}{2}\leq r\leq m$, there exist $r$ distinct integers $j_1,...,j_{r}\in [\gamma]$ such that $\sum_{i=1}^{q}s_{ij_{\ell}}\geq 3$ for every $1\leq \ell \leq m-r$ and $\sum_{i=1}^{q}s_{ij_{\ell}}\geq 1$ for every $m-r+1\leq \ell\leq r$.
\end{enumerate}
Then, there exist $C=C(q,m,f)$, which only depends on $q$, $m\geq 1$, and $f$, and $\delta=\delta(q)$ such that the following holds: if \rev{$x_{n_0}\leq \delta$ holds for some $n_0\geq 1$, then for all $n\geq n_0+5$, $m\geq 1$, and} polynomial $f(z_1,...,z_q)$ in $q$ variables, \rev{we have}
\begin{equation}\label{eq:prop:crucial}
    \left|\E f(Z_1,...,Z_q)-\E\Phi_{m,f} \bigg(\Big(Y_{ij}-\frac{1}{q}\Big)_{i\leq q, j\leq \gamma}\bigg)\right|\leq C x_n^m.
\end{equation}
\end{prop}
\begin{proof}
We first express the polynomial $f$ as
\begin{equation*}
    f(z_1,...,z_q)=\sum_{k=(k_1,...,k_q)\in \mathcal{K}}a_k\cdot  \prod_{i=1}^{q}z_i^{k_i}, 
\end{equation*}
where $\mathcal{K}$ is a subset of $\N^{q}$ and $a_k\in \R$. Then,
\begin{equation*}
    F\Big((x_{ij})_{i\leq q, j\leq \gamma}\Big)=\sum_{k=(k_1,...,k_q)\in \mathcal{K}}a_k\cdot  \prod_{i=1}^{q}\prod_{j=1}^{\gamma}\left(\sum_{s_{ij}=0}^{k_i}\binom{k_i}{s_{ij}}(\la q)^{s_{ij}}x_{ij}^{s_{ij}}\right)=\sum_{s\in \mathcal{S}_{\gamma}}b_{s}\cdot \prod_{i=1}^{q}\prod_{j=1}^{\gamma}x_{ij}^{s_{ij}},
\end{equation*}
where $\mathcal{S}_{\gamma}$ is the subset of $s=(s_{ij})_{i\leq q, j\leq \gamma}$ such that $s_{ij}\in \{0,...,k_i\}, \forall i\leq q, \forall j\leq \gamma$ for some $(k_1,...,k_q)\in \mathcal{K}$, and $b_s\equiv b_s(f,\gamma)\in \R$ is given by 
\begin{equation}\label{eq:def:b}
    b_s \equiv \sum_{k\in \mathcal{K}}a_k\left(\la q \right)^{\norm{s}_1}\prod_{i=1}^{q}\prod_{j=1}^{\gamma}\binom{k_i}{s_{ij}}.
\end{equation}
Let $\mathcal{S}_{\gamma}(m)\subset \mathcal{S}_{\gamma}$ be the set of $s$ which satisfy either $(a)$ or $(b)$ above, i.e. the deleted indices from $F$ to $\Phi_m[f]$. Then, by definition of $Z_i$'s and tower property,
\begin{equation*}
    \E f(Z_1,...,Z_q)-\E\Phi_m \bigg(\Big(Y_{ij}-\frac{1}{q}\Big)_{i\leq q, j\leq \gamma}\bigg)=\E\left[ \E\left[\sum_{s\in \mathcal{S}_{\gamma}(m)}b_{s}\cdot \prod_{i=1}^{q}\prod_{j=1}^{\gamma}\left(Y_{ij}-\frac{1}{q}\right)^{s_{ij}}\bbgiven \gamma \right]\right].
\end{equation*}
Following the same argument as in the proof of Lemma \ref{lem:Y:indep}, it follows from Proposition \ref{prop:Y:dist}-$(a),(b)$ that
\begin{equation*}
    \E\left[ \prod_{i=1}^{q}\prod_{j=1}^{\gamma}\left(Y_{ij}-\frac{1}{q}\right)^{s_{ij}}\bbgiven \gamma \right]=\prod_{j=1}^{\gamma}\E\left[\prod_{i=1}^{q}\left(Y_{i1}-\frac{1}{q}\right)^{s_{ij}}\right].
\end{equation*}
Thus, it follows from a \rev{triangle} inequality that
\begin{equation}\label{eq:deleted:mo:ineq}
    \left|\E f(Z_1,...,Z_q)-\E\Phi_m \bigg(\Big(Y_{ij}-\frac{1}{q}\Big)_{i\leq q, j\leq \gamma}\bigg)\right|\leq \E \sum_{s\in \mathcal{S}_{\gamma}(m)}|b_s|\cdot \prod_{j=1}^{\gamma}\left|\E\left[\prod_{i=1}^{q}\left(Y_{i1}-\frac{1}{q}\right)^{s_{ij}}\right]\right|,
\end{equation}
where the outer expectation of the right hand side above is with respect to $\gamma$.

The most important observation on the deleting procedure is as follows. By Lemma~\ref{lem:Y:higher:mo}, there exists a constant $C(q,m)>0$, which only depends on $q$ and $m$, such that if \rev{$x_{n_0}\leq\delta$ and $n\geq n_0+5,$ then}
\begin{equation}\label{eq:deleted:mo}
    \prod_{j=1}^{\gamma}\left|\E\left[\prod_{i=1}^{q}\left(Y_{i1}-\frac{1}{q}\right)^{s_{ij}}\right]\right|\leq C(q,m) x_n^{m}\quad\textnormal{for any $\gamma$ and $s\in \mathcal{S}_{\gamma}(m)$.}
\end{equation}
Note that $C(q,m)$ can be chosen to \textit{not} depend on $\ga$ since $|Y_{i1}-1/q|\leq 1$ a.s.. Moreover, we claim the following bound on $(b_s)_{s\in \mathcal{S}_{\gamma}}$:
\vspace{3mm}
\newline
{\bf Claim:} Denote $\norm{f}_{1}:=\sum_{k\in \mathcal{K}}|a_k|$ by the sum of the absolute values of the coefficients of $f$. Also, given $\gamma$ and $s\in \mathcal{S}_{\gamma}$, denote $\norm{s}_{*}:=\sum_{j=1}^{\gamma}\one\left(\sum_{i=1}^{q}s_{ij}\geq 1 \right)$ by the number of index $j$ such that $s_{ij}\geq 1$ for some $1\leq i\leq q $. Then, given $f$ and $\gamma$, we have for any $\ell, r\geq 0$ that
\begin{equation}\label{eq:claim}
\sum_{\substack{s\in \mathcal{S}_{\gamma}:\\\  \norm{s}_{*}=\ell, ~\norm{s}_1=r}} \left|b_s\right|\leq \left(|\la|q\right)^{r}\cdot \norm{f}_1 \cdot 2^{\deg(f)\cdot \ell}\cdot \binom{\gamma}{\ell} 
\end{equation}
We now prove the claim: recalling the definition of coefficients $b_s$ in \eqref{eq:def:b}, we can bound
\begin{equation}\label{eq:claim:1}
    \sum_{\substack{s\in \mathcal{S}_{\gamma}:\\\  \norm{s}_{*}=\ell, ~\norm{s}_1=r}} \left|b_s\right|\leq \norm{f}_1\cdot \left(|\la|q\right)^{r}\cdot \sup_{k\in \mathcal{K}}\left\{\sum_{s\in \mathcal{S}_{\gamma}:\norm{s}_{*}=\ell}~\prod_{j=1}^{\gamma}\prod_{i=1}^{q}\binom{k_i}{s_{ij}}\right\}.
\end{equation}
The sum inside the supremum in \eqref{eq:claim:1} can be bounded as follows. Note that if $\norm{s}_{*}=\ell$, the number of choosing the location of $j$'s such that $\sum_{i=1}^{q}s_{ij}\geq 1$ is $\binom{\gamma}{\ell}$ and $\prod_{i=1}^{q}\binom{k_i}{0}=1$ holds for the other $j$'s, so for each $k\in \mathcal{K}$,
\begin{equation*}
   \sum_{s\in \mathcal{S}_{\gamma}:\norm{s}_{*}=\ell}~\prod_{j=1}^{\gamma}\prod_{i=1}^{q}\binom{k_i}{s_{ij}}\leq \binom{\gamma}{\ell}\sup_{1\leq j_1,...,j_{\ell}\leq \gamma}\sum_{s_{i j_{t}}\geq 0, 1\leq t\leq \ell}~\prod_{t=1}^{\ell}\prod_{i=1}^{q}\binom{k_i}{s_{ij_{t}}}=\binom{\gamma}{\ell} 2^{\ell\cdot \sum_{i=1}^{q}k_i}.
\end{equation*}
Note that $\sup_{k\in \mathcal{K}}\{\sum_{i=1}^{q}k_i\}=\deg(f)$, so 
\begin{equation}\label{eq:claim:2}
    \sup_{k\in \mathcal{K}}\left\{\sum_{s\in \mathcal{S}_{\gamma}:\norm{s}_{*}=\ell}~\prod_{j=1}^{\gamma}\prod_{i=1}^{q}\binom{k_i}{s_{ij}}\right\}\leq \binom{\gamma}{\ell}2^{\ell\cdot \deg(f)}.
\end{equation}
Therefore, \eqref{eq:claim:1} and \eqref{eq:claim:2} finishes the proof of our claim \eqref{eq:claim}.

Now, observe that when $\norm{s}_{*}=\ell$ and $\norm{s}_1=r$ hold with $2\ell>r$, it follows that
\begin{equation}\label{eq:tech}
    \sum_{j=1}^{\gamma}\one\Big(\sum_{i=1}^{q}s_{ij}=1\Big)\geq 2\ell-r,
\end{equation}
since if we let $x$ to be the left hand side, then we must have $\norm{s}_1=r\geq x+2(\ell-x)$. Note that when $ \sum_{j=1}^{\gamma}\one(\sum_{i=1}^{q}s_{ij}=1)=1$, $\left|\E\prod_{i=1}^{q}(Y_{i1}-1/q)^{s_{ij}}\right|\leq \la x_n$ holds by Lemma \ref{lem:Y:moment}. Thus, when $\norm{s}_{*}=\ell, \norm{s}_1=r$ holds with $2\ell>r$ and $s\in \mathcal{S}_{\gamma}(m)$, we have a slightly better bound than \eqref{eq:deleted:mo}: by Lemma \ref{lem:Y:higher:mo} and \eqref{eq:tech}, \rev{we have}
\begin{equation}\label{eq:deleted:mo:2}
    \prod_{j=1}^{\gamma}\E\left|\left[\prod_{i=1}^{q}\left(Y_{i1}-\frac{1}{q}\right)^{s_{ij}}\right]\right|\leq C(q,m)|\la|^{2\ell -r} x_n^{m}\quad\textnormal{for any $\gamma$ and $s\in \mathcal{S}_{\gamma}(m)$ with $2\ell>r$.}
\end{equation}
Having \eqref{eq:deleted:mo:2} in mind, we split the right hand side of \eqref{eq:deleted:mo:ineq} by
\begin{equation}\label{eq:Deltas}
    \begin{split}
         &\left|\E f(Z_1,...,Z_q)-\E\Phi_m \bigg(\Big(Y_{ij}-\frac{1}{q}\Big)_{i\leq q, j\leq \gamma}\bigg)\right|\leq \Delta_1+\Delta_2\quad\textnormal{where}\\
         &\Delta_1:=
         \E\sum_{\ell=0}^{\gamma}\sum_{r= 2\ell}^{ \deg(f)\cdot \ell}\sum_{\substack{s\in \mathcal{S}_{\gamma}:\\\  \norm{s}_{*}=\ell, ~\norm{s}_1=r}}|b_s|\cdot \prod_{j=1}^{\gamma}\left|\E\left[\prod_{i=1}^{q}\left(Y_{i1}-\frac{1}{q}\right)^{s_{ij}}\right]\right|;\\
         &\Delta_2:=
         \E\sum_{\ell=0}^{\gamma}\sum_{r=0}^{2\ell}\sum_{\substack{s\in \mathcal{S}_{\gamma}:\\\  \norm{s}_{*}=\ell, ~\norm{s}_1=r}}|b_s|\cdot \prod_{j=1}^{\gamma}\left|\E\left[\prod_{i=1}^{q}\left(Y_{i1}-\frac{1}{q}\right)^{s_{ij}}\right]\right|.
    \end{split}
\end{equation}
The range $r\leq \deg(f)\cdot\ell$ in the definition of $\Delta_1$ holds since for $s\in \mathcal{S}_{\gamma}$, there exists $k\in \mathcal{K}$ such that $s_{ij}\leq k_i$ for all $1\leq i\leq q$ and $1\leq j\leq \gamma$, so 
\begin{equation*}
    \frac{r}{\ell}\equiv \frac{\norm{s}_1}{\norm{s}_*}\leq \max_{1\leq j \leq \gamma}\sum_{i=1}^{q} s_{ij}\leq \sum_{i=1}^{q}k_i\leq \deg(f).
\end{equation*}
The first component $\Delta_1$ can be bounded as follows. By the bound \eqref{eq:claim} on the coefficients $b_s$ and the moment bound \eqref{eq:deleted:mo} for the deleted $s\in S_{\gamma}(m)$, 
\begin{equation*}
    \Delta_1
    \leq C(q,m)\cdot \norm{f}_1 \cdot x_n^m \cdot \E\left[\sum_{\ell=0}^{\gamma}\sum_{r= 2\ell}^{ \deg(f)\cdot \ell} \left(|\la|q\right)^{r}\cdot 2^{\deg(f)\cdot \ell}\cdot \binom{\gamma}{\ell}\right].
\end{equation*}
Note that the sum above can be bounded by
\begin{equation}\label{eq:sum:bound:tech}
   \sum_{\ell=0}^{\gamma}\sum_{r= 2\ell}^{ \deg(f)\cdot \ell} \left(|\la|q\right)^{r}\cdot 2^{\deg(f)\cdot \ell}\cdot \binom{\gamma}{\ell}\leq  \frac{q}{q-1}\sum_{\ell=0}^{\gamma} \frac{\gamma^{\ell}}{\ell!}\cdot \la^{2\ell}\cdot (2q)^{\deg(f)\cdot \ell}\leq \frac{q}{q-1} \exp\left((2q)^{\deg(f)}\la^2 \gamma \right),
\end{equation}
where we used $\binom{\gamma}{\ell}\leq \frac{\gamma^{\ell}}{\ell!}$, $|\la|\leq 1$, and $\sum_{r=2\ell}^{\deg(f)\cdot \ell}q^r \leq \frac{q}{q-1} q^{\deg(f)\cdot \ell}$ in the first inequality. Thus, recalling the function $K(\cdot)$ from \eqref{eq:def:K} in Assumption \ref{assumption:unif:tail} and using the bound $d\la^2 \leq 1$, it follows that
\begin{equation}\label{eq:Delta:1}
    \Delta_1
    \leq C q,m)\cdot \norm{f}_1 \cdot \frac{q}{q-1} K\left((2q)^{\deg(f)}\right)\cdot x_n^m\equiv C(q,m,f)x_n^m.
\end{equation}
For the second component, we use the bounds \eqref{eq:claim} and \eqref{eq:deleted:mo:2} to have
\begin{equation*}
    \Delta_2\leq  C(q,m)\cdot \norm{f}_1 \cdot x_n^m \cdot \E\left[\sum_{\ell=0}^{\gamma}\sum_{r=0}^{2\ell} \la^{2\ell} q^{r}\cdot 2^{\deg(f)\cdot \ell}\cdot \binom{\gamma}{\ell}\right].
\end{equation*}
Notice that the extra factor $\la^{2\ell-r}$ in \eqref{eq:deleted:mo:2} as opposed to \eqref{eq:deleted:mo} gave us $\la^{2\ell}$ inside the sum above. Proceeding similarly as in \eqref{eq:sum:bound:tech},
\begin{equation}\label{eq:Delta:2}
\Delta_2\leq C(q,m)\cdot \norm{f}_1 \cdot \frac{q}{q-1}K\left(q^2 2^{\deg(f)}\right)\cdot x_n^m\equiv C^\prime(q,m,f) x_n^m.
\end{equation}
Therefore, the bounds \eqref{eq:Deltas}, \eqref{eq:Delta:1}, and \eqref{eq:Delta:2} conclude the proof.
\end{proof}
\begin{remark}\label{rmk:prop:crucical}
Note that in the proof of Proposition \ref{prop:crucial}, we proved a slightly more general result than what's stated in Proposition \ref{prop:crucial}. That is, let $\widetilde{\Psi}_m[\cdot]$ be \textit{any} operator acting on $\R\left[(x_{ij})_{i\leq q, j\leq \gamma}\right]$ that deletes a \textit{subset} of monomials $\prod_{i=1}^{q}\prod_{j=1}^{\gamma}x_{ij}^{s_{ij}}$, where $s=(s_{ij})_{i\leq q, j\leq \gamma}$ satisfies \textit{either} of $(a)$ or $(b)$ in Proposition \ref{prop:crucial}. For $f\in \R[z_1,...,z_q]$, let $\widetilde{\Phi}_{m,f}\equiv \widetilde{\Phi}_m[F_{f}]$. Then, there exist\rev{s} $C=C(q,m,f)$, which only depends on $q$, $m\geq 1$, and $f$, and $\delta=\delta(q)$ such that the following holds: if \rev{$x_{n_0}\leq \delta$ holds for some $n_0\geq 1$, then for any such $\widetilde{\Psi}_m[\cdot]$ and for any $n\geq n_0+5$, $m\geq 1$, and }polynomial $f(z_1,...,z_q)$ in $q$ variables, 
\begin{equation}\label{eq:prop:crucial:tilde}
    \left|\E f(Z_1,...,Z_q)-\E\widetilde{\Phi}_{m,f} \bigg(\Big(Y_{ij}-\frac{1}{q}\Big)_{i\leq q, j\leq \gamma}\bigg)\right|\leq C x_n^m.
\end{equation}
This is because for any $\widetilde{\Psi}_m[\cdot]$, the equation \eqref{eq:deleted:mo:ineq} in the proof of Proposition \ref{prop:crucial} remains true, since the corresponding set of deleted indices $\widetilde{\mathcal{S}}_{\gamma}(m)$ is a subset of $\mathcal{S}_{\gamma}(m)$.
\end{remark}
A major strength of Proposition \ref{prop:crucial} is that we can take advantage of the polynomial factorization of $f$ when computing $\Phi_m[f]$: suppose $f_0=f_1\cdot f_2$ for polynomials $f_1,f_2\in \R[z_1,...,z_q]$ and let $F_i=F_{f_i,\gamma}, i=0,1,2,$ be the corresponding polynomials in $q\times \gamma$ variables according to \eqref{eq:def:F}. Then, $F_0=F_1\cdot F_2$ holds, so
\begin{equation}\label{eq:psi:product}
    \Phi_{m,f}\equiv \Psi_m[F_1\cdot F_2]=\Psi_m\left[\Psi_m[F_1]\cdot\Psi_m[F_2]\right],
\end{equation}
where the last equality holds because if a monomial $\prod_{i=1}^{q}\prod_{j=1}^{\gamma} x_{ij}^{s_{ij}}$ satisfies either $(a)$ or $(b)$ in Proposition \ref{prop:crucial}, then its product with any other monomial also satisfies either $(a)$ or $(b)$. Observe that computing $\Psi_m\left[\Psi_m[F_1]\cdot\Psi_m[F_2]\right]$ is simpler than computing $\Psi_m[F_1\cdot F_2]$ since many monomials might be \textit{apriori} deleted when computing $\Psi_m[F_1]$ or $\Psi_m[F_2]$.

Furthermore, observe that in the final result of Proposition \ref{prop:crucial}, i.e. equation \eqref{eq:prop:crucial}, we evaluate $\Phi_{m,f}$ at $\left(Y_{ij}-1/q\right)_{i\leq q, j\leq \gamma}$. Since $\sum_{i=1}^{q}\left(Y_{ij}-1/q\right)=0$ for any $1\leq j\leq \gamma$ a.s., we can take advantage of this when computing $\E\Psi_{mf}\left(Y_{ij}-1/q\right)_{i\leq q, j\leq \gamma}$. These observations lead to the following generalization.
\begin{prop}\label{prop:crucial:product}
Given $m\geq 1,\gamma\geq 0$, define the operator $\Psi^\star_m[\cdot]$ acting on the space $\R\left[(x_{ij})_{i\leq q, j\leq \gamma}\right]$ as follows. Given $F\in \R\left[(x_{ij})_{i\leq q, j\leq \gamma}\right]$, recall the polynomial $\Psi_m[F]$ from Proposition \ref{prop:crucial}. Define $\Psi^\star_m[F]$ as follows.
\begin{enumerate}[label=(\alph*),start=3]
    \item  Let $a_{ij}$ be the coefficient of $x_{ij}$ in $\Psi_m[F]$. For any $1\leq j\leq \gamma$, if $(a_{ij})_{i\leq q}$ are all the same, delete the monomials $(x_{ij})_{i\leq q}$ from $\Psi_m[F]$. That is, let
    \begin{equation*}
    \Psi^\star_m[F]=\Psi_m[F]-\sumj \one\left(a_{ij}=a_{i^\prime j},~1\leq i,i^\prime\leq q \right)\sum_i a_{ij}x_{ij}.
    \end{equation*}
\end{enumerate}
For a positive integer $L$ and a set of polynomials $(f_{\ell})_{1\leq \ell \leq L}\in \R[z_1,..,z_q]$, define $\Phi^\star_{m,(f_{\ell})_{\ell\leq L}}\in \R[(x_{ij})_{i\leq q, j\leq \gamma}]$ by
\begin{equation}\label{eq:def:Phi:star}
    \Phi^\star_{m,(f_{\ell})_{\ell\leq L}}:=\Psi_m\left[\prod_{\ell=1}^{L}\Psi^\star_m[F_{f_{\ell}}]\right],
\end{equation}
where $F_{f_{\ell}}\equiv F_{f_{\ell}, \gamma}\in \R[(x_{ij})_{i\leq q, j\leq \gamma}]$ is from \eqref{eq:def:F}. Then, there exist $\delta=\delta(q)>0$ and $C=C\left(q,m,(f_{\ell})_{\ell \leq L}\right)$, which only depends on $q$, $m\geq 1$, and polynomials $(f_{\ell})_{\ell \leq L}\in \R[z_1,...,z_q]$, such that the following holds: if \rev{$x_{n_0}\leq \delta$ holds for some $n_0\geq 1$, then for all $n\geq n_0+5$,}
\begin{equation*}
    \left|\E \prod_{\ell=1}^{L}f_{\ell}(Z_1,...,Z_q)-\Phi^\star_{m,(f_{\ell})_{\ell\leq L}}\bigg(\Big(Y_{ij}-\frac{1}{q}\Big)_{i\leq q, j\leq \gamma}\bigg)\right|\leq C x_n^m.
\end{equation*}
\end{prop}
\begin{proof}
Denote $f=\prod_{\ell\leq L}f_{\ell}$. Then, $F_{f}=\prod_{\ell \leq L}F_{f_{\ell}}$ by definition of $F_{f}$ in \eqref{eq:def:F}, and by \eqref{eq:psi:product}, we have
\begin{equation}\label{eq:psi:product:2}
    \Phi_{m,f}=\Psi_m\left[\prod_{\ell=1}^{L}\Psi_m[F_{f_{\ell}}]\right].
\end{equation}
Thus, by Proposition \ref{prop:crucial}, we have that when \rev{$x_{n_0}\leq \delta$ and $n\geq n_0+5$}, 
\begin{equation*}
    \left|\E f(Z_1,...,Z_q)-\E\Psi_m\left[\prod_{\ell=1}^{L}\Psi_m[F_{f_{\ell}}]\right] \bigg(\Big(Y_{ij}-\frac{1}{q}\Big)_{i\leq q, j\leq \gamma}\bigg)\right|\leq C x_n^m.
\end{equation*}
Having Remark \ref{rmk:prop:crucical} in mind, we now aim to find a $\widetilde{\Psi}_m$, which deletes a \textit{subset} of monomials whose degrees satisfy either $(a)$ or $(b)$ in Proposition \ref{prop:crucial} \textit{and} induces $\Phi^\star_{m,(f_{\ell})_{\ell\leq L}}$. Note that by definition of $\Psi^\star_m[\cdot]$, there exist  $(a_{j,f_{\ell}})_{j\leq \gamma}\in \R$ such that for $1\leq \ell\leq L$, 
\begin{equation}\label{eq:psi:star:express}
\Psi^\star_m[F_{f_{\ell}}]=\Psi_m[F_{f_{\ell}}]-\sum_{j=1}^{\gamma}a_{j,f_{\ell}}\sum_{i=1}^{q} x_{ij},
\end{equation}
and if $a_{j,f_{\ell}}\neq 0$ for some $j$ and $\ell$, then $\Psi^\star_m[F_{f_{\ell}}]$ does not contain monomials $x_{ij}$'s.

Now, let $\widetilde{\Psi}_m[F_f]$ be the polynomial which deletes the monomials from $\prod_{\ell=1}^{L}\Psi_m[F_{f_{\ell}}]$ which are included in $\prod_{\ell=1}^{L}\Psi^\star_m[F_{f_{\ell}}]$ \textit{and} whose degrees satisfy either $(a)$ or $(b)$ in Proposition \ref{prop:crucial}. That is, define
\begin{equation}\label{eq:new:Phi}
\widetilde{\Phi}_{m,f}\equiv \widetilde{\Psi}_m[F_f]:=\Psi_m\left[\prod_{\ell=1}^{L}\Psi^\star_m[F_{f_{\ell}}]\right]+\prod_{\ell=1}^{L}\Psi_m[F_{f_{\ell}}]-\prod_{\ell=1}^{L}\Psi^\star_m[F_{f_{\ell}}].
\end{equation}
Then, because of \eqref{eq:psi:product:2}, $\tilde{\Phi}_{m,f}$ is the addition of $\Phi_{m,f}$ and monomials whose degrees satisfies $(a)$ or $(b)$ in Proposition \ref{prop:crucial}. Thus, by Remark \ref{rmk:prop:crucical}, we have
\begin{equation}\label{eq:prop:crucial:product:1}
    \left|\E f(Z_1,...,Z_q)-\E\widetilde{\Psi}_m[F_{f}] \bigg(\Big(Y_{ij}-\frac{1}{q}\Big)_{i\leq q, j\leq \gamma}\bigg)\right|\leq C x_n^m.
\end{equation}
Note that by \eqref{eq:psi:star:express}, the term $\prod_{\ell=1}^{L}\Psi_m[F_{f_{\ell}}]-\prod_{\ell=1}^{L}\Psi^\star_m[F_{f_{\ell}}]$ in \eqref{eq:new:Phi} is a linear combination of a multiple of $\sum_{i=1}^{q}x_{ij}$. Hence, since $\sum_{i=1}^{q}(Y_{ij}-1/q)=0$ a.s., we have
\begin{equation}\label{eq:prop:crucial:product:2}
    \E\widetilde{\Psi}_m[F_{f}] \bigg(\Big(Y_{ij}-\frac{1}{q}\Big)_{i\leq q, j\leq \gamma}\bigg)=\E\Psi_m\left[\prod_{\ell=1}^{L}\Psi^\star_m[F_{f_{\ell}}]\right]\bigg(\Big(Y_{ij}-\frac{1}{q}\Big)_{i\leq q, j\leq \gamma}\bigg).
\end{equation}
Therefore, \eqref{eq:prop:crucial:product:1} and \eqref{eq:prop:crucial:product:2} concludes the proof.
\end{proof}

The strength of Proposition \ref{prop:crucial:product} is demonstrated by the following corollary.
\begin{cor}\label{cor:4th:power}
There exist $\delta=\delta(q)$ and $C=C_q>0$ such that the following holds: if \rev{$x_{n_0}\leq \delta$ holds for some $n_0\geq 1$, then for all $n\geq n_0+5$,}
\begin{equation*}
\E\left(\sumz -q\right)^4\leq C_q x_n^4.
\end{equation*}
\end{cor}
\begin{proof}
Let $f_0(z_1,...,z_q)=\sum_{i=1}^{q}z_i-q$ and denote $F_0\equiv F_{f_0}$ corresponding to $f_0$ via \eqref{eq:def:F}. Then, we have
\begin{equation*}
    F_0\left((x_{ij})_{i\leq q , j\leq \gamma}\right)=\sum_{\ell=1}^{\gamma}(\la q)^{\ell}\sum_{i=1}^{q}~\sum_{1\leq j_1<...<j_{\ell}\leq \gamma}~\prod_{t=1}^{\ell}x_{ij_t}.
\end{equation*}
We now compute $\Psi^\star_4[F_0]$: note that the monomials corresponding to $\ell \geq 4$ in the outer sum above satisfies $(a)$ in Proposition \ref{prop:crucial}, so they are deleted. Similarly, $\ell=1$ in the outer sum above is also deleted because of $(c)$ in Proposition \ref{prop:crucial:product}. Thus,
\begin{equation}\label{eq:f:zero:psi:star}
    \Psi^\star_4[F_0]=(\la q)^2 \sum_{i=1}^{q} \sum_{1\leq j_1<j_2\leq \gamma}x_{ij_1}x_{ij_2}+(\la q)^3 \sum_{i=1}^{q}\sum_{1\leq j_1<j_2<j_3\leq \gamma}x_{ij_1}x_{ij_2}x_{ij_3}.
\end{equation}
Let $f(z_1,...,z_q):=(\sum_{i=1}^{q}z_i-q)^4=(f_0)^4$. Then, by definition of $\Phi^\star_{4,f}$ (cf. \eqref{eq:def:Phi:star}) and \eqref{eq:f:zero:psi:star}, it follows that
\begin{equation}\label{eq:cor:4th:power:1}
    \Phi^\star_{4,f}=\Psi_4\left[\left((\la q)^2 \sum_{i=1}^{q}\sum_{1\leq j_1<j_2\leq \gamma}x_{ij_1}x_{ij_2}+(\la q)^3 \sum_{i=1}^{q}\sum_{1\leq j_1<j_2<j_3\leq \gamma}x_{ij_1}x_{ij_2}x_{ij_3}\right)^4\right]=0,
\end{equation}
since it is straightforward to check that every monomial of the polynomial inside $\Psi_4$ above satisfies $(a)$ or $(b)$ in Proposition \ref{prop:crucial} with $m=4$. Therefore, Proposition \ref{prop:crucial:product} and \eqref{eq:cor:4th:power:1} concludes the proof.
\end{proof}
We are now ready to estimate $\Xi_2$ up to the third order with respect to $x_n$.
\begin{lemma}\label{lem:xi2}
There exist $\delta=\delta(q)$ and $C=C_q>0$ such that the following holds: if \rev{$x_{n_0}\leq \delta$ holds for some $n_0\geq 1$, then for all $n\geq n_0+5$,}
\begin{equation*}
\left|\Xi_2-\left(\frac{3q^3+9q^2}{(q-1)^2}\E\binom{\gamma}{3}\la^6 x_n^3-2q^2 \E\binom{\gamma}{2}\la^6 x_n w_n\right)\right|\leq C_q x_n^4.
\end{equation*}
\end{lemma}
\begin{proof}
Throughout, we assume \rev{$x_{n_0}\leq \delta$ and $n\geq n_0+5$} so that the previous results hold. Since $0\leq \frac{Z_1}{\sumz}\leq 1$ in the definition of $\Xi_2$ in \eqref{eq:def:xi}, we have by Corollary \ref{cor:4th:power} that
\begin{equation}\label{eq:xi2:alternate}
    \Xi_2 = \frac{1}{q^4}\E\left(2q Z_1-\sum_{i=1}^{q}Z_1 Z_i-q\right)\left(\sumz -q\right)^2+O_q(x_n^4).
\end{equation}
We now aim to approximate the expectation of the right hand side above using Proposition \ref{prop:crucial:product}. To this end, let 
\begin{equation*}
    f(z_1,...,z_q)=\left(2q z_1-\sum_{i=1}^{q}z_1 z_i-q\right)\left(\sum_{i=1}^{q}z_i-q\right)^2\equiv f_1(z_1,...,z_q)\left(f_0(z_1,...,z_q)\right)^2,
\end{equation*}
where $f_0(z_1,...,z_q)=\sum_{i=1}^{q}z_i-q$ as before, and $f_1(z_1,...,z_q):=2q z_1-\sum_{i=1}^{q}z_1 z_i-q$. Also, denote $F_{\ell}\equiv F_{f_{\ell}}$ for $\ell=0,1$, which corresponds to $f_{\ell}$ via \eqref{eq:def:F}. Then, by definition of $\Phi^\star_{4,f}$ in \eqref{eq:def:Phi:star}, we have
\begin{equation}\label{eq:lem:xi2:tech:1}
    \Phi^\star_{4,f}=\Psi_4\left[\Psi^\star_{4}[F_1]\cdot\left(\Psi^\star_{4}[F_0]\right)^2\right]=\Psi_4\left[\Psi^\star_{4}[F_1]\cdot\left(\sum_{\ell=2}^{3}(\la q)^{\ell}\sum_{i=1}^{q}~\sum_{1\leq j_1<...<j_{\ell}\leq \gamma}~\prod_{t=1}^{\ell}x_{ij_t}\right)^2\right],
\end{equation}
where the last equality is due to \eqref{eq:f:zero:psi:star}. An important observation regarding \eqref{eq:lem:xi2:tech:1} is that the multiplication of $\left(\sum_{\ell=2}^{3}(\la q)^{\ell}\sum_{i=1}^{q}~\sum_{1\leq j_1<...<j_{\ell}\leq \gamma}~\prod_{t=1}^{\ell}x_{ij_t}\right)^2$ with the monomials $\prod_{i=1}^{q}\prod_{j=1}^{\gamma}x_{ij}^{s_{ij}}$ which satisfies $\norm{s}_{\star}\equiv \sum_{j=1}^{\gamma}\one\left(\sum_{i=1}^{q}s_{ij}\geq 1\right)\geq 3$ will be deleted after applying $\Psi_4$. Therefore, we calculate $\Psi^\star_4[F_1]$ as
\begin{equation}\label{eq:lem:xi2:tech:2}
    \Psi^\star_4[F_1]=\la q^2\sum_{j=1}^{\ga}x_{1j}-\la q\sumi \sumj x_{ij}+(\la q)^2\sum_{1\leq j_1<j_2\leq \ga}\left(q x_{1j_1}x_{1j_2}-\sum_{i=1}^{q}x_{ij_1}x_{ij_2}\right)+R,
\end{equation}
where $R\in \R\left[(x_{ij})_{i\leq q, j\leq \gamma}\right]$ contains monomials $\prod_{i=1}^{q}\prod_{j=1}^{\gamma}x_{ij}^{s_{ij}}$ with $\norm{s}_{*}=3$. 

\rev{Let} $G\in \R\left[(x_{ij})_{i\leq q, j\leq \gamma}\right]$ \rev{be}
\begin{equation}\label{eq:def:G}
\begin{split}
    G
    &:=
    \Psi_4\Bigg[(\la q)^2\left(\sum_{1\leq j_1<j_2\leq \ga}\left(q x_{1j_1}x_{1j_2}-\sum_{i=1}^{q}x_{ij_1}x_{ij_2}\right)\right)\left(\sum_{\ell=2}^{3}(\la q)^{\ell}\sum_{i=1}^{q}~\sum_{1\leq j_1<...<j_{\ell}\leq \gamma}~\prod_{t=1}^{\ell}x_{ij_t}\right)^2\Bigg]\\
    &=(\la q)^6\sum_{\substack{(j_1,j_2,j_3)\in [\gamma]^3:\\ \textnormal{ distinct}}}\left(q x_{1j_1}x_{1j_2}-\sum_{i=1}^{q}x_{ij_1}x_{ij_2}\right)\left(\sum_{i=1}^{q}x_{ij_1}x_{ij_3}\right)\left(\sum_{i=1}^{q}x_{ij_2}x_{ij_3}\right),
\end{split}
\end{equation}
where the last equality is due to the definition of $\Psi_4$ in Proposition \ref{prop:crucial}. Similarly, we can compute the term $\Psi_4\left[\left( \sumj x_{1j}\right)\cdot\left(\Psi^\star_{4}[F_0]\right)^2\right]$ in the equations \eqref{eq:lem:xi2:tech:1} and \eqref{eq:lem:xi2:tech:2}, and then use Lemma \ref{lem:Y:indep} to have that
\begin{equation}\label{eq:compute:expectation:Phi}
\begin{split}
    &\E \Phi^\star_{4,f}\bigg(\Big(Y_{ij}-\frac{1}{q}\Big)_{i\leq q, j\leq \gamma}\bigg)-\E G\bigg(\Big(Y_{ij}-\frac{1}{q}\Big)_{i\leq q, j\leq \gamma}\bigg)\\
    &=2\la^5 q^6 \E\binom{\gamma}{2}\sum_{i_1, i_2=1}^{q} \E\left(Y_{11}-\frac{1}{q}\right)\left(Y_{i_1 1}-\frac{1}{q}\right)\left(Y_{i_2 1}-\frac{1}{q}\right)\cdot \E\left(Y_{i_1 1}-\frac{1}{q}\right)\left(Y_{i_2 1}-\frac{1}{q}\right)\\
    &\quad+12 \la^5 q^6 \E\binom{\gamma}{3}\sum_{i_1, i_2=1}^{q}\E\left(Y_{11}-\frac{1}{q}\right)\left(Y_{i_1 1}-\frac{1}{q}\right)\cdot\E\left(Y_{i_1 1}-\frac{1}{q}\right)\left(Y_{i_2 1}-\frac{1}{q}\right)\cdot \E\left(Y_{i_2 1}-\frac{1}{q}\right)\\
    &\quad+3\la^5 q^6 \E\binom{\gamma}{3}\E\left(Y_{11}-\frac{1}{q}\right)\sum_{i_1, i_2=1}^{q}\left(\E\left(Y_{i_1 1}-\frac{1}{q}\right)\left(Y_{i_2 1}-\frac{1}{q}\right)\right)^2\\
    &\quad+6\la^6 q^6 \E\binom{\gamma}{3}\sum_{i_1, i_2=1}^{q}\E\left(Y_{11}-\frac{1}{q}\right)\left(Y_{i_1 1}-\frac{1}{q}\right)\cdot\left(\E\left(Y_{i_1 1}-\frac{1}{q}\right)\left(Y_{i_2 1}-\frac{1}{q}\right)\right)^2.
\end{split}
\end{equation}
where we used the fact that $\sumi (Y_{ij}-1/q)=0$ a.s., so that we can neglect the terms coming from $\Psi_4\left[\left(\sumi \sumj x_{ij}\right)\cdot\left(\Psi^\star_{4}[F_0]\right)^2\right]$.

We now compute each sums in \eqref{eq:compute:expectation:Phi} separately: by Lemma \ref{lem:Y:moment}, the first sum equals
\begin{equation}\label{eq:compute:expectation:Phi:1}
\begin{split}
    &\sum_{i_1, i_2=1}^{q} \E\left(Y_{11}-\frac{1}{q}\right)\left(Y_{i_1 1}-\frac{1}{q}\right)\left(Y_{i_2 1}-\frac{1}{q}\right)\cdot \E\left(Y_{i_1 1}-\frac{1}{q}\right)\left(Y_{i_2 1}-\frac{1}{q}\right)\\
    &=\left(\la v_n+\frac{1-\la}{q} u_n\right)\cdot\frac{x_n}{q}+2(q-1)\left(-\frac{\la}{q-1}v_n-\frac{1-\la}{q(q-1)}u_n\right)\cdot\left(-\frac{x_n}{q(q-1)}\right)\\
    &\quad~~+(q-1)\left(-\frac{\la}{q-1}v_n-\frac{1-\la}{q(q-1)}u_n-\la w_n\right)\cdot\frac{x_n}{q}\\
    &\quad~~+(q-1)(q-2)\left(\frac{2\la}{(q-1)(q-2)}v_n+\frac{2(1-\la}{q(q-1)(q-2)}u_n+\frac{\la}{q-2}w_n\right)\cdot\left(-\frac{x_n}{q(q-1)}\right)\\
    &\quad~~+O_q(x_n^4)\\
    &=-\la x_n w_n+O_q(x_n^4).
\end{split}
\end{equation}
In the first equality above, we used the fact that for $i_1\neq i_2\in \{1,...,q\}$,
\begin{equation}\label{eq:Y:sec:mo:approx}
    \E\left(Y_{i_1 1}-\frac{1}{q}\right)^2=\frac{x_n}{q}+O_q(x_n^2),~~\E\left(Y_{i_1 1}-\frac{1}{q}\right)\left(Y_{i_2 1}-\frac{1}{q}\right)=-\frac{x_n}{q(q-1)}+O_q(x_n^2),
\end{equation}
which follows from Lemma \ref{lem:Y:moment} and Lemma \ref{lem:u:v:w:apriori}.
Similarly, we can calculate the rest of the terms in \eqref{eq:compute:expectation:Phi} by Lemma \ref{lem:Y:moment} and Lemma \ref{lem:u:v:w:apriori}, which gives
\begin{equation}\label{eq:compute:expectation:Phi:2}
\begin{split}
    &\sum_{i_1, i_2=1}^{q}\E\left(Y_{11}-\frac{1}{q}\right)\left(Y_{i_1 1}-\frac{1}{q}\right)\cdot\E\left(Y_{i_1 1}-\frac{1}{q}\right)\left(Y_{i_2 1}-\frac{1}{q}\right)\cdot \E\left(Y_{i_2 1}-\frac{1}{q}\right)\\
    &=\frac{\la}{(q-1)^2}x_n^3+O_q(x_n)^4,\\
    &\E\left(Y_{11}-\frac{1}{q}\right)\sum_{i_1, i_2=1}^{q}\left(\E\left(Y_{i_1 1}-\frac{1}{q}\right)\left(Y_{i_2 1}-\frac{1}{q}\right)\right)^2=\frac{\la}{q-1}x_n^3+O_q(x_n^4),\\
     &\sum_{i_1, i_2=1}^{q}\E\left(Y_{11}-\frac{1}{q}\right)\left(Y_{i_1 1}-\frac{1}{q}\right)\cdot\left(\E\left(Y_{i_1 1}-\frac{1}{q}\right)\left(Y_{i_2 1}-\frac{1}{q}\right)\right)^2=O_q(x_n^4)
    \end{split}
\end{equation}
Thus, plugging in \eqref{eq:compute:expectation:Phi:1} and  \eqref{eq:compute:expectation:Phi:2} into \eqref{eq:compute:expectation:Phi} shows that $\E \Phi^\star_{4,f}\Big(\big(Y_{ij}-\frac{1}{q}\big)_{i\leq q, j\leq \gamma}\Big)$ equals
\begin{equation*}
   \E G\bigg(\Big(Y_{ij}-\frac{1}{q}\Big)_{i\leq q, j\leq \gamma}\bigg)-2q^6 \E\binom{\gamma}{2}\la^6 x_n w_n+\frac{q^6(3q+9)}{(q-1)^2}\E \binom{\gamma}{3} \la^6 x_n^3+O_q(x_n^4),
\end{equation*}
where we used $\E\binom{\gamma}{3}\la^6\leq \E \exp(\gamma \la^2)\lesssim 1$ (cf. Assumption \ref{assumption:unif:tail}) for the last term in the right hand side of \eqref{eq:compute:expectation:Phi}. Finally, we have that $\E G\Big(\big(Y_{ij}-\frac{1}{q}\big)_{i\leq q, j\leq \gamma}\Big)$ equals
\begin{equation*}
\begin{split}
&\E\binom{\ga}{3}(\la q)^6 \bigg(q\sum_{i_1, i_2=1}^{q}\E\left(Y_{11}-\frac{1}{q}\right)\left(Y_{i_1 1}-\frac{1}{q}\right)\cdot\E\left(Y_{1 1}-\frac{1}{q}\right)\left(Y_{i_1 1}-\frac{1}{q}\right)\cdot \E\left(Y_{i_1 1}-\frac{1}{q}\right)\left(Y_{i_2 1}-\frac{1}{q}\right)\\
&\quad\quad-\sum_{i_1, i_2,i_3=1}^{q}\E\left(Y_{i_3 1}-\frac{1}{q}\right)\left(Y_{i_1 1}-\frac{1}{q}\right)\cdot\E\left(Y_{i_3 1}-\frac{1}{q}\right)\left(Y_{i_1 1}-\frac{1}{q}\right)\cdot \E\left(Y_{i_1 1}-\frac{1}{q}\right)\left(Y_{i_2 1}-\frac{1}{q}\right)\bigg)\\
&=O_q(x_n^4),
\end{split}
\end{equation*}
where we used $\E\binom{\gamma}{3}\la^6\lesssim 1$ and \eqref{eq:Y:sec:mo:approx} in the last equality. Therefore, recalling \eqref{eq:xi2:alternate}, we have by Proposition \ref{prop:crucial:product} that
\begin{equation*}
\begin{split}
    \Xi_2
    &=\frac{1}{q^4} \E \Phi^\star_{4,f}\bigg(\Big(Y_{ij}-\frac{1}{q}\Big)_{i\leq q, j\leq \gamma}\bigg)+O_q(x_n^4)\\
    &=\frac{3q^3+9q^2}{(q-1)^2}\E\binom{\gamma}{3}\la^6 x_n^3-2q^2 \E\binom{\gamma}{2}\la^6 x_n w_n+O_q(x_n^4).
\end{split}
\end{equation*}
\end{proof}

\subsection{Finer estimates of $u_n, w_n$}
Note that in the third order approximation of $\Xi_1$ (resp. $\Xi_2$) in Lemma \ref{lem:xi1} (resp. Lemma \ref{lem:xi2}), the term $x_n u_n$ (resp. $x_n w_n$) appears. Thus, on top of the \textit{apriori} estimates of $|u_n|\vee |w_n|$ in Lemma \ref{lem:u:v:w:apriori}, we need an \textit{exact} second order approximation of $u_n$ and $w_n$ in order to obtain the exact third order approximation of $x_{n+1}$. In doing so, the following two lemmas which follow from Proposition \ref{prop:crucial:product} are useful throughout.
\begin{lemma}\label{lem:useful}
Let $L$ and $M$ be positive integers such that $L+2M \geq 5$. Let $(f_{\ell})_{1\leq \ell\leq L}\in \R[z_1,...,z_q]$ be polynomials such that
\begin{equation}\label{eq:f:no:constant}
    f_{\ell}(1,1,...,1)=0,\quad\textnormal{for all}\quad 1\leq \ell \leq L.
\end{equation}
Then, there exist $\delta=\delta(q)>0$ and $C=C\left(q,M,(f_{\ell})_{\ell \leq L}\right)$, which only depends on $q,M$, and polynomials $(f_{\ell})_{\ell \leq L}$, such that the following holds: if \rev{$x_{n_0}\leq \delta$ holds for some $n_0\geq 1$, then for all $n\geq n_0+5$}
\begin{equation}\label{eq:lem:useful}
    \left|\E\prod_{\ell=1}^{L}f_{\ell}(Z_1,...,Z_q)\Big(\sumz -q\Big)^{M}\right|\leq C x_n^3.
\end{equation}
\end{lemma}
\begin{proof}
Denote $f=\left(\prod_{\ell=1}^{L}f_{\ell}\right)\cdot(f_0)^{M}\in \R[z_1,...,z_q]$, where $f_0=\sumi z_i-q$. Let $F_0=F_{f_0}\in \R\left[(x_{ij})_{i\leq q, j\leq \ga}\right]$ be the polynomial corresponding to $f_0$ via \eqref{eq:def:F}. Then, analogous to $\Psi^\star_4[F_0]$ in \eqref{eq:f:zero:psi:star}, we have
\begin{equation}\label{eq:Psi:star:zero:3}
    \Psi^\star_3[F_0]=(\la q)^2 \sum_{i=1}^{q} \sum_{1\leq j_1<j_2\leq \gamma}x_{ij_1}x_{ij_2}.
\end{equation}
Now, note that $\Psi_3$ defined in Proposition \ref{prop:crucial} deletes any monomials $\prod_{i=1}^{q}\prod_{j=1}^{\ga}x_{ij}^{s_{ij}}$ where $s=(s_{ij})_{i\leq q, j\leq \ga}$ satisfies either of the following.
\begin{itemize}
    \item $\norm{s}_{*}\equiv \sum_{j=1}^{\gamma}\one\left(\sum_{i=1}^{q}s_{ij}\geq 1 \right)\geq 3$ or \item $\norm{s}_{*}=2$ and $\norm{s}_1\geq 5$.
\end{itemize}
For $\ell\leq L$, the polynomial $F_{f_{\ell}}\in \R\left[(x_{ij})_{i\leq q, j\leq \ga}\right]$ satisfies $F_{f_{\ell}}(0,0,....,0)=0$ since $f_{\ell}(1,1,...,1)=0$ holds. That is, $F_{f_{\ell}}$ does not have a constant term for $\ell \leq L$. Thus, it is straightforward to check with the expression of $\Psi^\star_3[F_0]$ above that when $L+2M\geq 5$, all the monomials $\prod_{i=1}^{q}\prod_{j=1}^{\ga}x_{ij}^{s_{ij}}$ in $\prod_{\ell=1}^{L}\Psi^\star_3[F_{f_{\ell}}]\cdot\left(\Psi^\star_3[F_0]\right)^M$ satisfy either $\norm{s}_{*}\geq 3$ or $\norm{s}_{*}=2$ and $\norm{s}_1=5$. Hence, it follows that
\begin{equation*}
    \Phi^\star_{3,f}=\Psi_3\left[\prod_{\ell=1}^{L}\Psi^\star_3[F_{f_{\ell}}]\cdot \left(\Psi^\star_3[F_0]\right)^3\right]=0.
\end{equation*}
Therefore, Proposition \ref{prop:crucial:product} concludes the proof.
\end{proof}
\begin{lemma}\label{lem:useful:2}
There exist $\delta=\delta(q)$ and $C=C_q>0$ such that the following holds: if \rev{$x_{n_0}\leq \delta$ holds for some $n_0\geq 1$, then for all $n\geq n_0+5$,}
\begin{equation*}
    \left|\E\Big(\sumz-q\Big)(Z_1+Z_2-2)(Z_1-Z_2)\right|\leq C_q x_n^3
\end{equation*}
\end{lemma}
\begin{proof}
Let $f=f_0\cdot f_1\cdot f_2\in \R[z_1,...,z_q]$, where $f_0(z_1,...,z_q)=\sumi z_i -q$ as before, $f_1(z_1,...,z_q)=z_1+z_2-2$, and $f_2(z_1,...,z_q)=z_1-z_2$. Recalling the expression of $\Psi^\star_3[F_0]$ for $F_0=F_{f_0}$ in \eqref{eq:Psi:star:zero:3}, we have that
\begin{equation*}
\begin{split}
    \Phi^\star_{3,f}
    &= (\la q)^2 \cdot \Psi_3\left[\Psi^\star_3[F_{f_1}]\cdot \Psi^\star_3[F_{f_2}]\cdot \left( \sum_{i=1}^{q} \sum_{1\leq j_1<j_2\leq \gamma}x_{ij_1}x_{ij_2}\right)\right]\\
    &=(\la q)^4\sum_{\substack{(j_1,j_2)\in [\gamma]^2:\\ \textnormal{ distinct}}}\Big(x_{1j_1}-x_{2j_1}\Big)\Big(x_{1j_2}+x_{2j_2}\Big)\left(\sumi x_{ij_1}x_{ij_2}\right).
\end{split}
\end{equation*}
We then calculate $\E \Phi^\star_{3,f}\bigg(\Big(Y_{ij}-\frac{1}{q}\Big)_{i\leq q, j\leq \gamma}\bigg)$ by using Lemma \ref{lem:Y:indep} and $\sumi (Y_{ij}-1/q)=0$:
\begin{equation*}
\begin{split}
    &\E \Phi^\star_{3,f}\bigg(\Big(Y_{ij}-\frac{1}{q}\Big)_{i\leq q, j\leq \gamma}\bigg)\\
    &=\E \binom{\ga}{2}(\la q)^4\sumi \left(2\left(\E\left(Y_{11}-\frac{1}{q}\right)\left(Y_{i 1}-\frac{1}{q}\right)\right)^2-2\left(\E\left(Y_{21}-\frac{1}{q}\right)\left(Y_{i 1}-\frac{1}{q}\right)\right)^2\right)\\
    &=O_q(x_n^3),
\end{split}
\end{equation*}
where the last equality follows from \eqref{eq:Y:sec:mo:approx} and $\E\binom{\ga}{2}\la^4\lesssim 1$ (cf. Assumption \ref{assumption:unif:tail}). Thus, Proposition \ref{prop:crucial:product} concludes the proof.
\end{proof}

The two lemmas below are generalizations of Lemma \ref{lem:recurse:first:order}.
\begin{lemma}\label{lem:u:second:order}
There exist $\delta=\delta(q)$ and $C=C_q>0$ such that the following holds: if \rev{$x_{n_0}\leq \delta$ holds for some $n_0\geq 1$, then for all $n\geq n_0+5$,}
\begin{equation}\label{eq:u:iter}
    \left|u_{n+1}-\left(d\la^3 u_n+ \frac{6(q-2)}{q-1}\E\binom{\gamma}{2}\la^4 x_n^2\right)\right|\leq C_q x_n^3.
\end{equation}
\end{lemma}
\begin{proof}
By \eqref{eq:lem:basic:1} in Lemma \ref{lem:basic} and the expression of $X^{+}(n+1)=\frac{Z_1}{\sumz}$ in \eqref{eq:key:recursion}, it follows that 
\begin{equation}\label{lem:u:second:order:tech:1}
\begin{split}
    u_{n+1}
    &=\frac{q-1}{q}\E\left[\left(\frac{Z_1}{\sumz}-\frac{1}{q}\right)^2-\left(\frac{Z_2}{\sumz}-\frac{1}{q}\right)^2\right]\\
    &=\frac{q-1}{q^4}\E\left[\frac{(Z_1-Z_2)\Big((q-2)Z_1+(q-2)Z_2-2\sum_{i=3}^{q}Z_i\Big)}{\left(1+\frac{\sumz-q}{q}\right)^2}\right],
\end{split}
\end{equation}
where we used $a^2-b^2=(a-b)(a+b)$ in the last equality. To this end, define the polynomials
\begin{equation*}
    f_1(z_1,...,z_q)\equiv z_1-z_2,\quad\quad f_2(z_1,...,z_q)\equiv (q-2)z_1+(q-2)z_2-2\sum_{i=3}^{q}z_i,
\end{equation*}
and $f_0(z_1,...,z_q)=\sum_{i=1}^{q}z_i-q$ as before. Using the identity $\frac{1}{(1+x)^2}=1-2x+3x^2-4x^3+\frac{x^4}{(1+x)^2}(6+2x)$ for $x=(\sumz-q)/q$ in the right hand side of \eqref{lem:u:second:order:tech:1}, it follows that
\begin{equation*}
\begin{split}
    &u_{n+1}
    =\frac{q-1}{q^4}\bigg(\E f_1(Z)f_2(Z)-2\E f_0(Z)f_1(Z)f_2(Z)\bigg)+R_1+R_2,\quad\textnormal{where}\\
    &R_1 :=\frac{q-1}{q^4}\bigg(3\E \big(f_0(Z)\big)^2f_1(Z)f_2(Z)-4\E \big(f_0(Z)\big)^3 f_1(Z)f_2(Z)\bigg),\\
    &R_2 := \frac{q-1}{q^2}\E\bigg[\left(\frac{Z_1-Z_2}{\sumz}\right)\left(\frac{(q-2)Z_1+(q-2)Z_2-2\sum_{i=3}^{q}Z_i}{\sumz}\right)\\
    &\qquad\qquad\qquad\qquad\cdot\left(\frac{\sumz-q}{q}\right)^4\left(6+2\frac{\sumz-q}{q}\right)\bigg].
\end{split}
\end{equation*}
In the above, we abbreviated $Z=(Z_1,...,Z_q)$. Note that $f_1$ and $f_2$ satisfy the condition \eqref{eq:f:no:constant}, so Lemma \ref{lem:useful} implies that $|R_1|=O_q(x_n^3)$. Moreover, note that $6+2\frac{\sumz-q}{q}>0$. Thus, we can use the trivial bounds $\left|\frac{Z_1-Z_2}{\sumi Z_i}\right|\leq 1$ and $\left|\frac{(q-2)Z_1+(q-2)Z_2-2\sum_{i=3}^{q}Z_i}{\sumz}\right|\leq q$ to have that
\begin{equation*}
    |R_2|\leq \frac{q-1}{q}\E\left[\left(\frac{\sumz-q}{q}\right)^4\left(6+2\frac{\sumz-q}{q}\right)\right]=O_q(x_n^3),
\end{equation*}
where the last equality holds since $\E(\sumz -q)^{M}=O_q(x_n^3)$ for $M=4,5$ by Lemma \ref{lem:useful}. Thus, it follows that
\begin{equation}\label{lem:u:second:order:tech:2}
    u_{n+1}=\frac{q-1}{q^4}\E f_1(Z)f_2(Z)-\frac{2(q-1)}{q^4}\E f_0(Z)f_1(Z)f_2(Z)+O_q(x_n^3).
\end{equation}
We estimate the first expectation in the right hand side above by Lemma \ref{lem:est:Z:mono} while we use Lemmas \ref{lem:useful} and \ref{lem:useful:2} to show that the second expectation is $O_q(x_n^3)$: by Lemma \ref{lem:Y:moment} and Lemma \ref{lem:est:Z:mono}, we have that $\E f_1(Z)f_2(Z)$ equals
\begin{equation}\label{lem:u:second:order:tech:3}
\begin{split}
    &(q-2)\E Z_1^2-(q-2)\E Z_2^2 -2\sum_{i=3}^{q}\E Z_1Z_i+2\sum_{i=3}^{q}\E Z_2 Z_i\\
    &=\frac{q^4}{q-1}d\la^3 u_n+\E \binom{\ga}{2}\la^4 (q-2)\bigg((3x_n+\la q u_n)^2-\frac{3}{(q-1)^2}\left((q-3)x_n-\la q u_n\right)^2\\
    &\qquad\qquad\qquad\qquad\qquad\qquad\qquad+\frac{2}{(q-1)^2(q-2)^2}\left((3q-6)x_n-2\la q u_n\right)^2\bigg)+O_q(x_n^3)\\
    &=\frac{q^4}{q-1}d\la^3 u_n+\frac{6q^4(q-2)}{(q-1)^2}\E \binom{\ga}{2}\la^4 x_n^2+O_q(x_n^3),
\end{split}
\end{equation}
where we used $|u_n|=O_q(x_n^2)$ from Lemma \ref{lem:u:v:w:apriori} in the last equality.

For the second expectation in \eqref{lem:u:second:order:tech:2}, we can split it as
\begin{equation*}
\E f_0(Z)f_1(Z)f_2(Z)
=q\E(Z_1+Z_2-2)(Z_1-Z_2)\Big(\sumz -q\Big)-2\E(Z_1-Z_2)\Big(\sumz-q\Big)^2.
\end{equation*}
Note that the first piece in the right hand side is $O_q(x_n^3)$ by Lemma \ref{lem:useful:2} and the second piece is also $O_q(x_n^3)$ by Lemma \ref{lem:useful}. Thus, it follows that 
\begin{equation}\label{lem:u:second:order:tech:4}
   \left|\E f_0(Z)f_1(Z)f_2(Z)\right|=O_q(x_n^3).
\end{equation}
Therefore, plugging in \eqref{lem:u:second:order:tech:3} and \eqref{lem:u:second:order:tech:4} into \eqref{lem:u:second:order:tech:2} concludes the proof.
\end{proof}

\begin{lemma}\label{lem:w:second:order}
There exist $\delta=\delta(q)$ and $C=C_q>0$ such that the following holds: if \rev{$x_{n_0}\leq \delta$ holds for some $n_0\geq 1$, then for all $n\geq n_0+5$,}
\begin{equation}\label{eq:w:iter}
    \left|w_{n+1}-\left(d\la^4 w_n- \frac{2(q+1)}{(q-1)^2}\E\binom{\gamma}{2}\la^4 x_n^2\right)\right|\leq C_q x_n^3.
\end{equation}
\end{lemma}
\begin{proof}
By definition of $w_{n+1}$ in \eqref{eq:def:v:w} and the expression of $X^{+}(n+1)=\frac{Z_1}{\sumz}$ in \eqref{eq:key:recursion}, we have
\begin{equation}\label{eq:recur:w:tech}
    w_{n+1}=\E\left(\frac{Z_1}{\sumz}-\frac{1}{q}\right)\left(\frac{Z_2}{\sumz}-\frac{1}{q}\right)\left(\frac{Z_1-Z_2}{\sumz}\right).
\end{equation}
We use the identity $\frac{1}{(1+x)^3}=1-3x+6x^2-10x^3+\frac{x^4}{(1+x)^3}(10x^2+24x+15)$ for $x=\frac{\sumz -q}{q}$ to expand the right hand side. To this end, let $f_0(z_1,...,z_q)=\sumi z_i-q \in \R[z_1,..,z_q]$ and let $f,g \in \R[z_1,...,z_q]$ be
\begin{equation*}
    f=-3 f_0+6 \left(f_0\right)^2-10 \left(f_0\right)^3,\quad\quad g= 15+24 f_0+10 \left(f_0\right)^2.
\end{equation*}
Then, we can expand the right hand side of \eqref{eq:recur:w:tech} by
\begin{equation*}
\begin{split}
    &w_{n+1}
     =\frac{1}{q^3}\E\left(Z_1-\frac{\sumz}{q}\right)\left(Z_2-\frac{\sumz}{q}\right)\Big(Z_1-Z_2\Big)+R_1+R_2,\quad\textnormal{where}\\
     &R_1:=\frac{1}{q^3}\E \left(Z_1-\frac{\sumz}{q}\right)\left(Z_2-\frac{\sumz}{q}\right)\Big(Z_1-Z_2\Big)f(Z),\\
     &R_2:=\E\left(\frac{Z_1}{\sumz}-\frac{1}{q}\right)\left(\frac{Z_2}{\sumz}-\frac{1}{q}\right)\left(\frac{Z_1-Z_2}{\sumz}\right)\left(\frac{\sumz-q}{q}\right)^{4}g(Z)
\end{split}
\end{equation*}
Observe that $\E\left(Z_1-\frac{\sumz}{q}\right)\left(Z_2-\frac{\sumz}{q}\right)\Big(Z_1-Z_2\Big)\left(f_0\right)^{M}= O_q(x_n^3)$ for $M=1,2,3$, by Lemma \ref{lem:useful}. Thus, it follows that $|R_1|=O_q(x_n^3)$. Moreover, $g(Z)>0$ a.s., since $15+24x+10x^2>0$, so we can use the bound $\left|\left(\frac{Z_1}{\sumz}-\frac{1}{q}\right)\left(\frac{Z_2}{\sumz}-\frac{1}{q}\right)\left(\frac{Z_1-Z_2}{\sumz}\right)\right|\leq 1$ to have that
\begin{equation*}
    |R_2|\leq \E\left(\frac{\sumz-q}{q}\right)^{4}g(Z)=O_q(x_n^3),
\end{equation*}
where we used $\E(\sumz-q)^{M}= O_q(x_n^3)$ for $M=4,5,6,$ by Lemma \ref{lem:useful}. Therefore, we have that
\begin{equation*}
    w_{n+1}
    =\frac{1}{q^3}\E\left(Z_1-\frac{\sumz}{q}\right)\left(Z_2-\frac{\sumz}{q}\right)\Big(Z_1-Z_2\Big)+O_q(x_n^3)
\end{equation*}
Note that we can express $Z_i-\frac{\sumz}{q}=Z_i-1-\left(\frac{\sumz -q}{q}\right)$ for $i=1,2$, which gives
\begin{equation}\label{eq:recur:tech:2}
\begin{split}
    w_{n+1}
    &=\frac{1}{q^3}\bigg(\E\Big(Z_1-1\Big)\Big(Z_2-1\Big)\Big(Z_1-Z_2\Big)-\E\Big(Z_1+Z_2-2\Big)\left(\frac{\sumz-q}{q}\right)\Big(Z_1-Z_2\Big)\\
    &\qquad\qquad\qquad\qquad+\E\left(\frac{\sumz-q}{q}\right)^2 \Big(Z_1-Z_2\Big)\bigg)+O_q(x_n^3)\\
    &=\frac{1}{q^3}\E\Big(Z_1-1\Big)\Big(Z_2-1\Big)\Big(Z_1-Z_2\Big)+O_q(x_n^3),
\end{split}
\end{equation}
where we used Lemma \ref{lem:useful} and Lemma \ref{lem:useful:2} in the last equality.

We now es\rev{t}imate $\E(Z_1-1)(Z_2-1)(Z_1-Z_2)$:
\begin{equation}\label{eq:recur:tech:3}
\begin{split}
    \E\Big(Z_1-1\Big)\Big(Z_2-1\Big)\Big(Z_1-Z_2\Big)
    &=\E Z_1^2 Z_2-\E Z_1 Z_2^2 -\E Z_1^2+\E Z_2^2+\E Z_1-\E Z_2\\
    &=d\la^4 q^3 w_n-\frac{2q^3(q+1)}{(q-1)^2}\E \binom{\ga}{2}\la^4 x_n^2 +O_q(x_n^3),
\end{split}
\end{equation}
where the second equality holds because of Lemma \ref{lem:est:Z:mono}, Lemma \ref{lem:Y:moment}, and the estimate $|u_n|\vee|v_n|\vee|w_n|=O_q(x_n^2)$ from Lemma \ref{lem:u:v:w:apriori}. Therefore, plugging in \eqref{eq:recur:tech:3} into \eqref{eq:recur:tech:2} finishes the proof.
\end{proof}
Having Lemmas \ref{lem:u:second:order} and \ref{lem:w:second:order} in hand, we now estimate $u_n$ and $w_n$ solely in terms of $x_n$.
\begin{lemma}\label{lem:exact:u:w}
There exists a constant $C_q>0$ such that the following holds: for any $\eps>0$, there exists a positive constant $\delta(q,\eps)>0$ and a positive integer $N(q,\eps)$, which only depend on $q$ and $\eps$, such that if \rev{$x_{n_0}\leq \delta(q,\eps)$ for some $n_0\geq 1$ and $n$ is large enough so that $n\geq n_0+5$ and $n\geq N(q,\eps)$ holds,} then
\begin{equation}\label{eq:u:exact}
    \left|u_n-\frac{6(q-2)}{q-1}\E \binom{\ga}{2}\frac{\la}{d(d\la-1)} x_n^2\right|\leq C_q x_n^3+\eps x_n^2,
\end{equation}
and
\begin{equation}\label{eq:w:exact}
    \left|w_n+\frac{2(q+1)}{(q-1)^2}\E \binom{\ga}{2}\frac{1}{d(d-1)} x_n^2\right|\leq C_q x_n^3+\eps x_n^2
\end{equation}
\end{lemma}
\begin{proof}
Throughout, we fix $\eps >0$. To start with, we have by Lemma \ref{lem:recurse:first:order} that
\begin{equation*}
    \left|\frac{x_{n+1}}{d\la^2 x_n}-1\right| \leq C_q x_n.
\end{equation*}
Thus, if $x_n$ is small enough, $x_{n+1}\approx d\la^2 x_n$. More precisely, \rev{fix} $\eta \equiv \eta(q,\eps)>0$ to be determined later, and let $\delta\equiv \delta(q,\eps)$ \rev{be small enough so that $\delta\leq \eta/C_q$. Since $(x_{n})_{n\geq 0}$ is non-increasing by Lemma~\ref{lem:equiv:xn:nonreconstruct}, w.l.o.g., we can take
$n_0=\inf\{n\geq 0:x_n\leq \delta\}$. Then, for $n\geq n_0$, we have 
}
\begin{equation}\label{eq:x:onestep}
    \left|\frac{x_{n+1}}{d\la^2 x_n}-1\right|\leq \eta.
\end{equation}
We consider $\delta$ small enough so that the results of Lemmas \ref{lem:u:second:order} and \ref{lem:w:second:order} hold. \rev{In particular, if we denote $n_0^\prime =n_0+5$, then for $n\geq n_0^\prime$, the estimate \eqref{eq:u:iter} holds.}

We first start with the estimate for $u_n$: we have from Lemma \ref{lem:u:second:order} that for \rev{$n\geq 1$},
\begin{equation}\label{eq:u:onestep}
    \left|u_{n+n_0\rev{^\prime}}-\left(d\la^3 u_{n+n_0\rev{^\prime}-1}+ \frac{6(q-2)}{q-1}\E\binom{\gamma}{2}\la^4 x_{n+n_0\rev{^\prime}-1}^2\right)\right|\leq C_q x_{n+n_0\rev{^\prime}-1}^3.
\end{equation}
Note that we can use Lemma \ref{lem:u:second:order} once more to estimate $u_{n+n_0\rev{^\prime}-1}$ in \eqref{eq:u:onestep}. By iterating,
\begin{equation}\label{eq:u:onestep:0}
\begin{split}
     &\left|u_{n+n_0\rev{^\prime}}-\left((d\la^3)^{n} u_{n_0\rev{^\prime}}+ \frac{6(q-2)}{q-1}\E\binom{\gamma}{2}\la^4 \sum_{\ell=1}^{n} (d\la^3)^{\ell-1} \cdot x_{n+n_0\rev{^\prime}-\ell}^2\right)\right|\\
     &\leq C_q \sum_{\ell=1}^{n}(d|\la|^3)^{\ell-1}\cdot x_{n+n_0\rev{^\prime}-\ell}^3.
\end{split}
\end{equation}
\rev{
Note that by iterating \eqref{eq:x:onestep}, $x_{n+n_0^\prime-\ell}\leq \big(d\la^2(1-\eta)\big)^{-\ell}x_{n+n_0^\prime}$ holds for $1\leq \ell \leq n$. Thus, 
\begin{equation}\label{eq:u:onestep:0:5}
    \sum_{\ell=1}^{n}(d|\la|^3)^{\ell-1}\cdot x_{n+n_0^\prime-\ell}^3\leq \left(d\la^2(1-\eta)\right)^{-3}\sum_{\ell=1}^{n}\left(\frac{1}{(1-\eta)^3 d^2|\la|^3}\right)^{\ell-1}x_{n+n_0^\prime}^3.
\end{equation}
}
Recall that we assumed $d^2|\la|^3\geq 1+c_0$ \rev{and $d\la^2\geq 1-c_0$} for a universal constant in \eqref{eq:dlambda:lower}. Thus, taking $\eta$ small enough so that \rev{$(1-\eta)^3(1+c_0)\geq 1+\frac{c_0}{2}$}, it follows that
\begin{equation}\label{eq:u:onestep:1}
     \sum_{\ell=1}^{n}(d|\la|^3)^{\ell-1}\cdot x_{n+n_0\rev{^\prime}-\ell}^3 \lesssim  x_{n+n_0\rev{^\prime}}^3.
\end{equation}
For the first sum \rev{in \eqref{eq:u:onestep:0}}, it follows from a \rev{triangle} inequality that 
\begin{equation*}
\begin{split}
    \left|\sum_{\ell=1}^{n} (d\la^3)^{\ell-1} \frac{x_{n+n_0\rev{^\prime}-\ell}^2}{x_{n+n_0\rev{^\prime}}^2}-\sum_{\ell=1}^{n}\frac{(d\la^3)^{\ell-1}}{(d\la^2)^{2\ell}}\right|
    \leq \sum_{\ell=1}^{n} (d|\la|^3)^{\ell-1}\left|\frac{x_{n+n_0\rev{^\prime}-\ell}^2}{x_{n+n_0\rev{^\prime}}^2}-\frac{1}{(d\la^2)^{2\ell}}\right|.
\end{split}
\end{equation*}
By \eqref{eq:x:onestep}, we can further bound
\begin{equation}\label{eq:u:onestep:1:5}
    \sum_{\ell=1}^{n} (d|\la|^3)^{\ell-1}\left|\frac{x_{n+n_0\rev{^\prime}-\ell}^2}{x_{n+n_0\rev{^\prime}}^2}-\frac{1}{(d\la^2)^{2\ell}}\right|\leq \sum_{\ell=1}^{n} \frac{(d|\la|^3)^{\ell-1}}{(d\la^2)^{2\ell}}\left((1+\eta)^{\ell}-1\right)\leq \frac{1}{d^2\la^4}\sum_{\ell=1}^{\infty} \frac{(1+\eta)^{\ell}-1}{(1+c_0)^{\ell-1}},
\end{equation}
where the last inequality holds because $d|\la|\geq 1+c_0$ by \eqref{eq:dlambda:lower}. Thus, taking $\eta$ small enough so that $\eta <c_0$, we have that
\begin{equation}\label{eq:u:onestep:2}
 \left|\sum_{\ell=1}^{n} (d\la^3)^{\ell-1} \frac{x_{n+n_0\rev{^\prime}-\ell}^2}{x_{n+n_0\rev{^\prime}}^2}-\sum_{\ell=1}^{n}\frac{(d\la^3)^{\ell-1}}{(d\la^2)^{2\ell}}\right|\leq \frac{1}{d^2\la^4}g(\eta),~~\textnormal{where}~~ g(\eta):=(1+c_0)\left(\frac{1+\eta}{c_0-\eta}-\frac{1}{c_0}\right).
\end{equation}
Note that $\eta \to g(\eta)$ is continuous for $\eta <c_0$ and $g(0)=0$. Hence, we can take $\eta\equiv \eta(q,\eps)$ to be small enough so that
\begin{equation*}
    \frac{6(q-2)}{q-1}\frac{1}{d^2}\E\binom{\gamma}{2}g(\eta)\leq \frac{\eps}{2},
\end{equation*}
since $\frac{1}{d^2}\E \binom{\gamma}{2}\lesssim 1$ by Assumption \ref{assumption:unif:tail}. Now, observe that $\sum_{\ell=1}^{n}\frac{(d\la^3)^{\ell-1}}{(d\la^2)^{2\ell}}=\frac{1}{d\la^3(d\la-1)}+o_n(1)$ holds in the left hand side of \eqref{eq:u:onestep:2}. Moreover, \rev{
since $n_0^\prime=n_0+5$, Lemma~\ref{lem:u:v:w:apriori} shows that the term $(d\la^3)^{n}u_{n_0^\prime}$ in \eqref{eq:u:onestep:0} is bounded by
\begin{equation}\label{eq:power:u:neg}
    |(d\la^3)^{n}u_{n_0^\prime}|\leq C_q (d|\la|^3)^{n}x_{n_0^\prime}^2\leq C_q \left(\frac{1}{d|\la|(1-\eta)^2}\right)^{n}x_{n+n_0^\prime}^2\,,
\end{equation}
where the last inequality holds by \eqref{eq:x:onestep}. Since $d|\la|\geq 1+c_0$, by taking $\eta$ small enough, we have $|(d\la^3)^{n}u_{n_0^\prime}|=o_n(1)x_{n+n_0^\prime}^2$.
}
Therefore, \eqref{eq:u:onestep:0}, \eqref{eq:u:onestep:1}, and \eqref{eq:u:onestep:2} conclude the proof of \eqref{eq:u:exact}.

We proceed in a similar manner to show the estimate \eqref{eq:w:exact} for $w_n$: iterating the equation \eqref{eq:w:iter} in Lemma \ref{lem:w:second:order} \rev{from $n_0^\prime \equiv n_0+5$} gives
\begin{equation}\label{eq:w:onestep:0}
\begin{split}
      &\left|w_{n+n_0\rev{^\prime}}-\left((d\la^4)^{n} w_{n_0\rev{^\prime}}- \frac{2(q+1)}{(q-1)^2}\E\binom{\gamma}{2}\la^4 \sum_{\ell=1}^{n} (d\la^4)^{\ell-1} \cdot x_{n+n_0\rev{^\prime}-\ell}^2\right)\right|\\
      &\leq C_q \sum_{\ell=1}^{n}(d\la^4)^{\ell-1}\cdot x_{n+n_0\rev{^\prime}-\ell}^3.
\end{split}
\end{equation}
For the sum in the right hand side, we can use the previous estimates \eqref{eq:u:onestep:0:5} and \eqref{eq:u:onestep:1} since $d\la^4\leq d|\la|^3$. For the term $(d\la^4)^{n} w_{n_0\rev{^\prime}}$, we have that $(d\la^4)^{n} w_{n_0\rev{^\prime}}=o_n(1)\cdot x_{n+n_0\rev{^\prime}}^2$ \rev{by the same argument as done in~\eqref{eq:power:u:neg}.} For the sum in the left hand side of \eqref{eq:w:onestep:0}, we have
\begin{equation*}
\left|\sum_{\ell=1}^{n} (d\la^4)^{\ell-1} \frac{x_{n+n_0\rev{^\prime}-\ell}^2}{x_{n+n_0\rev{^\prime}}^2}-\sum_{\ell=1}^{n}\frac{(d\la^4)^{\ell-1}}{(d\la^2)^{2\ell}}\right|
    \leq \sum_{\ell=1}^{n} (d\la^4)^{\ell-1}\left|\frac{x_{n+n_0\rev{^\prime}-\ell}^2}{x_{n+n_0}^2}-\frac{1}{(d\la^2)^{2\ell}}\right|.
\end{equation*}
Since $d\la^4\leq d|\la|^3$, we can use the previous estimate \eqref{eq:u:onestep:1:5} and \eqref{eq:u:onestep:2} to have
\begin{equation}\label{eq:w:onestep:2}
    \left|\sum_{\ell=1}^{n} (d\la^4)^{\ell-1} \frac{x_{n+n_0-\ell}^2}{x_{n+n_0}^2}-\sum_{\ell=1}^{n}\frac{(d\la^4)^{\ell-1}}{(d\la^2)^{2\ell}}\right|\leq \frac{1}{d^2\la^4}g(\eta).
\end{equation}
Note that $\sum_{\ell=1}^{n}\frac{(d\la^3)^{\ell-1}}{(d\la^2)^{2\ell}}=\frac{1}{d(d-1)\la^4}+o_n(1)$ holds. Therefore, having \eqref{eq:w:onestep:0} and \eqref{eq:w:onestep:2} in hand, we can repeat the previous argument to conclude the proof of \eqref{eq:w:exact}.
\end{proof}
\subsection{Proof of Theorem \ref{thm:four:antiferro}}
By the estimates of $\Xi_1, \Xi_2$ in Lemma \ref{lem:xi1} and Lemma \ref{lem:xi2}, and the estimates of $u_n$ and $w_n$ in Lemma \ref{lem:exact:u:w}, the following proposition is immediate.
\begin{prop}\label{prop:final}
For any $\eps>0$, there exists a positive constant $\delta(q,\eps)>0$ and a positive integer $N(q,\eps)$, which only depend on $q$ and $\eps$, such that if \rev{$x_{n_0}\leq \delta(q,\eps)$ for some $n_0\geq 1$ and $n$ is large enough so that $n\geq n_0+5$ and $n\geq N(q,\eps)$ holds,} then
\begin{equation}\label{eq:f:q}
\begin{split}
   &\left|x_{n+1}-f_q(x_n,d,\la)\right|\leq \eps x_n^3\quad\textnormal{where}\\
   &f_q(x_n,d,\la):= d\la^2 x_n+\E\binom{\gamma}{2}\la^4 \frac{q(q-4)}{q-1}x_n^2+\E\binom{\gamma}{3}\la^6 \frac{q^2(q^2-18q+42)}{(q-1)^2}x_n^3\\
   &\qquad\qquad\qquad\quad\quad+\left(\E\binom{\gamma}{2}\right)^2 \la^6\left(\frac{4q^2(q+1)}{(q-1)^2}\frac{1}{d(d-1)}-\frac{24q^2(q-2)}{(q-1)^2}\frac{1}{d(d\la-1)}\right)x_n^3
\end{split}
\end{equation}
\end{prop}
With Proposition \ref{prop:final} and Lemma \ref{lem:drop:small} in hand, we now prove Theorem \ref{thm:four:antiferro}.
\begin{proof}[Proof of Theorem \ref{thm:four:antiferro}]
We first consider the case where $q=4$. Throughout, we fix $d$ such that $1<d<d^\star$. A direct calculation shows that $f_4(\cdot)$ in \eqref{eq:f:q} equals
\begin{equation}\label{eq:def:g}
\begin{split}
    &f_4(x_n,d, \la)=d\la^2 x_n+g_4(d,\la)x_n^3,\quad\textnormal{where}\\
    &g_4(d,\la):=\frac{16}{9}\la^6\left(\left(\E[\gamma(\gamma-1)]\right)^2\left(\frac{5}{d(d-1)}-\frac{12}{d(d\la-1)}\right)-\frac{7}{3}\E[\ga(\ga-1)(\ga-2)]\right).
\end{split}
\end{equation}
Observe that $d^\star$ in \eqref{eq:def:d:star} is defined so that 
\begin{equation*}
    d<d^\star \implies g_4(d,-d^{-1/2})>0.
\end{equation*}
That is, at $\la= -d^{-1/2}$, the coefficient of $x_n^3$ in front of $f_4(x_n,d,\la)$ is positive.

Now, we show that there exists some $-d^{-1/2}<\la<0$ such that there is reconstruction at $\la$. That is, $\lim_{n\to\infty} x_n>0$ (cf. Lemma \ref{lem:equiv:xn:nonreconstruct}). To this end, fix $\kappa\equiv \kappa(d)<d^{1/2}$ and consider the constant $c\equiv c(4,\kappa,d)$ in Lemma \ref{lem:drop:small}. Then, consider
\begin{equation}\label{eq:drop:small:conseq}
    \eps(d):= \frac{1}{3}g_4(d,-d^{-1/2})>0,\qquad \eta(d):= \rev{\min\left(\frac{3}{4}c^{N\left(4,\eps(d)\right)}, c^{5}\delta\left(4,\eps(d)\right)\right)}>0,
\end{equation}
where \rev{$\delta(q,\eps)$ and }$N(q,\eps)$ are from Proposition \ref{prop:final}. Then, because $x_0=\frac{q-1}{q}=\frac{3}{4}$, it follows from Lemma \ref{lem:drop:small} that
\begin{equation}\label{eq:x:n:lower:1}
    x_n\geq \eta(d)\quad\textnormal{for all}\quad n\leq N\left(4, \eps(d)\right).
\end{equation}
Now, take $\la_0\equiv\la_0(d)$ so that $-d^{-1/2}<\la_0<-\kappa$ and it also satisfies
\begin{equation}\label{eq:def:lambda:0}
    d\la_0^2 +\left(g(d,\la_0)-\eps(d)\right)\eta(d)^2\geq 1,\quad\textnormal{and}\quad g(d,\la_0)>\eps(d).
\end{equation}
Such $\la_0$ exists by continuity \rev{since $\eps(d)\equiv g_4(d,-d^{-1/2})/3>0$}. Suppose by contradiction that there is nonreconstruction at $\la_0$, i.e. $\lim_{n\to\infty} x_n=0$. \rev{Then, let $n_0\equiv \inf\{n\geq 0:x_n\leq \delta\}$ is well-defined. Note that since $(x_n)_{n\geq 0}$ is non-increasing by Lemma~\ref{lem:equiv:xn:nonreconstruct} and $x_{n+1}\geq cx_n$ holds by Lemma~\ref{lem:drop:small}, we have that
\begin{equation}\label{eq:x:n:lower:2}
   x_n\geq \eta(d)\quad\textnormal{for all}\quad n\leq n_0+5.
\end{equation}
}
Furthermore, it follows from Proposition \ref{prop:final} that
\begin{equation*}
    x_{n+1}\geq \Big(d\la_0^2+\big(g(d,\la_0)-\eps(d)\big)x_n^2\Big) x_n\quad\textnormal{for all}\quad n\geq \rev{\max\Big(n_0+5, N\left(4, \eps(d)\right)\Big)}.
\end{equation*}
\rev{
Meanwhile, by \eqref{eq:x:n:lower:1} and \eqref{eq:x:n:lower:2}, $x_n\geq \eta(d)$ holds for $n\leq \max\big(n_0+5, N(4,\eps(d))\big)$. Thus, combining the inequality above with~\eqref{eq:def:lambda:0} shows that $x_n\geq \eta(d)$ for all $n\geq 0$.} This contradicts $\lim_{n\to\infty} x_n=0$, thus we conclude that there is reconstruction at $\la_0>-d^{-1/2}$.

Next, we consider the case $q\geq 5$. Note that Proposition \ref{prop:final} implies the following. For any $\eps>0$, there exist $\delta(q,\eps)>0$ and $N(q,\eps)$ such that if \rev{$x_{n_0}\leq \delta(q,\eps)$ and $n$ is large enough so that $n\geq n_0+5$ }and $n\geq N(q,\eps)$ hold, then
\begin{equation}\label{eq:q:geq:5}
\begin{split}
    \left|x_{n+1}-\left(d\la^2 x_n+ g_q(d,\la) x_n^2\right) \right|\leq \eps x_n^2,\quad\textnormal{where}\quad g_q(d,\la):=\E \binom{\ga}{2}\la^4 \frac{q(q-4)}{q-1}.
\end{split}
\end{equation}
Observe that $g_q(d,\la)>0$ holds for $q\geq 5$ and $\la \neq 0$ regardless of $d>1$:
\begin{equation}\label{eq:q:geq:5:1}
    g_q(d,\la)\geq \frac{q(q-4)}{q-1}\binom{d}{2}\la^4>0,
\end{equation}
where the inequality is due to Cauchy-Schwartz and $\E \ga =d$. Having \eqref{eq:q:geq:5} and \eqref{eq:q:geq:5:1} in hand, the previous argument for $q=4$ case can be repeated to show that $\la^{+}(q,d)<d^{1/2}$ and $\la^{-}(q,d)>-d^{-1/2}$ when $q\ge 5$.
\end{proof}

\section{Normal approximation for large degrees}
\label{sec:clt}
In this section, we prove Theorem \ref{thm:KS:tight}. Throughout, we assume that $\{\mu_d\}_{d>1}$ satisfies Assumptions \ref{assumption:unif:tail} and \ref{assumption:tight}. We begin with the necessary notations: \rev{let} $\psi: \R^q \to \R$ \rev{be} 
\begin{equation*}
    \psi(w_1,...,w_q)=\frac{e^{w_1}}{\sum_{i=1}^{q}e^{w_i}}.
\end{equation*}
\rev{Let} $U_{ij}, i\leq q, j \leq \ga$ \rev{be} the random variables \rev{defined by}
\begin{equation}\label{eq:def:U}
    U_{ij}:= \log \left(1+\la q \left(Y_{ij}-\frac{1}{q}\right)\right)
\end{equation}
and denote $U_j=(U_{1,j},....,U_{qj})\in \R^q$. Then, by the expression $X^{+}(n+1)=\frac{Z_1}{\sumi Z_i}$ in \eqref{eq:key:recursion}, we have
\begin{equation}\label{eq:x:alter}
    x_{n+1}=\E \bigg[\psi\Big(\sumj U_j\Big)\bigg]-\frac{1}{q}.
\end{equation}
To perform the normal approximation of $\sumj U_j$, we first estimate the means and covariances of $U_{ij}$'s. Recall that we denoted $T_j$ \rev{as} the $j$'th subtree of $T$ for $1\leq j \le \ga$, and $\E_{\sig}$ \rev{as} the conditional expectation w.r.t. $T$.
\begin{lemma}\label{lem:U:mo}
There exists constant $C_q>0$ and $d\equiv d(q)$ such that when $d>d(q)$, the following holds almost surely for $1\leq j \leq \ga$ conditional on the tree $T$. For $i\geq 2$,
\begin{equation}\label{eq:lem:U:mo:1}
    \left|\E_{\sig}U_{1j}-\frac{q}{2}\la^2  \cdot x_n(T_j)\right| \leq C_q d^{-3/2},\quad \big|\E_{\sig}U_{ij}+\left(\frac{q}{2}+\frac{q}{q-1}\right)\la^2 \cdot x_n(T_j)\big| \leq C_q d^{-3/2}.
\end{equation}
For $1\leq i_1<i_2\leq q$, we have
\begin{equation}\label{eq:lem:U:mo:2}
    \left|\Var_{\sig}(U_{i_1 j})-q\la^2  \cdot x_n(T_j)\right| \leq C_q d^{-3/2},\quad \big|\Cov_{\sig}(U_{i_1j}, U_{i_2 j})+\frac{q}{q-1}\la^2 \cdot x_n(T_j)\big| \leq C_q d^{-3/2}.
\end{equation}
\end{lemma}
\begin{proof}
Note that conditional on $T$ and $\sig_{u_j}=i^\prime$, $Y_{ij}$ is distributed as $X^{+}(n,T_j)$ when $i^\prime = i$ and $X^{-}(n,T_j)$ when $i^\prime \neq i$. Thus, we can calculate
\begin{equation}\label{eq:conditional:mo:1}
\begin{split}
    \E_{\sig}\left(Y_{1j}-\frac{1}{q}\right)
    &=\P_{\sig}(\sig_{u_j}^1=1)\E_{\sig}\left(X^{+}(n,T_j)-\frac{1}{q}\right)+\left(1-\P_{\sig}(\sig_{u_j}^1=1)\right)\E_{\sig}\left(X^{-}(n,T_j)-\frac{1}{q}\right)\\
    &=\la x_n(T_j).
\end{split}
\end{equation}
Similarly, we can calculate the other conditional moments of $Y_{ij}$'s which are the conditional analogues for Lemma \ref{lem:Y:moment}: for all distinct pairs $2\leq i_1, i_2\leq q$, we have
\begin{equation}\label{eq:conditional:mo:2}
    \begin{split}
    &\E_{\sig}\left(Y_{i_1 j}-\frac{1}{q}\right)=\frac{-\la x_n(T_j)}{q-1},\quad\quad\E_{\sig}\left(Y_{1j}-\frac{1}{q}\right)^2=\frac{x_n(T_j)}{q}+\la\left(z_n(T_j)-\frac{x_n(T_j)}{q}\right),\\
    &\E_{\sig}\left(Y_{i_1 j}-\frac{1}{q}\right)^2=\frac{x_n(T_j)}{q}-\frac{\la}{q-1}\left(z_n(T_j)-\frac{x_n(T_j)}{q}\right),\\
    &\E_{\sig}\left(Y_{1j}-\frac{1}{q}\right)\left(Y_{i_1 j}-\frac{1}{q}\right)
    =-\frac{x_n(T_j)}{q(q-1)}-\frac{\la}{q-1}\left(z_n(T_j)-\frac{x_n(T_j)}{q}\right),\\
    &\E_{\sig}\left(Y_{i_1 j}-\frac{1}{q}\right)\left(Y_{i_2 j}-\frac{1}{q} \right)=-\frac{x_n(T_j)}{q(q-1)}+\frac{2\la}{(q-1)(q-2)}\left(z_n(T_j)-\frac{x_n(T_j)}{q}\right).
    \end{split}
\end{equation}
Now, let $x=\la q (Y_{1j}-1/q)$. Since $0\leq Y_{1j}\leq 1$, it follows that $|x|\leq d^{-1/2}(q-1)$. Thus, for large enough $d\geq d(q)$, we can approximate $\left|\log(1+x)-(x-\frac{1}{2}x^2)\right| \leq |x|^3$, which gives
\begin{equation*}
    \left|\E_{\sig}U_{1j}-\left(\la q\E_{\sig}\left(Y_{1j}-\frac{1}{q}\right)-\frac{1}{2}\la^2 q^2\E_{\sig}\left(Y_{1j}-\frac{1}{q}\right)^2\right)\right|\leq \E_{\sig}|\la|^3 q^3\left|Y_{1j}-\frac{1}{q}\right|^3\le C_q d^{-3/2}.
\end{equation*}
Thus, we can use the identities in \eqref{eq:conditional:mo:1} and \eqref{eq:conditional:mo:2} to have that 
\begin{equation*}
    \E_{\sig}U_{1j}=\frac{q}{2}\la^2  \cdot x_n(T_j)-\frac{q^2}{2}\la^3 \left(z_n(T_j)-\frac{x_n(T_j)}{q}\right)+O_q(d^{-3/2})=\frac{q}{2}\la^2  \cdot x_n(T_j)+O_q(d^{-3/2}),
\end{equation*}
where the final equality holds since $|\la|^3\leq d^{-3/2}$ and $0\leq z_n(T_j)\leq x_n(T_j)\leq 1$ by Lemma \ref{lem:basic}. Other identities in \eqref{eq:lem:U:mo:1} and \eqref{eq:lem:U:mo:2} follow similarly.
\end{proof}
We now show that $(x_n(T_j))_{j\le \ga}$ appearing in Lemma \ref{lem:U:mo} is self-averaging under Assumption \ref{assumption:tight}.
\begin{lemma}\label{lem:xn}
For any $\eps>0$, there exists $d^\prime(\eps)>0$ such that if $d>d^\prime (\eps)$, then
\begin{equation*}
    \E\bigg|\frac{1}{d}\sum_{j=1}^{\ga}x_n(T_j)-x_n\bigg|<\eps,
\end{equation*}
where the expectation is with respect to $T\sim \GW(\mu_d)$.
\end{lemma}
\begin{proof}
Fix arbitary $\eps>0$. By Assumption \ref{assumption:tight}, there exists $C_i\equiv C_i(\eps/2), i=1,2,$ and $d_0\equiv d_0(\eps/2)$ such that if $d\geq d_0$, then $C_1 d\le \ga \le C_2 d$ holds with probability at least $1-\frac{\eps}{2}$. Thus, we have
\begin{equation*}
    \E\bigg|\frac{1}{d}\sum_{j=1}^{\ga}x_n(T_j)-x_n\bigg|\leq \sup_{C_1 d\le D \le C_2 d}\E\Bigg[\bigg|\frac{1}{d}\sum_{j=1}^{\ga}x_n(T_j)-x_n\bigg|\,\,\Bigg|\,\, \ga =D\Bigg]+\frac{\eps}{2}.
\end{equation*}
Note that by definition of Galton-Watson trees $T\sim \GW(\mu_d)$, conditional on $\gamma=D$, $(T_j)_{j\le D}$ are i.i.d. with $T_j \stackrel{d}{=} T$. Hence, conditional on $\gamma=D$, $(x_n(T_j))_{j\le D}$ are i.i.d. with mean $\E x_n(T_j)= x_n$ and variance $\Var(x_n(T))$, whence by Cauchy-Scwartz inequality,
\begin{equation*}
    \E\Bigg[\bigg|\frac{1}{d}\sum_{j=1}^{\ga}x_n(T_j)-x_n\bigg|\,\,\Bigg|\,\, \ga =D\Bigg]\le \Var\left(\frac{1}{d}\sum_{j=1}^{D}x_n(T_j)\right)=\frac{D}{d^2}\Var(x_n(T)).
\end{equation*}
Since $0\leq x_n(T)\leq 1$ holds a.s. by Lemma \ref{lem:basic}, we have $\Var(x_n(T))\leq 1$. Thus, the two displays above altogether show
\begin{equation*}
    \E\bigg|\frac{1}{d}\sum_{j=1}^{\ga}x_n(T_j)-x_n\bigg|\leq \frac{C_2(\eps)}{d}+\frac{\eps}{2},
\end{equation*}
whence by taking $d^\prime(\eps)$ large enough, the right hand side is less than $\eps$ for any $d>d^\prime(\eps)$.
\end{proof}

The following normal approximation can be established using a classical Lindeberg method \cite{Lindeberg:22, Chatterjee06}. For completeness, we give a quantitative proof of Proposition \ref{prop:clt} in Section 1 of the Supplementary material\rev{.}
\begin{prop}\label{prop:clt}
Let $\phi:\R^{q}\to \R$ be a thrice differentiable and bounded function with bounded derivatives up to third order. Let $V_1,...,V_D\in \R^q$ be independent random vectors. Suppose there exists a deterministic vector $\mu \in \R^q$ and a matrix $\Sigma \in \R^{q\times q}$ such that the following holds: for some constant $C>0$, 
\begin{equation}\label{eq:clt:condition}
\begin{split}
    &\max\left(\Big\|\sum_{j=1}^{D}\E V_j - \mu\Big\|_{\infty},\Big\|\sum_{j=1}^{D}\Cov(V_j) -\Sigma \Big\|_{\infty} \right)\le \frac{C}{\sqrt{D}},\\
    &\max\left(\big\|\mu\big\|_{\infty}, \big\|\Sigma\big\|_{\infty}\right)\leq C,\quad\quad\max_{1\leq j\leq q}\big\|V_j\big\|_{\infty}\leq \frac{C}{\sqrt{D}},~~\textnormal{almost surely.}
\end{split}
\end{equation}
Then, for any $\eps>0$, there exists $D_0\equiv D_0(\eps, q,\phi,C)$ such that if $D>D_0$, then 
\begin{equation*}
    \bigg|\E\phi\bigg(\sum_{j=1}^{D} V_j\bigg)-\E \phi(W)\bigg|\leq \eps,
\end{equation*}
where $W\sim \normal(\mu,\Sigma)$ is a Gaussian vector with mean $\mu$ and covariance $\Sigma$. 
\end{prop}
Let $\mu^{\dagger}\equiv \mu^{\dagger}(q)\equiv (\mu^{\dagger}_i)_{i\le q} \in \R^q$ be 
\begin{equation}\label{eq:def:mu:dagger}
    \mu^{\dagger}_i=
    \begin{cases}
    \frac{q}{2}& \qquad\qquad i=1,\\
    -\frac{q}{2}-\frac{q}{q-1} &\qquad\qquad i\neq 1, 
    \end{cases}
\end{equation}
and let $\Sigma^{\dagger}\equiv \Sigma^{\dagger}(q)\equiv (\Sigma^{\dagger}_{ij})_{1\leq i,j\leq q}\in \R^{q\times q}$ be
\begin{equation}\label{eq:def:sigma:dagger}
    \Sigma^{\dagger}_{ij}=
    \begin{cases}
    q& \qquad\qquad\qquad i=j,\\
    -\frac{q}{q-1} &\qquad\qquad\qquad i\neq j. 
    \end{cases}
\end{equation}
Then, let $W \sim \normal(0,\Sigma^{\dagger})$. Recalling that $\psi(w_1,...,w_q)=e^{w_1}/\sumi e^{w_i}$, define the function $g_q:\R_{\geq 0} \to\R$ as
\begin{equation*}
    g_q(s):=\E \psi\left(s\mu+\sqrt{s}W\right)-\frac{1}{q}.
\end{equation*}
The following lemma is from \cite{Sly:11}.
\begin{lemma}[Lemma 4.4 in \cite{Sly:11}]
\label{lem:g}
For each $q$, the function $g_q$ is monotonically increasing on $\R_{\geq 0}$ and continuously differentiable on the interval $(0,\frac{q-1}{q}]$.
\end{lemma}

\begin{prop}\label{prop:large:d}
For any $\eps>0$, there exists $d(\eps)$ such that if $d>d(\eps)$, then 
\begin{equation*}
    \left|x_{n+1}-g_q\left(d\la^2 x_n\right)\right|\leq \eps
\end{equation*}
\end{prop}
\begin{proof}
Recalling the equation $x_{n+1}=\E\psi\big(\sumj U_j\big)-1/q$ from \eqref{eq:x:alter}, we estimate the conditional expectation $\E_{\sig}\psi\big(\sumj U_j\big)$ using Proposition \ref{prop:clt}. In order to use Proposition \ref{prop:clt}, we need to ensure that the tree $T$ has large enough degree at the root. Recall from Assumption \ref{assumption:tight} that we have $\P(\ga\in [C_1d, C_2 d])\leq 1-\frac{\eps}{3}$ for $C_i\equiv C_i(\eps/3), i=1,2$. To this end, let $\TT_{\eps}$ be the set of the rooted trees such that the degree at the root is at least $C_2 d$. Under the law $T\sim \GW(\mu_d)$, we have
\begin{equation}\label{eq:prop:large:d:1}
\P\left(\left(\TT_{\eps}\right)^{\mathsf{c}}\right)\le \frac{\eps}{3}.
\end{equation}
Define $\hat{x}_n:=\frac{1}{d}\sum_{j=1}^{\ga}x_n(T_j)$. Then, conditioned on $T\in \TT_{\eps}$, it follows from Lemma \ref{lem:U:mo} that
\begin{equation*}
\begin{split}
     &\max\left(\Big\|\sum_{j=1}^{\ga}\E_{\sig} U_j - \mu^{\dagger}d\la^2 \hat{x}_n\Big\|_{\infty},\Big\|\sum_{j=1}^{\ga}\Cov_{\sig}(U_j) -\Sigma^{\dagger}d\la^2 \hat{x}_n\Big\|_{\infty} \right)\\
     &\le C_q \ga d^{-3/2}\le C_q C_1^{-3/2}\ga^{-1/2}, 
\end{split}
\end{equation*}
where the last inequality holds becasue $\ga \geq C_1 d$. Note that $\hat{x}_n\leq d^{-1}\ga\le C_2$ holds for $T\in \TT_{\eps}$, so the second condition in \eqref{eq:clt:condition} is also satisfied. Moreover, since $Y_{ij}\in [0,1]$ holds and $|\la|\leq \frac{1}{\sqrt{d}}=O(\gamma^{-1/2})$, it follows from the definition of $U_{ij}$ that $ \norm{U_j}_{\infty}=O(\gamma^{-1/2}).$ Thus, all three conditions in \eqref{eq:clt:condition} is satisfied. Also, it is straight forward to check that $\phi$ has bounded derivatives up to the third order. Therefore, Proposition \ref{prop:clt} implies that if $d\geq d(\eps)$ for some $d(\eps)$ so that $\ga \ge C_1 d$ is large enough, we have that conditioned on $T\in \TT_{\eps}$,
\begin{equation}\label{eq:prop:large:d:2}
    \bigg|\E_{\sig}\psi\Big(\sumj U_j\Big)-\left(g_q\left(d\la^2 \hat{x}_n\right)+\frac{1}{q}\right)\bigg|\leq \frac{\eps}{3}.
\end{equation}
By Lemma \ref{lem:xn} and Lemma \ref{lem:g}, we can approximate $g_q(d\la^2 \hat{x}_n)$ for large enough $d>d(\eps)$ by
\begin{equation}\label{eq:prop:large:d:3}
    \E\left|g_q\left(d\la^2\hat{x}_n\right)-g_q(x_n)\right|\leq C_q \E\left|\hat{x}_n-x_n\right|\leq \frac{\eps}{3}. 
\end{equation}
Therefore, by \eqref{eq:prop:large:d:1}-\eqref{eq:prop:large:d:3}, we can bound for $d>d(\eps)$,
\begin{equation*}
    \left|x_{n+1}-g_q\left(d\la^2 x_n\right)\right|
    \leq \frac{2\eps}{3}+\sup_{T\in \TT_{\eps}} \bigg|\E_{\sig}\psi\Big(\sumj U_j\Big)-\left(g_q\left(d\la^2 \hat{x}_n\right)+\frac{1}{q}\right)\bigg|\leq \eps,
\end{equation*}
where in the second equality, we also used $0\leq \psi(\cdot)\leq 1$ and $0\leq g_q(\cdot)+1/q\leq 1$.
\end{proof}
It was shown in Lemma 4.7 of \cite{Sly:11} that $g_q(s)<s$ for all $0<s\leq \frac{q-1}{q}$ when $q=3$ while there exists $0<s\leq \frac{q-1}{q}$ such that $g_q(s)>s$ for $q\geq 5$. The following lemma whose proof is given in Section 2 of the Supplementary material deals with the critical case $q=4$.
\begin{lemma}\label{lem:g4}
When $q=4$, for all $0<s\leq \frac{q-1}{q}$, we have $g_q(s)<s$.
\end{lemma}
With Proposition \ref{prop:final}, Proposition \ref{prop:clt}, and Lemma \ref{lem:g4} in hand, we prove Theorem \ref{thm:KS:tight}.
\begin{proof}[Proof of Theorem \ref{thm:KS:tight}]
We first consider the case where $q=4$ in the antiferromagnetic regime. By Lemma \ref{lem:la:nonreconstruct}, it suffices to prove that there is nonreconstruction at $\la=-d^{-1/2}$ for $d\geq d_0$. Recall that the function $f_4(\cdot)$ in \eqref{eq:f:q} can be expressed by $f_4(x_n,d,\la)=d\la^2 x_n+g_4(d,\la) x_n^3$ for $g_4(d,\la)$ defined in \eqref{eq:def:g}. Then, by a direct calculation,
\begin{equation*}
    g_4(d,-d^{1/2})=-\frac{16}{9d^3}\left(\frac{7}{3}\E\ga(\ga-1)(\ga-2)-\left(\E\ga(\ga-1)\right)^2\frac{12\sqrt{d}-7}{d(d-1)}\right).
\end{equation*}
Note that $f(x)=x(x-1)(x-2)$ is convex for $x>1$, so $\E\ga(\ga-1)(\ga-2)\geq d(d-1)(d-2)$ holds by Jensen's inequality. Moreover, $\E \ga(\ga-1)\leq C d^2$ holds for a constant $C$ which only depends on the function $K(\cdot)$ by Assumption \ref{assumption:unif:tail}. Thus, for large enough $d\geq d_0$, it follows that $g_4(d,-d^{1/2})\rev{<-4:}$
\begin{equation*}
    g_4(d,-d^{1/2})\leq -\frac{16}{9}\left(\frac{7(d-1)(d-2)}{3d^2}-C^2 \frac{12\sqrt{d}-7}{d-1}\right)\rev{<-4}.
\end{equation*}
Recalling the constant $\delta(q,\eps)$ from Proposition \ref{prop:final}, let \rev{$\delta\equiv \delta(4,3)>0$. Then, for $d\geq d_0$ and $\la=-d^{1/2}$, Proposition \ref{prop:final} with $q=4$ and $\eps=3$ shows that if $x_{n_0}\leq \delta$ for some $n_0\geq 1$ and $n$ is large enough so that $n\geq n_0+5$ and $n\geq N(4,3)$, then
\begin{equation}\label{eq:x:decrease}
    x_{n+1}\leq x_n-x_n^3.
\end{equation}
Meanwhile, for large enough $d$ and $\la=-d^{1/2}$, $x_{n_0}\leq \delta$ indeed holds for large enough $n_0$ by Proposition \ref{prop:large:d} and Lemma \ref{lem:g4}: let $\eta \equiv \frac{1}{2}\inf_{\frac{\delta}{2}\le s\le \frac{3}{4}}(s-g_4(s))$, which is positive by Lemma \ref{lem:g4} and continuity of $g$. Then, Proposition \ref{prop:large:d} implies that when $d\geq d(\eta)$,
}
\begin{equation*}
    \limsup_{n\to\infty} x_n\leq g_4 (\limsup_{n\to\infty} x_n)+\eta,
\end{equation*}
so we must have $\limsup_{n\to\infty} x_n \leq \rev{\frac{\delta}{2}}$. Therefore, for large enough $d$, \rev{\eqref{eq:x:decrease} yields that $x_n\to 0$ as $n\to\infty$ at $\la=-d^{1/2}$.}

For the case where $q=4$ in the ferromagnetic regime, note that at $\la=d^{1/2}$
\begin{equation*}
    g_4(d,d^{1/2})=-\frac{16}{9d^3}\left(\frac{7}{3}\E\ga(\ga-1)(\ga-2)+\left(\E\ga(\ga-1)\right)^2\frac{12\sqrt{d}+7}{d(d-1)}\right)\rev{<-4},
\end{equation*}
so repeating the previous argument shows that \rev{$x_n\to 0$ as $n\to\infty$ at $\la=d^{1/2}$. Therefore, KS bound is sharp for large enough $d$ when $q=4$. For $q=3$, $g_3(s)<s$ holds for $0<s<2/3$ by Lemma 4.7 of \cite{Sly:11}, thus Propositions~\ref{prop:final} and \ref{prop:large:d} again yields that the KS bound is sharp for large enough $d$.}
\end{proof}
\section{Coupling SBM to broadcast processes}
\label{s:coupling}

In this section we prove Theorem~\ref{t:nonrecon.to.sbm} and Corollary \ref{cor:two.point}.

For an estimator of the states $\hat{\sigma}\equiv \hat{\sigma}(G)$ we define its expected success rate on the graph $G$ to be
\[
S(\hat{\sigma})= \E\bigg[\max_{\Gamma\in S_q}\frac1{n}\sum_u \one(\Gamma(\hat{\sigma}_u)=\sigma_u)\bbgiven G\bigg]
\]
Let $u_1,\ldots,u_m$ be $m\equiv \rev{\lfloor n^{1/5} \rfloor}$ vertices chosen uniformly at random without replacement from $G$.  Let
\[
S_m(\hat{\sigma})=\E\bigg[\max_{\Gamma\in S_q}\frac1{m}\sum_{j=1}^m \one(\Gamma(\hat{\sigma}_{u_j})=\sigma_{u_j})\bbgiven G\bigg]
\]
and note that by random sampling of the $u_j$ and the choice of the permutation we have that
\[
\E \max_{\hat{\sigma}}S_m(\hat{\sigma})\geq \E \max_{\hat{\sigma}} S(\hat{\sigma}),
\]
where $\hat{\sigma}$ is maximized over $G$\rev{-}measurable estimators.
Our plan is to couple the local neighborhoods of the $u_i$ with independent Poisson Galton Watson trees, with spins given by the broadcast process and then use nonreconstruction to show that one cannot correctly assign a better than $\frac1{q}+o(1)$ fraction of the $\sigma_{u_i}$\rev{'s}.  We first give some notation for local neighborhoods.

In the nonreconstruction regime, we can fix a constant $\ell$ large enough such that for $T\sim\GW(\textsf{Poi}(d))$,
\begin{equation}\label{eq:ell.defn}
\E\Big[\max_i\P\big(\sigma_\rho=i\mid T, \sigma(\ell)\big)\Big]\leq \frac1{q} + \epsilon/4,
\end{equation}
where we recall that $\sigma(\ell)$ are the spins at depth $\ell$ neighborhood of the root of $T$. That is, for most trees and most boundary conditions, $\P(\sigma_\rho=i\mid T, \sigma(\ell))\approx \frac1{q}$.  Let $T_\ell$ be $T$ truncated at depth $\ell$ and let $\TT_\ell$ be the set of rooted depth $\ell$ trees.  Let $\TT_{\ell,r}$ be the set of depth $\ell$ trees with maximal degree $r$ which is of course a finite set.  We can choose $r$ large enough such that
\begin{equation}\label{eq:r.defn}
\P(T_\ell \not\in \TT_{\ell,r}) \leq \epsilon/4.
\end{equation}
Define the set of depth $\ell$ trees with spin configurations as
\[
\RR_\ell=\Big\{(t,s,s(\ell)):t\in\TT_\ell,~s\in[q],~s(\ell)\in [q]^{L(t,\ell)},~ \P\big(\sigma_\rho=s,\sigma(\ell)=s(\ell)\given T=t\big)>0\Big\}
\] 
where $L(t,\ell)$ is the set of vertices of $t$ at level $\ell$.  Similarly let $\RR_{\ell,r}=\{(t,s,s(\ell))\in\RR_\ell:t\in\TT_{\ell,r}\}$.
If $\P$ is the measure where $t$ is selected from $\GW(\textsf{Poi}(d))$ and $\sigma$ is generated from the broadcast process then, since the Poisson distribution puts positive mass on every integer we have that for some $\delta=\delta(\ell,r,q,\lambda)$,
\begin{equation}\label{eq:tree.freq}
\inf_{(t,s,s(\ell))\in \RR_{\ell,r}} \P\big(T_\ell=t,~\sigma_\rho=s,~\sigma(\ell) = s(\ell)\big) =\delta>0.
\end{equation}

Next we will establish the coupling of local \rev{neighborhoods}.  We let $\{T^{(j)}_\ell\}_{j\leq m}$ be the \rev{i.i.d.} $\GW(\textsf{Poi}(d))$ depth $\ell$ trees rooted at $\rho^{(j)}$ and let $\{\sigma^{(j)}\}_{j\leq m}$ be spin configurations generated by the broadcast process on the $T^{(j)}_\ell$.  In order to show this we need the following lemma to couple Binomial and Poisson distributions.

\begin{lemma}\label{l:bin.poi}
The total variation distance between a Poisson and Bionomial distribution is bounded by
\begin{equation}\label{eq:poi.bin.dtv}
d_{TV}(\textsf{Poi}(\theta),\textsf{Bin}(r,p))\leq \theta^2/r + |rp-\theta|.    
\end{equation}
\end{lemma}

\begin{proof}
Since $d_{TV}(\textsf{Ber}(p),\textsf{Ber}(p')) = |p'p|$ we can couple random variables with distributions $\textsf{Ber}(p)$ and $\textsf{Ber}(p')$ except with probability $|p-p'|$.  So if we had a sequence of $r$ independent random variables with $\textsf{Ber}(p)$ distribution and $r$ independent random variables with $\textsf{Ber}(p')$ we could couple the sequence with probability at least $1-r|p-p'|$.  Then by taking the sum we have that
\[
d_{TV}(\textsf{Bin}(r,p),\textsf{Bin}(r,p'))\leq |rp-rp'|.
\]
Also since $d_{TV}(\textsf{Poi}(p),\textsf{Ber}(p)) = p-e^{-p}\leq p^2$ by an analogous argument we have that
\[
d_{TV}(\textsf{Poi}(rp),\textsf{Bin}(r,p))\leq rp^2.
\]
Setting $\theta=rp'$ and combining the last two equations we have that
\begin{align*}
d_{TV}(\textsf{Poi}(\theta),\textsf{Bin}(r,p))&\leq d_{TV}(\textsf{Poi}(\theta),\textsf{Bin}(r,\theta/r)) + d_{TV}(\textsf{Bin}(r,\theta/r),\textsf{Bin}(r,p))\\
&\leq \theta^2/r + |rp-\theta|.   
\end{align*}
\end{proof}

We are now ready to couple the local neighborhoods of the graph with the independent \rev{Galton-Watson} trees.
\begin{lemma} \label{lem:coup_trees_g}
With probability $1-o(1)$ we can couple the local neighborhoods of the $u_j$ such that for all $1\leq j\leq m$
\begin{equation}\label{eq:coupling}
B_\ell(u_j) \cong T^{(j)}_\ell, \qquad \sigma_{B_\ell(u_j)} = \sigma^{(j)}_{T^{(j)}_\ell},
\end{equation}
where $\cong$ is graph isomorphism.  Furthermore the $B_\ell(u_j)$ are disjoint and
\[
\sum_{j} |T^{(j)}_\ell|\leq n^{1/4}.
\]
\end{lemma}
\begin{proof}
Let $K=\sum_{j} |T^{(j)}_\ell|$ be the total number of vertices revealed in the $m$ neighborhoods and let $\AA_1$ be the event $\{K\leq n^{1/4}\}$.  Since $\E[K]=\rev{\lfloor n^{1/5} \rfloor \cdot } \E|T^{(1)}_\ell|$ and $\E|T^{(1)}_\ell|=\sum_{s=0}^\ell d^r<\infty$ then by Markov's inequality,
\[
\P(\AA_1)=O(md^\ell n^{-1/4})=o(1).
\]
In the SBM graph $G$, we can construct the depth $\ell$ neighborhoods of the $u_j$ by exploring each neighborhood and labeling by breadth first search.  The first step is to assign the spins $\sigma_{u_j}$ independently.  We can couple these so that $\sigma_{u_j} = \sigma^{(j)}_{\rho^{(j)}}$.

Now, we first reveal the total number of vertices of each state $N_i=|\{u:\sigma_u=i\}|$ for each $i$.  Let $\AA_2$ be the event $|N_i-n/q|\leq n^{3/5}$ for all $i$. Since these are binomial
\[
\P(\AA_2)\to 1.
\]
Let $w_k$ be the $k$-th vertex revealed in the process (with $u_j=w_j$ for $1\leq j \leq m$).  Let 
\[
M_{i,k}=|\{k'\leq k: \sigma_{w_{k'}}=i\}|
\]
be the number of state $i$ vertices in $w_1\ldots, w_k$. 

We begin revealing the neighbours of the neighbours of the $u_j$ and then their descendants an so on to depth $\ell$ according to depth first search.
If $\sigma_{w_k}=i$, the number of vertices that $w_k$ connects to of state $i$ in $V\setminus\{w_1,\ldots,w_k\}$ is $\textsf{Bin}(N_i-M_{i,k},\frac{a}{n})$  and   for $i\neq i'$ it connects to $\textsf{Bin}(N_{i'}-M_{i',k},\frac{b}{n})$. On the other hand on the tree, if $\sigma^{(j)}_w=i$ then has $\textsf{Poi}(d(\frac{1+(q-1)\lambda}{q}))$ children of state $i$ and $\textsf{Poi}(d(\frac{1-\lambda}{q}))$ children of state $i'$ for each $i'\neq i$.  We will take the optimal coupling between these Binomial and Poisson random variables, and let $\AA_3$ be the event that the number of children for every vertex in the SBM and the corresponding vertex in the \rev{Galton Watson} tree had the same number of descendants of each spin type for the first $\ell$ levels.  If this holds then~\eqref{eq:coupling} holds.

Recall that
\[
\frac{a}{q}=d\frac{1+(q-1)\lambda}{q},\qquad \frac{b}{q}=d\frac{1-\lambda}{q}.
\]
On the event $\AA_1\cap\AA_2$ we have that  $|N_i-M_{i,k}|\leq 2n^{3/5}$ for all $i$ and $k\leq n^{1/4}$ so by equation~\eqref{eq:poi.bin.dtv}, when $|N_i-M_{i,k}-n/q|\leq 2n^{3/5}$,
\[
d_{TV}(\hbox{Bin}(N_i-M_{i,k},\frac{a}{n}),\textsf{Poi}(d(\frac{1+(q-1)\lambda}{q}))) \leq O(n\cdot n^{-2} + n^{3/5-1})=O(n^{-2/5})
\]
and
\[
d_{TV}(\hbox{Bin}(N_i-M_{i,k},\frac{b}{n}),\textsf{Poi}(d(\frac{1-\lambda}{q}))) =O(n^{-2/5}).
\]
On the event $\AA_1\cap\AA_2$, since there are at most $n^{1/4}$ vertices in the trees, we have
\[
\P(\AA_3^c\given \AA_1\cap\AA_2)\leq O(mn^{-2/5})=o(1).
\]
On the event $\AA=\AA_1\cap\AA_2\cap\AA_3$ the required coupling holds and $\P(\AA)=1-o(1)$ which completes the proof.

\end{proof}

Let $\AA$ be the event from the statement of lemma~\ref{lem:coup_trees_g}.    
For each  $(t,s,s(\ell))\in \RR_{\ell}$ let
\[
\NN\big(t,s,s(\ell)\big)=\Big\{j\leq m:T^{(j)}_\ell=t,\sigma^{(j)}_{\rho^{(j)}}=s,\sigma^{(j)}(\ell)=s(\ell)\Big\},
\]
and $\NN\big(t,s(\ell)\big)=\bigcup_{s\in[q]} \NN\big(t,s,s(\ell)\big)$. Define the event
\[
\DD=\Bigg\{\sup_{(t,s,s(\ell))\in \RR_{\ell,r}} \bigg|\frac{\big|\NN\big(t,s,s(\ell)\big)\big|}{m} - \P\big(T_\ell=t,~\sigma_\rho=s,~\sigma(\ell) = s(\ell)\big)\bigg| \leq m^{-1/3}\Bigg\}.
\]
Since $|\NN(t,s,s(\ell))|$ has distribution $\textsf{Bin}\Big(m,\P\big(T_\ell=t,\sigma_\rho=s,\sigma(\ell) = s(\ell)\big)\Big)$ we have that
\[
\P(\DD) \to 1.
\]
Similarly to $\NN$, on the SBM we define
\[
\MM(t,s,s(\ell))=\Big\{j\leq m:B_\ell(u_j)=t,~\sigma_{u_j}=s,~\sigma_{L_\ell(u_j)}=s(\ell)\Big\},
\]
and $\MM\big(t,s(\ell)\big)=\bigcup_{s\in[q]} \MM\big(t,s,s(\ell)\big)$, where $L_\ell(u_j)$ is the set of vertices at distance $\ell$ from $u_j$.  On the event $\AA$ we have that $\MM(t,s,s(\ell))=\NN(t,s,s(\ell))$ because of the coupling.  We denote the $\sigma$ algebra generated by $G$, the sizes of $\MM$'s and the boundary spins by $\FF:=\sigma\left(G,\left\{\left|\MM_{t,s,s(\ell)}\right|\right\}_{t,s,s(\ell)},\left(\sigma_{L_\ell(u_j)}\right)_{j\leq m}\right)$.  Let $\breve{\sigma}$ be any estimator that is allowed to use all the information in $\FF$.  Since this includes $G$, \rev{the set of $\breve{\sigma}$'s} is a bigger class of estimators \rev{than the set of $G$-measurable $\hat{\sigma}$'s} so we have that
\[
\E \max_{\breve{\sigma}}S_m(\breve{\sigma}) \geq \E \max_{\hat{\sigma}}S_m(\hat{\sigma})
\]

\begin{claim}\label{claim:exchangeable}
For each $(t,s(\ell))$ the configurations
\[
\Big\{\sigma_{B_{\ell}(u_j)} \Big\}_{j\in \MM(t,s(\ell))}
\]
are conditionally exchangeable given $\FF$.
\end{claim}
\begin{proof}
Let $\pi$ be a permutation on $\MM(t,s(\ell))$ and construct $\sigma^\pi$ by setting $\sigma^\pi_{B_{\ell-1}(u_j)}=\sigma_{B_{\ell-1}(u_{\pi(j)})}$ with all spins outside of $\bigcup_{j\in \MM(t,s(\ell))} B_{\ell-1}(u_j)$ remaining the same.  To establish the claim it suffices to observe that under the SBM measure $\P((G,\sigma)\mid\FF)=\P((G,\sigma^\pi)\mid \FF)$.  This holds because the empirical distribution of spins on the vertices and edges remains unchanged by the permutation.
\end{proof}

\rev{\begin{claim}\label{claim:near:optimal}
Let $\breve{\sigma}$ be any estimator measurable with respect to $\FF$.  Then for some $C>0$
\[
\E\bigg[\max_{\Gamma\in S_q} \sum_{j\in \MM(t,s(\ell))} \one(\Gamma(\breve{\sigma}_{u_j})=\sigma_{u_j})\bbgiven \FF\bigg] \leq \max_i |\MM(t,i,s(\ell))| + Cq!\sqrt{|\MM(t,s(\ell))|}
\]
\end{claim}
\begin{proof}
By the conditional exchangeability from the previous claim, given $\FF$, we have no information about which $u_j$ in $\MM(t,s(\ell))$ are which color except according to the frequencies given by the $\MM(t,i,s(\ell))$.  Hence
\[
\E\bigg[\sum_{j\in \MM(t,s(\ell))} \one(\breve{\sigma}_{u_j}=\sigma_{u_j})\bbgiven \FF\bigg] = \sum_{j\in \MM(t,s(\ell))} \frac{|\MM(t,\breve{\sigma}_{u_j},s(\ell))|}{|\MM(t,s(\ell))|} \leq \max_i |\MM(t,i,s(\ell))|.
\]
Moreover, a standard application of the Azuma-Hoeffding inequality gives that
\begin{align*}
&\P\bigg(\sum_{j\in \MM(t,s(\ell))} \one(\breve{\sigma}_{u_j}=\sigma_{u_j}) - \E\bigg[\sum_{j\in \MM(t,s(\ell))} \one(\breve{\sigma}_{u_j}=\sigma_{u_j})\bbgiven \FF\bigg] \geq x \bbgiven \FF\bigg)\\
&\qquad\qquad\leq \exp\big(-x^2/(8|\MM(t,s(\ell))|)\big),
\end{align*}
and hence
\[
\E\bigg[\Big(\sum_{j\in \MM(t,s(\ell))} \one(\breve{\sigma}_{u_j}=\sigma_{u_j}) - \max_i |\MM(t,i,s(\ell))|\Big)^+\bbgiven \FF\bigg] \leq C \sqrt{|\MM(t,s(\ell))|}.
\]
Since $\max\{x_i\} \leq \sum x_i^+$ it follows that
\begin{align*}
&\E\bigg[\max_{\Gamma\in S_q} \sum_{j\in \MM(t,s(\ell))} \one(\Gamma(\breve{\sigma}_{u_j})=\sigma_{u_j}) - \max_i |\MM(t,i,s(\ell))|\bbgiven \FF\bigg]\\ 
&\qquad \leq \E\bigg[\sum_{\Gamma\in S_q} \Big( \sum_{u\in \MM(t,s(\ell))} \one(\Gamma(\breve{\sigma}_{u_j})=\sigma_{u_j}) - \max_i |\MM(t,i,s(\ell))| \Big)^+\bbgiven \FF\bigg]\\
&\qquad \leq C q!\sqrt{|\MM(t,s(\ell))|},
\end{align*}
which completes the proof.
\end{proof}
}
\rev{With Claim~\ref{claim:exchangeable} and Claim \ref{claim:near:optimal} in hand, we now prove Theorem~\ref{t:nonrecon.to.sbm}.}
\begin{proof}[Proof of Theorem~\ref{t:nonrecon.to.sbm}]
By equation~\eqref{eq:tree.freq} on the event $\DD$, for $j\in \MM(t,s(\ell))$ and  $(t,i,s(\ell)) \in \RR_{\ell,r}$,
\begin{equation}\label{eq:graph:to:trees}
\frac{\left|\MM(t,i,s(\ell))\right|}{\left|\MM(t,s(\ell))\right|}= \frac{\P(T_\ell=t,\sigma_\rho=i,\sigma(\ell) = s(\ell))}{\P(T_\ell=t,\sigma(\ell) = s(\ell))} +o(1) = \P(\sigma_\rho = i \mid T_\ell=t, \sigma(\ell)=s(\ell))+o(1).
\end{equation}
If $t\in \TT_{\ell,r}$ but $(t,i,s(\ell)) \not\in \RR_{\ell,r}$ then $|\MM(t,i,s(\ell))|=0$.
Hence we have that
\begin{align*}
&\E [S_m(\breve{\sigma})\one(\AA\cap\DD)]\\
&\rev{\stackrel{(a)}{=}\E\Big[\E[S_m(\breve{\sigma})\given \FF]\one(\AA\cap \DD)\Big]}\\
&\rev{\stackrel{(b)}{\leq}} \E\bigg[ \bigg(1+\frac1{m}\sum_{(t,s(\ell)), t\in \TT_{\ell,r}} |\MM(t,s(\ell))| \Big(\max_i \frac{|\MM(t,i,s(\ell))|}{|\MM(t,s(\ell))|}-1\Big) + C q!\sqrt{|\MM(t,s(\ell))|}\bigg) \one(\AA\cap\DD)\bigg]\\
&\rev{\stackrel{(c)}{\leq}}  1 + \sum_{(t,s(\ell)), t\in \TT_{\ell,r}} \P\big(T_\ell=t, \sigma(\ell)=s(\ell)\big) \Big(\max_i\P\big(\sigma_\rho = i\,\big|\,T_\ell=t, \sigma(\ell)=s(\ell)\big) -1\Big)+o(1)\\
&\leq  \P(T_\ell \not \in \TT_{\ell,r}) +  \E\Big[\max_i\P\big(\sigma_\rho = i\,\big|\,T_\ell=t, \sigma(\ell)=s(\ell)\big)\Big]+o(1)\\
&\rev{\stackrel{(d)}{\leq}} \frac1{q} + \epsilon/2+o(1),
\end{align*}
where \rev{$(a)$ is by tower property, $(b)$ is by Claim~\ref{claim:near:optimal}, and $(c)$ is by the estimate~\eqref{eq:graph:to:trees} and the fact $\sqrt{|\MM(t,s(\ell))|}=o(m)$. Moreover, $(d)$ follows from equations~\eqref{eq:ell.defn} and~\eqref{eq:r.defn}}. Since $\P(\AA\cap\DD)=1-o(1)$ we have that
\[
\E [S_m(\breve{\sigma})]\leq \P\big(\AA\cap\DD)^c\big) + \E [S_m(\breve{\sigma})\one(\AA\cap\DD)]\leq \frac1{q} + \epsilon/2+o(1).
\]
This establishes that $\E[S(\hat{\sigma})]\leq \frac1{q}+o(1)$ which implies Theorem~\ref{t:nonrecon.to.sbm}.
\end{proof}
We now prove Corollary~\ref{cor:two.point}.
\begin{proof}[Proof of Corollary~\ref{cor:two.point}]
We will show that equations~\eqref{eq:SBM.defn} and \eqref{eq:SBM.2.point.defn} are equivalent.  First suppose that~\eqref{eq:SBM.2.point.defn} fails so we have two point correlations
\[
\limsup_n \E\sum_{i=1}^{q}\left|\P(\sig_u=i \given G,\sig_v=1) - \frac{1}{q}\right|>0
\]
and hence for some $\epsilon>0$
\[
\limsup_n \frac1{n}\E\sum_{u}\max_i \P(\sig_u=i \given G,\sig_v=1) > \frac{1}{q}+\epsilon.
\]
Let the estimator $\hat{\sigma}(G)\equiv (\hat{\sigma}_u)_{u\in V}$ be defined by
\[
\hat{\sigma}_u=\argmax_{1\leq i \leq q}~\P(\sig_u=i \given G,\sig_v=1)
\]
with ties broken randomly then, since conditioning on $\sig_v=1$ provides no extra information by symmetry,
\[
\limsup_n \E S(\hat{\sigma})= \limsup_n  \E\left[\E\bigg[\max_{\Gamma\in S_q}\frac1{n}\sum_u \one(\Gamma(\hat{\sigma}_u)=\sigma_u)\bbgiven G\bigg]\right]>\frac1{q}+\epsilon.
\]
Hence, \eqref{eq:SBM.defn} fails and we have detection.

Now suppose that we have detection for some estimator $\hat{\sigma}$.  Let $A_i=\{u:\hat{\sigma}_u=i\}$.  For at least one choice of $i$ and $j$ and some $\epsilon>0$ and infinitely many $n$ we must have
\[
\E[|\sum_{u\in A_j} \one(\sigma_u = i) - |A_j|/q |] \geq \epsilon n
\]
and hence for some $G$-measurable subset of the vertices $A$,
\[
\E [\sum_{u,u'\in A} \one(\sigma_u = \sigma_{u'}) - |A|^2/q ]= \E [\sum_i (\sum_{u\in A} \one(\sigma_u = i) - |A|/q )^2 ] \geq \epsilon^2 n^2.
\]
Then
\begin{align*}
\epsilon^2 &\leq \frac1{n^2}\E [\sum_i (\sum_{u\in A} \one(\sigma_u = i) - |A|/q )^2 ]\\
&= \frac1{n^2}\sum_{u,u' } \E\Big[ \one(u,u'\in A) \big( \one(\sigma_u = \sigma_{u'})-\frac1{q}\big)\Big] \\
&= \frac1{n^2}\sum_{u,u' } \E\Big[ \one(u,u'\in A) \E[\big( \one(\sigma_u = \sigma_{u'})-\frac1{q}\big)\mid G, \sigma_u]\Big] \\
&= \frac1{n^2} \sum_{u,u' } \E\Big[ \one(u,u'\in A) \big(\P(  \sigma_u = \sigma_{u'} | G,\sigma_u)-\frac1{q}\big)  \Big] \\
\end{align*}
Now by symmetry of the colors, $\P(  \sigma_u = \sigma_{u'} | G,\sigma_u)=\P( \sigma_{u'} = 1 | G,\sigma_u=1)$ and so
\[
\epsilon^2 \leq \E\Big[ \one(u,u'\in A) \big(\P(  \sigma_{u'}=1 | G,\sigma_u=1)-\frac1{q}\big)  \Big]
\]
which contradicts $\P( \sigma_{u'} = 1 | G,\sigma_u=1)$ converging to $\frac1{q}$ in probability  and hence ~\eqref{eq:SBM.2.point.defn} fails.
\end{proof}

\bibliographystyle{amsalpha}
\bibliography{all,my}

\newpage

\appendix
\section{Proof of Proposition \ref{prop:clt}}
\label{sec:appendix:clt}

In this section we use the Lindeberg method \cite{Lindeberg:22, Chatterjee06} to prove Proposition \ref{prop:clt}. Although Proposition \ref{prop:clt} holds directly through multi-dimensional analogue of the classical Lindberg-Feller central limit theorem for triangular array (see e.g. Proposition 2.27 in \cite{Vaart98}) even for a continuous and bounded function $\phi$, here we give a quantitative proof of Proposition \ref{prop:clt} by the Lindeberg method.

To this end, let $(Z_j)_{j\leq D}$ be independent Gaussian vectors whose first and second moments match $(V_j)_{j\leq D}$, i.e. $Z_j\sim\normal\big(\E V_j, \Cov(V_j)\big)$, and which are independent of $(V_j)_{j\leq D}$. Then, we first claim that for large enough $D>D_0(\eps,q,\phi,C)$,
\begin{equation}\label{eq:first:claim}
\left|\E\phi\bigg(\sum_{j=1}^{D} Z_j\bigg)-\E\phi(W)\right|\leq \frac{\eps}{2}.
\end{equation}
Indeed, since $\sum_{j}Z_j$ is just a Gaussian vector with mean $\mu_V:=\sum_j \E V_j$ and covariance $\Sigma_V:=\sum_j \Cov(V_j)$, we can couple it with $W\sim \normal(\mu,\Sigma)$ by $\sum_j Z_j\stackrel{d}{=}\Sigma_V^{1/2}Z+\mu_V$ and $W\stackrel{d}{=}\Sigma^{1/2}Z+\mu$, where $Z\sim \normal(0, I_q)$. Thus,
\begin{equation*}
\begin{split}
    \left|\E\phi\bigg(\sum_{j=1}^{D} Z_j\bigg)-\E\phi(W)\right|
    &\leq \sup_{x\in \R^q}\norm{\nabla \phi(x)}_2 \cdot \E\norm{\mu-\mu_V+(\Sigma^{1/2}-\Sigma_V^{1/2})Z}_2\\
    &\leq \sup_{x\in \R^q}\norm{\nabla \phi(x)}_2 \cdot\left(\norm{\mu-\mu_V}_2+ \left(\E\norm{(\Sigma^{1/2}-\Sigma_V^{1/2})Z}_2^2\right)^{1/2}\right)\\
    &=\sup_{x\in \R^q}\norm{\nabla \phi(x)}_2 \cdot\left(\norm{\mu-\mu_V}_2+ \norm{\Sigma^{1/2}-\Sigma_V^{1/2}}_F^2\right), 
\end{split}
\end{equation*}
where $\norm{A}_F:=\left(\tr(A^T A)\right)^{1/2}$ denotes the Frobenius norm of the matrix $A$. Note that by our assumption \eqref{eq:clt:condition}, we have $\norm{\mu-\mu_V}_2\leq \sqrt{q}\frac{C}{\sqrt{D}}$, and for large enough $D$,
\begin{equation*}
    \norm{\Sigma^{1/2}-\Sigma_V^{1/2}}_F\leq \sup_{\substack{\norm{A-B}_{\infty}\leq \frac{C}{\sqrt{D}}\\\max(\norm{A}_{\infty},\norm{B}_{\infty})\leq 2C,~A,B~\succeq 0}}~\norm{A^{1/2}-B^{1/2}}_F.
\end{equation*}
Since $A\to A^{1/2}$ is continuous for $A\succeq 0$, the right hand side above tends to $0$ as $D\to\infty$. Therefore, our claim \eqref{eq:first:claim} holds for $D>D_0(\eps,q,\phi,C)$.

Having \eqref{eq:first:claim} in hand, it suffices to establish by Lindeberg method that for $D>D_0(\eps,q,\phi,C)$,
\begin{equation}\label{eq:second:claim}
    \left|\E\phi\bigg(\sum_{j=1}^{D} V_j\bigg)-\E\phi\bigg(\sum_{j=1}^{D} Z_j\bigg)\right|\leq \frac{\eps}{2}.
\end{equation}
To this end, for $1\leq k \leq D+1$, denote
\begin{equation*}
    S_k=\sum_{j\leq k-1}V_j+\sum_{j\leq k}Z_j,\quad\quad T_k=\sum_{j\leq k-1}V_j+\sum_{j\leq k+1}Z_j.
\end{equation*}
Thus, the left hand side of \eqref{eq:second:claim} equals $\big|\E\phi(S_{D+1})-\E\phi(S_1)\big|$ which can be bounded by 
\begin{equation}\label{eq:sec:claim:2}
      \Big|\E\phi(S_{D+1})-\E\phi(S_1)\Big|\leq \sum_{k=1}^{D}\Big|\E\phi(S_{k+1})-\phi(S_k)\Big|=\sum_{k=1}^{D}\Big|\E\phi\big(T_k+V_k\big)-\phi\big(T_k+Z_k\big)\Big|.
\end{equation}
A crucial observation is that $(T_k,V_k,Z_k)$ is independent. Moreover, since $\phi$ has bounded third derivatives, Taylor approximation shows that
\begin{equation*}
\begin{split}
    &\Big|\E\phi\big(T_k+V_k\big)-\E\big\langle\nabla \phi(T_k), V_k\big\rangle-\frac{1}{2}\E\big\langle\nabla^2\phi(T_k)V_k,V_k\big\rangle\Big|
    \leq C_{q,\phi}\E \norm{V_k}_{\infty}^3,\\
    &\Big|\E\phi\big(T_k+Z_k\big)-\E\big\langle\nabla \phi(T_k), Z_k\big\rangle-\frac{1}{2}\E\big\langle\nabla^2\phi(T_k)Z_k,Z_k\big\rangle\Big|
    \leq C_{q,\phi}\E \norm{Z_k}_{\infty}^3,
\end{split}
\end{equation*}
where $C_{q,\phi}>0$ is a constant that only depends on $q$ and $\phi$. Note that since $Z_k$ and $V_k$ have the same mean and covariance, it follows from independence of $T_k$, $V_k$, and $Z_k$ that $\E\big\langle\nabla \phi(T_k), V_k\big\rangle=\E\big\langle\nabla \phi(T_k), Z_k\big\rangle$ and $\E\big\langle\nabla^2\phi(T_k)V_k,V_k\big\rangle=\E\big\langle\nabla^2\phi(T_k)Z_k,Z_k\big\rangle$. Moreover, since $\norm{V_k}_{\infty} \leq \frac{C}{\sqrt{D}}$ almost surely and $Z_k\sim \normal\big(\E V_k,\Cov(V_k)\big)$, we have that $\max(\E\norm{V_k}_{\infty}^3, \E\norm{Z_k}_{\infty}^3)=O_q\left(\frac{C^3}{D\sqrt{D}}\right)$. Hence, we have that
\begin{equation*}
    |\E\phi\big(T_k+V_k\big)-\phi\big(T_k+Z_k\big)\Big|\leq C^\prime_{q,\phi}\cdot\frac{C^3}{D\sqrt{D}}.
\end{equation*}
Therefore, plugging the above estimate into \eqref{eq:sec:claim:2} shows that
\begin{equation*}
    \Big|\E\phi(S_{D+1})-\E\phi(S_1)\Big|\leq \sum_{k=1}^{D}\Big|\leq C^\prime_{q,\phi}\cdot\frac{C^3}{\sqrt{D}},
\end{equation*}
which implies \eqref{eq:second:claim} for large enough $D>D_0(\eps,q,\phi,C)$.
\section{Proof of Lemma \ref{lem:g4}}
\label{sec:appendix}
It was shown in Lemma 4.7 of \cite{Sly:11} that $g_3(s)<s$ holds by a rigorous method of numerical integration. We use a similar strategy here. Unfortunately, for the critical case $q=4$, $g_4(s)$ is much closer to $s$ especially when $s$ is small. This is expected since Lemma 4.5 of \cite{Sly:11} showed that $g_4(s)=s-\frac{112}{27}s^3+O(s^4)$ holds while $g_3(s)=s-\frac{3}{4}s^2+O(s^3)$. Thus, previous techniques of \cite{Sly:11} do not directly apply, so we use more sophisticated bounds on Gaussian integrals and higher Taylor expansions.

Since $g_4(s)=s-\frac{112}{27}s^3+O(s^4)$, $g_4(s)<s$ holds for $s$ sufficiently small. However, we need that $g_4(s)<s$ for the full range where $0<s\leq \frac{3}{4}$. We proceed by dividing the range of $s$ by $s\leq 0.0612$ and $0.0612\leq s\leq 0.75$. Throughout, we let $W= (W_1,W_2,W_3,W_4) \sim \normal(0,\Sigma)$ for $\Sigma\equiv (\Sigma_{ij})_{i,j\leq 4}\equiv \Sigma^{\dagger}(4)$, defined in \eqref{eq:def:sigma:dagger}, and  denote $\mu\equiv(\mu_i)_{i\leq 4}\equiv \mu^\dagger(4)$, defined in \eqref{eq:def:mu:dagger}.

First, we show $g_4(s)<s$ for $0<s\leq 0.0612$. Note that $g_4(s)$ can be expressed by
\begin{equation*}
\begin{split}
    g_4(s)
    &=\E\psi(s\mu+\sqrt{s}W)-\frac{1}{4}\\
    &=\E\sum_{\ell=1}^{6}(-1)^{i-1}\frac{\Big(\sum_{i=1}^{4}\exp(s\mu_i+\sqrt{s}W_i)-4\Big)^{\ell-1}\exp(s\mu_1+\sqrt{s}W_1)}{4^{\ell}}\\
    &\quad\quad+\E\frac{\Big(\sum_{i=1}^{4}\exp(s\mu_i+\sqrt{s}W_i)-4\Big)^{6}}{4^6}\frac{\exp(s\mu_1+\sqrt{s}W_1)}{\sumi \exp(s\mu_i+\sqrt{s}W_i)}-\frac{1}{4},
\end{split}
\end{equation*}
where we used the identity $\frac{1}{1+x}=\sum_{\ell=1}^{6}(-1)^{\ell-1}x^{\ell-1}+\frac{x^6}{1+x}$ for $x=\frac{\sumi \exp(s\mu_i+\sqrt{s}W_i)-q}{q}$. We bound $\frac{\exp(s\mu_1+\sqrt{s}W_1)}{\sumi \exp(s\mu_i+\sqrt{s}W_i)}\leq 1$ and then use the fact that $\E e^{V}=e^{m+\sig^2/2}$ holds for $V\sim \normal(m,\sig^2)$. With the help of Mathematica, we have that
\begin{equation}\label{eq:g4:expand:zero}
\begin{split}
    g_4(s)
    &\leq \E\sum_{\ell=1}^{6}(-1)^{i-1}\frac{\Big(\sum_{i=1}^{4}\exp(s\mu_i+\sqrt{s}W_i)-4\Big)^{\ell-1}\exp(s\mu_1+\sqrt{s}W_1)}{4^{\ell}}\\
    &\quad\quad+\E\frac{\Big(\sum_{i=1}^{4}\exp(s\mu_i+\sqrt{s}W_i)-4\Big)^{6}}{4^6}-\frac{1}{4}\\
    &\leq \frac{3}{4}+\frac{135}{1024}e^{-12 s}-\frac{351}{256} e^{-28 s/3}+\frac{165}{128}e^{-8 s}+\frac{255}{1024}e^{-20 s/3}+\frac{297}{64}e^{-4 s}-\frac{165}{32} e^{-8 s/3}-\frac{561}{128} e^{-4 s/3}\\
    &\quad+\frac{261}{64} e^{4 s/3}+\frac{75}{128} e^{8 s/3}+\frac{225}{2048}e^{4 s}-\frac{105}{128}e^{8 s}+\frac{105}{1024}e^{28 s/3}-\frac{27 }{64}e^{12 s}-\frac{9}{64} e^{52 s/3}+\frac{45}{128} e^{56 s/3}\\
    &\quad+\frac{15}{1024}e^{20 s}+\frac{33
}{2048}e^{76 s/3}+\frac{15}{2048}e^{92 s/3}-\frac{9}{256} e^{100 s/3}+\frac{3}{2048}e^{52 s}.
\end{split}
\end{equation}
Observe that the largest coefficient (in absolute value) in front of $s$ in the exponential is $52$. By Taylor's theorem, for $|x|\leq 52\cdot 0.0612=3.1824$, we have that
\begin{equation*}
 \left|e^x-\sum_{i=0}^{12}\frac{x^i}{i!}\right|\leq\frac{e^{3.1824}}{13!}|x|^{13}\leq \frac{25}{13!}|x|^{13}.   
\end{equation*}
We can apply this to every exponential term in the right hand side of \eqref{eq:g4:expand:zero}. With the help of Mathematica, we have that
\begin{equation}\label{eq:g4:upper:1}
\begin{split}
    &g_4(s)-s\leq s^3 h(s),
\end{split}
\end{equation}
where 
\begin{equation*}
\begin{split}
    h(s)&=-\frac{112}{27}+\frac{64}{3}s-\frac{91648 }{1215}s^2+\frac{173440}{81}s^3+\frac{18592600576 }{229635}s^4+\frac{10146392576}{10935}s^5\\
    &\quad+\frac{19133607950336
}{2657205}s^6+\frac{6110816204800}{137781}s^{7}+\frac{2137414093270245376 }{9207215325}s^{8}\\
&\quad\quad+\frac{298924336312352768 }{279006525}s^{9}+\frac{2385675002892473434880}{18467043309}s^{10}.
\end{split}
\end{equation*}
Note that all the coefficients of $s^{\ell}$ in $h(s)$ is positive for $\ell \geq 3$. Thus, we have that for $s>0$,
\begin{equation*}
    h^\prime(s)\geq \frac{64}{3}-2\cdot\frac{91648}{1215}s+3\cdot\frac{173440}{81}s^2=\frac{173440}{27} \left(s-\frac{716}{60975} \right)^2+\frac{329498032}{74084625}>0,
\end{equation*}
whence $h(s)$ is increasing. Moreover, $h(0.0612)$ can be computed to arbitrarily high precision (e.g., in Mathematica). With the help of Mathematica, we have that $h(0.0612)<-0.02$, so $h(s)<0$ holds for any $0<s\leq 0.0612$. Therefore, we have that $g_4(s)<s$ for $0<s\leq 0.0612$ by \eqref{eq:g4:upper:1}.

Next, we show $g_4(s)<s$ for $0.0612\le s \leq 0.75$. Note that since $g_4$ is increasing by Lemma \ref{lem:g}, it suffices to show that 
\begin{equation}\label{eq:goal:g4}
\begin{split}
    g_4(s)<s-0.0003\quad\textnormal{for}\quad s\in &\mathcal{S}_1:=\{0.0615, 0.0618,...,0.0687,0.069\}\\
    g_4(s)<s-0.0005\quad\textnormal{for}\quad s\in &\mathcal{S}_2:=\{0.0695, 0.07,...,0.0795,0.08\}\\
    g_4(s)<s-0.0005\quad\textnormal{for}\quad s\in &\mathcal{S}_3:=\{0.0805, 0.081,...,0.0945,0.095\}\\
    g_4(s)<s-0.001\quad\textnormal{for}\quad s\in &\mathcal{S}_4:=\{0.096,0.097,....,0.199,0.2\}\\
    g_4(s)<s-0.01\quad\textnormal{for}\quad s\in &\mathcal{S}_4:=\{0.21,0.22,....,0.49,0.5\}\\
    g_4(s)<0.5\quad\textnormal{for}\quad s\in &\mathcal{S}_6:=\{0.75\}
\end{split}
\end{equation}
We show \eqref{eq:goal:g4} by a rigorous method of numerical integration. Note that we can express $g_4(s)$ by 
\begin{equation*}
\begin{split}
    g_4(s)
    &=\E\left[\frac{1}{1+\sum_{i=2}^{4}\exp\Big(s(\mu_i-\mu_1)+\sqrt{s}(W_i-W_1)\Big)}\right]-\frac{1}{4}\\
    &=\E_{G}\left[\frac{1}{1+e^{-\frac{16s}{3}}\left(e^{\sqrt{\frac{16s}{3}}(G_1+G_2)}+e^{\sqrt{\frac{16s}{3}}(G_1+G_3)}+e^{\sqrt{\frac{16s}{3}}(G_2+G_3)}\right)}\right]-\frac{1}{4},
\end{split}
\end{equation*}
where $\E_G$ denotes the expectation with respect to $3$ dimensional standard Gaussian vector $G\equiv (G_1,G_2,G_3)\sim \normal(0, I_3)$. Note that the function inside of $\E_G$ is bounded above by $1$, so we can truncate the range of $G_i$'s by the standard inequality
\begin{equation*}
    \P(|G_i|\geq x) \leq 2\frac{\exp(-x^2/2)}{x\sqrt{2\pi}}.
\end{equation*}
We use this inequality for $x=5$, so that
\begin{equation*}
    \P\Big(\exists~1\leq i\leq 3\textnormal{ s.t. }|G_i|\geq 5\Big)\leq 6\frac{\exp(-5^2/2)}{5\sqrt{2\pi}}\leq 2\cdot 10^{-6}.
\end{equation*}
Thus, with the inequality above and the fact that $G_i$'s have distribution symmetric around $0$, we can bound 
\begin{equation}\label{eq:g:bound:by:h}
\begin{split}
    g_4(s)
    &\leq -\frac{1}{4}+2\cdot 10^{-6}+\E\left[\sum_{\eps_1,\eps_2,\eps_3\in \{\pm 1\}}\frac{1}{1+e^{-\frac{16s}{3}}\sum_{1\leq i<j\leq 3}e^{\sqrt{\frac{16s}{3}}(\eps_i G_i+\eps_j G_j)}}\one\left(0\leq G_1,G_2,G_3\leq 1\right)\right]\\
    &\qquad\qquad\qquad\qquad\qquad\qquad+\E\left[\frac{1}{1+e^{-\frac{16s}{3}}\sum_{1\leq i<j\leq 3}e^{\sqrt{\frac{16s}{3}}(G_i+ G_j)}}\one\left(1\leq |G_1|,|G_2|,|G_3|\leq 5\right)\right]\\
    &=-\frac{1}{4}+2\cdot 10^{-6}+(2\pi)^{-\frac{3}{2}}\sum_{i=0}^{n-1}\,\sum_{j=0}^{n-1}\,\sum_{k=0}^{n-1}\,\int_{\frac{k}{n}}^{\frac{k+1}{n}}\int_{\frac{j}{n}}^{\frac{j+1}{n}}\int_{\frac{i}{n}}^{\frac{i+1}{n}}p(x_1,x_2,x_3,s)e^{-\frac{x_1^2+x_2^2+x_3^2}{2}}dx_1 dx_2 dx_3\\
    &\quad+(2\pi)^{-\frac{3}{2}}\sum_{\substack{-5n\leq i,j,k\leq 5n-1\\(i,j,k)\notin \{-n,...,n-1\}^3}}\,\int_{\frac{k}{n}}^{\frac{k+1}{n}}\int_{\frac{j}{n}}^{\frac{j+1}{n}}\int_{\frac{i}{n}}^{\frac{i+1}{n}}q(x_1,x_2,x_3,s)e^{-\frac{x_1^2+x_2^2+x_3^2}{2}}dx_1 dx_2 dx_3,
\end{split}
\end{equation}
where $n$ is a positive integer and the functions $p(x_1,x_2,x_3), q(x_1,x_2,x_3)$ are defined by 
\begin{equation*}
\begin{split}
    &p(x_1,x_2,x_3,s)= \sum_{\eps_1,\eps_2,\eps_3\in \{\pm 1\}} \frac{1}{1+e^{-\frac{16s}{3}}\sum_{1\leq i<j\leq 3}e^{\sqrt{\frac{16s}{3}}(\eps_i x_i+\eps_j x_j)}}\\
    &q(x_1,x_2,x_3,s)=\frac{1}{1+e^{-\frac{16s}{3}}\sum_{1\leq i<j\leq 3}e^{\sqrt{\frac{16s}{3}}(x_i+x_j)}}.
\end{split}
\end{equation*}
The reason we split the integrals into parts where $|G_i|\leq 1$ or $|G_i|\geq 1$ in \eqref{eq:g:bound:by:h} is explained by the following lemma.
\begin{lemma}\label{lem:h:monotone}
Fix $s\geq 0$.
For $0\leq x_1,x_2,x_3\leq 1$, $(x_1,x_2,x_3) \to p(x_1,x_2,x_3, s)$ is monotonically decreasing in each coordinate. Moreover, for any $x_1,x_2,x_3\in \R$, $(x_1,x_2,x_3)\to q(x_1,x_2,x_3,s)$ is monotonically decreasing.
\end{lemma}
\begin{proof}
The claim on $q(x_1,x_2,x_3,s)$ is trivial, so we focus on the case for $p(x_1,x_2,x_3)$. By symmetry, it suffices to show that $\frac{\partial p}{\partial x_1}\leq 0$. By a direct calculation, we have the following. Abbreviate $\alpha=\sqrt{\frac{16s}{3}}\geq 0$, then
\begin{equation}\label{eq:partial:p}
\begin{split}
    \frac{\partial p}{\partial x_1}
    &=-\alpha\sum_{\eps_1,\eps_2,\eps_3\in \{\pm 1\}}\frac{\eps_1 e^{-\alpha^2}\left(e^{\alpha(\eps_1 x_1+\eps_2 x_2)}+e^{\alpha(\eps_1 x_1+\eps_3 x_3)}\right)}{\left(1+e^{-\alpha^2}\sum_{1\leq i<j\leq 3}e^{\alpha(\eps_i x_i+\eps_j x_j)}\right)^2}\\
    &=-\alpha e^{-\alpha^2}\sum_{\eps_2,\eps_3\in \{\pm 1\}}\Big(e^{\alpha \eps_2 x_2}+e^{\alpha \eps_3 x_3}\Big)\left(\frac{e^{\alpha x_1}}{\left(\zeta+\xi e^{\alpha x_1}\right)^2}-\frac{e^{-\alpha x_1}}{\left(\zeta+\xi e^{-\alpha x_1}\right)^2}\right),
\end{split}
\end{equation}
where we abbreviated
\begin{equation*}
\begin{split}
    \zeta 
    &\equiv \zeta(\eps_2,\eps_3,x_2,x_3,\alpha )\equiv 1+e^{-\alpha^2+\alpha(\eps_2x_2+\eps_3x_3)},\\
    \xi&\equiv\xi(\eps_2,\eps_3,x_2,x_3,\alpha)\equiv e^{-\alpha^2}\left(e^{\alpha\eps_2x_2}+e^{\alpha\eps_3x_3}\right).
\end{split}
\end{equation*}
Since $\alpha \geq 0$ in \eqref{eq:partial:p}, it suffices to show that $\frac{e^{\alpha x_1}}{\left(\zeta+\xi e^{\alpha x_1}\right)^2}-\frac{e^{-\alpha x_1}}{\left(\zeta+\xi e^{-\alpha x_1}\right)^2}\geq 0$ for every $\eps_2,\eps_3\in \{\pm 1\}$ and $0\leq x_1,x_2,x_3\leq 1$. By a direct calculation,
\begin{equation*}
    \frac{e^{\alpha x_1}}{\left(\zeta+\xi e^{\alpha x_1}\right)^2}-\frac{e^{-\alpha x_1}}{\left(\zeta+\xi e^{-\alpha x_1}\right)^2}=\frac{e^{\alpha x_1}-e^{-\alpha x_1}}{\left(\zeta+\xi e^{\alpha x_1}\right)^2\left(\zeta+\xi e^{-\alpha x_1}\right)^2}(\zeta+\xi)(\zeta-\xi), 
\end{equation*}
so it suffices to show that $\zeta - \xi \geq 0$. Note that
\begin{equation*}
    e^{\alpha^2}(\zeta-\xi)=e^{\alpha^2}-1+(e^{\alpha \eps_2 x_2}-1)(e^{\alpha \eps_3 x_3}-1), 
\end{equation*}
thus when $\eps_2=\eps_3=+1$ or $\eps_2=\eps_3=-1$, $\zeta-\xi \geq e^{\alpha^2}-1\geq 0$. For the case where $\eps_2=+1$ and $\eps_3=-1$, we have
\begin{equation*}
e^{\alpha^2}(\zeta-\xi)=e^{\alpha^2}+e^{\alpha(x_2-x_3)}-e^{\alpha x_2}-e^{-\alpha x_3}.
\end{equation*}
It is straightforward to see that the right hand side is a decreasing function of $x_2$ and $x_3$, so we have that
\begin{equation*}
    e^{\alpha^2}(\zeta-\xi)\ge e^{\alpha^2}+1-e^{\alpha}-e^{-\alpha}\geq 0,
\end{equation*}
where the last inequality can be seen by Taylor expansion of $e^{x}$ for $x\in \R$. The case where $\eps_2=-1$ and $\eps_3=+1$ follows by symmetry, so for all cases we have that $\zeta-\xi\geq 0$. Therfore, $\frac{\partial p}{\partial x_1}\le 0$.
\end{proof}
To bound the integration of the Gaussina kernel in \eqref{eq:g:bound:by:h}, we use the following lemma.
\begin{lemma}\label{lem:gaussian:kernel:bound}
For a non-negative integer $i$, we have that
\begin{equation*}
    \int_{\frac{i}{n}}^{\frac{i+1}{n}}e^{-\frac{x^2}{2}}dx\leq \phi(i,n):=
    \begin{cases}
    \frac{1}{n} &\qquad ~~i=0\\
    \min\left(\frac{1}{n}\left(1-\frac{i}{2n^2}\right)e^{-\frac{1}{2}\left(\frac{i}{n}\right)^2}, \frac{n}{i}\left(e^{-\frac{1}{2}\left(\frac{i}{n}\right)^2}-e^{-\frac{1}{2}\left(\frac{i+1}{n}\right)^2}\right)\right) &\quad 1\leq i \leq n-1\\
    \min\left(\frac{1}{2n}\left(e^{-\frac{1}{2}\left(\frac{i}{n}\right)^2}+e^{-\frac{1}{2}\left(\frac{i+1}{n}\right)^2}\right), \frac{n}{i}\left(e^{-\frac{1}{2}\left(\frac{i}{n}\right)^2}-e^{-\frac{1}{2}\left(\frac{i+1}{n}\right)^2}\right)\right) &\qquad ~~i\geq n
    \end{cases}
\end{equation*}
\end{lemma}
\begin{proof}
The case $i=0$ follows from $e^{-\frac{x^2}{2}}\leq 1$. For $i\geq 1$, note that
\begin{equation*}
    \int_{\frac{i}{n}}^{\frac{i+1}{n}}e^{-\frac{x^2}{2}}dx\leq \frac{n}{i}\int_{\frac{i}{n}}^{\frac{i+1}{n}}xe^{-\frac{x^2}{2}}dx=\frac{n}{i}\left(e^{-\frac{1}{2}\left(\frac{i}{n}\right)^2}-e^{-\frac{1}{2}\left(\frac{i+1}{n}\right)^2}\right).
\end{equation*}
Moreover, note that for $f(x)=e^{-\frac{x^2}{2}}$, we have $f^\prime(x)=-x e^{-\frac{x^2}{2}}$ and $f^{\prime\prime}(x)=(x^2-1)e^{-\frac{x^2}{2}}$. Thus, $f(x)$ is concave for $0\leq x\leq 1$, so $f(x) \leq f(a)+(x-a)f^\prime(a)$ holds for $a\leq x \leq 1$. For $1\leq i \leq n-1$, letting $a=\frac{i}{n}$ and integrating both sides give
\begin{equation*}
    \int_{\frac{i}{n}}^{\frac{i+1}{n}}e^{-\frac{x^2}{2}}dx\leq \frac{1}{n}\left(1-\frac{i}{2n^2}\right)e^{-\frac{1}{2}\left(\frac{i}{n}\right)^2}.
\end{equation*}
On the other hand, $f(x)$ is convex for $x\geq 1$, so for $a\leq x\leq b$, where $a\geq 1$, we have $f(x)\leq \frac{b-x}{b-a}f(a)+\frac{x-a}{b-a}f(b)$. For $i\geq n$, letting $a=\frac{i}{n}$ and $b=\frac{i+1}{n}$, and integrating both sides give
\begin{equation*}
     \int_{\frac{i}{n}}^{\frac{i+1}{n}}e^{-\frac{x^2}{2}}dx\leq \frac{1}{2n}\left(e^{-\frac{1}{2}\left(\frac{i}{n}\right)^2}+e^{-\frac{1}{2}\left(\frac{i+1}{n}\right)^2}\right).
\end{equation*}
Therefore, combining the three inequalities above completes the proof.
\end{proof}
Note that since $x\to e^{-\frac{x^2}{2}}$ is even, we have by Lemma \ref{lem:gaussian:kernel:bound} that $\int_{\frac{i}{n}}^{\frac{i+1}{n}}e^{-\frac{x^2}{2}}\leq \phi(-i-1,n)$ for negative integer $i$. Thus, by \eqref{eq:g:bound:by:h}, Lemma \ref{lem:h:monotone} and Lemma \ref{lem:gaussian:kernel:bound}, we can bound 
\begin{equation}\label{eq:rigorous:numerical}
    g_4(s)\leq -\frac{1}{4}+2\cdot 10^{-6}+\sum_{i,j,k=0}^{n-1}S_n(i,j,k,s)+\sum_{\substack{-5n\leq i,j,k\leq 5n-1\\(i,j,k)\notin \{-n,...,n-1\}^3}}T_n(i,j,k,s)
\end{equation}
where $S_n(i,j,k,n,s)$ and $T_n(i,j,k,n,s)$ are defined by
\begin{equation*}
\begin{split}
    &S_n(i,j,k,n,s):=p\left(\frac{i}{n},\frac{j}{n},\frac{k}{n},s\right)\cdot \phi(i,n)\cdot \phi(j,n)\cdot \phi(k,n),\\
    &T_n(i,j,k,n,s):=q\left(\frac{i}{n},\frac{j}{n},\frac{k}{n},s\right)\cdot \phi\Big(\min(|i|,|i+1|),n\Big
    )\cdot  \phi\Big(\min(|j|,|j+1|),n\Big
    )\cdot  \phi\Big(\min(|k|,|k+1|),n\Big
    ).
\end{split}
\end{equation*}
The right hand side of \eqref{eq:rigorous:numerical} is a combination of basic arithmetic operations and exponentials, so it can be rigorously computed to aribtrarily high preicion (e.g. in R or Mathematica). Recalling our goal \eqref{eq:goal:g4}, using $n=200$ for $s\in \mathcal{S}_1\cup \mathcal{S}_2$ and using $n=100$ for $s\in \mathcal{S}_3\cup \mathcal{S}_4\cup\mathcal{S}_5\cup \mathcal{S}_6$ establishes \eqref{eq:goal:g4}.
\end{document}